\documentclass[11pt]{amsart}
\usepackage[T1]{fontenc}
\usepackage{amssymb,amsmath, amsthm, amsfonts}
\usepackage[dvipsnames]{xcolor}
\usepackage{graphicx}
\usepackage{listings}
\usepackage[normalem]{ulem}
\usepackage[margin=1in]{geometry}
\usepackage{lstautogobble}
\usepackage{enumerate}
\usepackage[shortlabels]{enumitem}
\usepackage{thmtools}
\usepackage{thm-restate}
\usepackage{verbatim}
\usepackage{accents}
\usepackage{mathtools}
\usepackage{physics}
\usepackage[colorlinks=true, allcolors=magenta,linkcolor=blue, citecolor=magenta]{hyperref}
\usepackage[numbered]{bookmark}
\usepackage[capitalise]{cleveref}
\usepackage[bottom]{footmisc}
\crefformat{equation}{(#2#1#3)}
\usepackage{float}
\restylefloat{table}
\usepackage{mathrsfs}
\setlist{listparindent=\parindent,parsep=0pt}

\theoremstyle{plain}
\newtheorem{thm}{Theorem}[section]
\newtheorem{prop}[thm]{Proposition}
\newtheorem{lemma}[thm]{Lemma}
\newtheorem{cor}[thm]{Corollary}

\theoremstyle{definition}
\newtheorem{mydef}[thm]{Definition}

\newtheorem{remark}[thm]{Remark}

\Crefname{thm}{Theorem}{Theorems}
\Crefname{prop}{Proposition}{Propositions}
\Crefname{cor}{Corollary}{Corollaries}
\crefname{cor}{Corollary}{Corollaries}

\numberwithin{equation}{section}


\DeclareMathOperator{\supp}{supp}

\DeclareMathOperator{\diam}{diam}

\DeclareMathOperator{\dist}{dist}

\newcommand{\mmd}{\mathsf{MMD}}
\newcommand{\dissip}{\mathsf{D}}

\newcommand{\p}{{\partial}}

\renewcommand{\d}{\mathsf{d}}
\newcommand{\R}{{\mathbb{R}}}

\newcommand{\N}{{\mathbb{N}}}

\newcommand{\T}{{\mathbb{T}}}
\newcommand{\g}{{\mathsf{g}}}

\newcommand{\nab}{\nabla}

\newcommand{\D}{\Delta}

\newcommand{\ul}{\underline}

\newcommand{\ep}{\epsilon}

\newcommand{\be}{\beta}

\newcommand{\la}{\lambda}

\newcommand{\indic}{\mathbf{1}}

\newcommand{\Ec}{\mathcal{E}}

\renewcommand{\P}{\mathcal{P}}

\let\div\relax
\DeclareMathOperator{\div}{\mathrm{div}}

\def\XXint#1#2#3{{\setbox0=\hbox{$#1{#2#3}{\int}$ }
    \vcenter{\hbox{$#2#3$ }}\kern-.6\wd0}}

\let\oldtocsection=\tocsection
\let\oldtocsubsection=\tocsubsection
\let\oldtocsubsubsection=\tocsubsubsection
\renewcommand{\tocsection}[2]{\hspace{0em}\oldtocsection{#1}{#2}}
\renewcommand{\tocsubsection}[2]{\hspace{1em}\oldtocsubsection{#1}{#2}}
\renewcommand{\tocsubsubsection}[2]{\hspace{2em}\oldtocsubsubsection{#1}{#2}}

\title{Wasserstein gradient flows for Coulomb discrepancies}

\author[Chodron de Courcel]{Antonin Chodron de Courcel}
\address{CNRS, École Normale Supérieure, PSL Université, 45 rue d'Ulm, 75005 Paris, France}
\email{decourcel@ihes.fr}

\author[Rosenzweig]{Matthew Rosenzweig}
\address{Department of Mathematical Sciences, Carnegie Mellon University, Pittsburgh, PA 15213, USA}
\email{mrosenz2@andrew.cmu.edu}
\thanks{A.C. is supported by the European Research Council (ERC project Wolf). M.R. was supported by NSF grants DMS-2441170, DMS-2345533, and DMS-2342349.}
\hypersetup{
  pdftitle={Wasserstein gradient flows for Coulomb discrepancies},
  pdfauthor={Antonin Chodron de Courcel and Matthew Rosenzweig}
}

\begin{document}

\begin{abstract}
    We study the long-time behavior of the Wasserstein gradient flow of the squared Maximum Mean Discrepancy (MMD) between a probability measure $\rho$ and a target measure $\mu$, where the underlying kernel is given by a Coulomb potential.

    First, for target densities $\mu\in\P(\Omega)\cap L^\infty(\Omega)$, we establish the existence of global weak solutions starting from arbitrary Borel probability measures and prove an ultracontractive estimate, showing that the density $\rho_t$ becomes instantly bounded in $L^\infty$ for $t>0$. We also investigate the regularity of these solutions, showing that the H\"older norm can grow exponentially in time.

    Second, on the flat torus $\T^\d$, we prove a global-in-the-source metric PL inequality for every source of finite Coulomb energy and for nearly uniform bounded targets. For general bounded, uniformly positive targets, we prove exponential decay of the squared MMD without requiring a lower bound on the initial data, using a defective PL inequality whose defect accounts for possible vacuum regions in the evolving density. We also prove that the usual PL inequality may fail when the target vanishes only at one point and, in dimensions at least two, that, for each prescribed lower bound $\lambda\in(0,1)$, no PL constant depending only on $(\d,\lambda)$ can hold uniformly over all targets satisfying $\mu\ge\lambda$.

    On $\R^\d$, for $\d\ge2$, under radial symmetry, source-support inclusion, and target-positivity assumptions, we establish a PL inequality and exponential convergence. On the unrestricted whole-space class, by contrast, we identify an obstruction at spatial infinity: for a compactly supported bounded target, uniformly localized sources initially separated from the target by distance $D$ retain a fixed fraction of their initial squared MMD for times of order $D$. Consequently, neither a multiplicative squared-MMD decay modulus uniform over the initial datum nor a global PL inequality can hold on this class.

    Finally, in every dimension and on both $\R^\d$ and $\T^\d$, we prove that every Lagrangian critical point for which the positive part $(\rho-\mu)^+$ is absolutely continuous coincides with the target. In dimension two, finite Coulomb energy supplies uniform tightness through a logarithmic-capacity estimate. As a consequence, every solution constructed here whose Coulomb energy is finite at some positive time converges to the target narrowly, weak-$*$ in $L^\infty$, and strongly in the inhomogeneous space $H^{-\alpha}(\R^2)$ for every $\alpha>0$.
\end{abstract}

\maketitle

\section{Introduction}

\subsection{Problem setting and variational structure}

We consider the Cauchy problem for the following nonlocal transport equation:
\begin{equation}\label{eq:PDE}
    \begin{cases}
        \p_t \rho - \div ( \rho \nab \g\ast (\rho - \mu)) = 0 \\
        \rho\vert_{t=0} = \rho_0
    \end{cases}
    \qquad (t,x) \in [0,\infty) \times \Omega,
\end{equation}
where $\Omega$ is either the Euclidean space $\R^\d$ or the flat torus $\T^\d$, {which we identify with the hypercube $(-\frac12,\frac12)^\d$ with periodic boundary conditions}, and $\mu, \rho_0 \in \P(\Omega)$ are given probability measures. The kernel $\g$ denotes the Coulomb potential, which satisfies
\begin{equation}
    -\D\g = \begin{cases}
        \delta_0 - 1 & \text{on } \T^\d, \\
        \delta_0 & \text{on } \R^\d.
    \end{cases}
\end{equation}
The measure $\mu$ acts as a fixed target configuration, while $\rho_t$ represents the evolving probability measure. Under the additional hypothesis $\mu\in L^\infty(\Omega)$, the global weak solution constructed in \cref{thm:Cauchy} has a bounded density for every $t>0$.

Equation \eqref{eq:PDE} admits a natural variational interpretation as the Wasserstein gradient flow of the functional
\begin{equation}
    \mmd^2(\rho,\mu) := \frac12 \int_{\Omega} \g\ast (\rho- \mu )\, d(\rho-\mu),
\end{equation}
with respect to the Wasserstein distance $W_2$. For background on the general theory of Wasserstein gradient flows, we refer to \cite{AGS08,San15,FG23}. We use $\mmd^2$ for the energy-normalized squared MMD; with the present Coulomb normalization,
\begin{equation}
2\mmd^2(\rho,\mu)=\|\rho-\mu\|_{\dot H^{-1}}^2
\end{equation}
whenever the Coulomb energy is finite.

The functional $\mmd^2(\rho,\mu)$ is often referred to as the squared Maximum Mean Discrepancy (MMD) in statistics \cite{Gretton12,AKSG19}. In the context of generative modeling, equations of the form \eqref{eq:PDE} have recently received renewed attention as learning dynamics for transport methods that aim to match an easy-to-sample source distribution to a complex target $\mu$ \cite{deng2026generativemodelingdrifting,han2026onestepgenerativemodelingwasserstein,gretton2026wassersteingradientflowinterpretation}. In a complementary direction, MMD has been used as the matching loss for diffeomorphic generative models parameterized by kernel-based ODE transport maps \cite{Pandey24}. Closely related measure-valued Wasserstein gradient flows arise as mean-field descriptions of the training dynamics of shallow neural networks \cite{ChizatBach18,Mei18}; equation \eqref{eq:PDE} may be viewed as their Coulomb-kernel analogue. From the perspective of statistical physics, $\mmd^2(\rho,\mu)$ corresponds to the relative (or modulated) energy of the configuration $\rho$ with respect to $\mu$, and \eqref{eq:PDE} describes the mean-field evolution of a system of particles interacting via Coulomb forces in the background field generated by a fixed distribution $\mu$ of particles of the opposite sign. For more on this topic, we refer to the book \cite{Ser24} and the recent surveys \cite{chaintronPropagationChaosReview2022,golseMeanFieldLimitsStatistical2022} on mean-field limits.

The gradient flow structure of \eqref{eq:PDE} implies the formal dissipation identity
\begin{equation}
    \frac{d}{dt} \mmd^2(\rho_t,\mu) = -\dissip(\rho_t\mid\mu), \quad \text{where} \quad \dissip(\rho\mid\mu) := \int_\Omega |\nabla \g\ast (\rho-\mu)|^2\, d\rho.
\end{equation}
The convergence of $\rho_t$ toward the equilibrium $\mu$ is typically mediated by a functional coercivity estimate. Indeed, if $\rho$ is uniformly bounded from below by $\lambda > 0$, an integration by parts yields
\begin{align}
    \dissip(\rho\mid\mu) &\ge \lambda \int_\Omega |\nabla \g\ast (\rho-\mu)|^2\, dx = -\lambda \int_\Omega \g\ast (\rho-\mu) \, \Delta\g\ast(\rho-\mu) \, dx \notag\\
    &= 2\lambda \mmd^2(\rho,\mu),
\end{align}
where we used $-\Delta\g \ast (\rho-\mu) = \rho-\mu$. Combined with the dissipation identity, this leads via Gr\"onwall's lemma to the exponential decay
\begin{equation}
\mmd^2(\rho_t,\mu) \le e^{-2\lambda t}\mmd^2(\rho_0,\mu).
\end{equation}
However, the global validity of a coercivity estimate of the form
\begin{equation}\label{eq:PL-intro}
    \dissip(\rho\mid\mu) \ge c_{\mathrm{PL}} \mmd^2(\rho,\mu),
\end{equation}
which is often referred to as a Polyak--\L{}ojasiewicz (PL) inequality in the optimization literature, is a subtle question that forms a core part of our investigation. The preceding lower-bound argument gives $c_{\mathrm{PL}}=2\lambda$. For a fixed target $\mu$, we call \eqref{eq:PL-intro} global if it holds for every admissible source $\rho$, without a smallness, proximity, or other local condition relating $\rho$ to $\mu$.

Two themes guide our analysis: the Cauchy theory and regularity of \eqref{eq:PDE}, and the coercive relation between its energy and dissipation. The latter governs quantitative convergence, but depends sensitively on the geometry of the domain and on the admissible source and target measures. We now discuss the status of these themes in the existing literature.

\subsection{Literature review}

We begin with the zero-target case $\mu=0$, which represents pure repulsion. In this setting, the existence and uniqueness theory, instantaneous $L^\infty$ regularization, and asymptotic behavior are now well understood \cite{linHydrodynamicLimitGinzburgLandau2000,BLL12,SV14}.

Regarding the Coulomb discrepancy flow with a nontrivial target, Boufad\`ene and Vialard \cite{BV25} proved global existence and uniqueness for H\"older-continuous initial and target densities, assuming compact support in the Euclidean case. On a closed Riemannian manifold, they also obtained exponential convergence when both the initial density $\rho_0$ and the target density $\mu$ are bounded away from zero. Their maximum principle propagates the bounds
\begin{equation}
    \min\!\left\{\operatorname*{ess\,inf}\rho_0,\operatorname*{ess\,inf}\mu\right\}
    \le \rho_t \le
    \max\!\left\{\|\rho_0\|_{L^\infty},\|\mu\|_{L^\infty}\right\},
\end{equation}
so the lower bound required for \eqref{eq:PL-intro} remains uniform in time. Their result therefore does not address whether the PL inequality can hold without a lower bound on the evolving density.

Boufad\`ene and Vialard also distinguish Lagrangian critical points from their stronger, blocked-JKO notion of Wasserstein criticality. They show that a Lagrangian critical measure agrees with the target on the interior of its support. In the Euclidean finite-Coulomb-energy setting, they also exclude Wasserstein critical points for which $(\rho-\mu)^+$ satisfies an explicit scale-uniform lower mass-growth condition, and they record a codimension-one example \cite{BV25}. 
These results leave open, among other cases, absolutely continuous densities supported on positive-measure sets with empty interior and mixed measures for which $(\rho-\mu)^+$ is absolutely continuous. Our rigidity result below settles every Lagrangian critical point for which $(\rho-\mu)^+$ is absolutely continuous; see \cref{sec:criticality-planar} for the precise comparison.

In the one-dimensional whole-space setting, the Coulomb kernel is
$\g(x)=-\frac12|x|$ and, after rescaling, 
\eqref{eq:PDE} is the Wasserstein gradient flow of the
energy-distance-kernel MMD studied by Duong, Stein, Beinert,
Hertrich, and Steidl
\cite{DuongSteinBeinertHertrichSteidl26}. Using the isometric
identification of $\P_2(\R)$ with the cone of quantile functions in
$L^2(0,1)$, they characterized this gradient flow for arbitrary initial
and target measures in $\P_2(\R)$. For discrete targets, they derived
explicit solution formulas, and they also establish invariance and
smoothing properties of the flow. The scopes of their work and our own are complementary: their
argument is specific to the real line and exploits the exactly solvable structure of the one-dimensional case, whereas our Eulerian construction applies in every dimension
on both $\R^\d$ and $\T^\d$, starts from an arbitrary probability
initial measure, and yields instantaneous $L^\infty$ regularization
when the target has a bounded density.

During the preparation of our manuscript, we learned of the independent work of Chizat, Colombo, Colombo, and Fern\'andez-Real \cite{CCCFR26}, which studies Wasserstein gradient flows of Sobolev and Riesz kernel discrepancies on the torus. At the Coulomb endpoint, they prove global weak-$*$ convergence to an arbitrary bounded target density for bounded initial densities, without assuming a positive lower bound on either density. When the target is uniformly positive, they obtain quantitative exponential convergence of the energy even when the initial density has vacuum regions, by means of a maximum-principle and gap-filling argument; under additional lower-bound and regularity assumptions, they also obtain convergence in stronger topologies. For higher-order Sobolev discrepancies, they establish local algebraic rates under additional regularity and smallness assumptions. Thus, qualitative convergence to a bounded target which may vanish is known in the bounded-density class, while quantitative rates and stronger-topology convergence for degenerate targets remain poorly understood. Their well-posedness and convergence theory is formulated in a bounded-density weak class. By contrast, the construction below starts from an arbitrary Borel probability measure, uses the instantaneous $L^\infty$ regularization to enter a bounded-density regime at positive times, and derives quantitative convergence for uniformly positive targets directly from a defective PL inequality. The present paper also addresses global coercivity obstructions at spatial infinity and the radial Euclidean problem with source-support inclusion.

Wasserstein steepest-descent flows for discrepancy functionals with Riesz kernels have also been studied from variational, approximation, and numerical perspectives by Hertrich, Gr\"af, Beinert, and Steidl \cite{HGBS24} and by Altekr\"uger, Hertrich, and Steidl \cite{AHS23}. These works are complementary to the coercivity and long-time questions considered here.

The present paper is also closely related to three companion manuscripts in preparation, concerning energy-kernel MMD dynamics; unconfined, target-free sub-Coulomb Riesz flows; and static Riesz-MMD quantization. Because these projects overlap with parts of the present analysis, we explain the division of scope after stating our main results.

\subsection{Our contributions}

This work makes four main contributions. They concern the well-posedness and regularity of the flow \eqref{eq:PDE}, PL coercivity and long-time behavior, obstructions on the unrestricted whole-space class, and, finally, rigidity of Lagrangian critical points together with its application to planar convergence. We summarize them as follows.

\begin{enumerate}
    \item \textbf{The Cauchy problem.} Assuming $\mu\in L^\infty(\Omega)$, we establish the existence of global weak solutions starting from arbitrary Borel probability measures $\rho_0\in\P(\Omega)$. We prove that the flow is instantly regularizing in $L^\infty$: for $t>0$, the density $\rho_t$ is bounded by an estimate independent of the initial data. This ultracontractivity is a direct consequence of the Coulomb repulsion and was already known in the zero-target case $\mu\equiv0$. Concerning the regularity of solutions to \eqref{eq:PDE}, we show that the H\"older seminorm $[\rho_t]_{C^{0,\beta}}$ can increase at an exponential rate, even starting from smooth initial data and targets.

    \item \textbf{PL coercivity and long-time behavior on $\T^\d$.} We prove a global-in-the-source metric PL inequality for every source of finite Coulomb energy, for the uniform target and, in dimensions $\d\ge2$, for bounded targets satisfying $\operatorname*{ess\,sup}\mu<\frac{\d}{\d-1}\operatorname*{ess\,inf}\mu$; in dimension one, the inequality holds for every bounded uniformly positive target. We recover the integral-dissipation form whenever the source potential has a $C^1$ representative, in particular for bounded sources. Thus the result applies to targets sufficiently close to the uniform density in $L^\infty$, without any lower bound on the source. For arbitrary uniformly positive targets, we derive exponential convergence rates for $\mmd^2(\rho_t,\mu)$ from a defective PL inequality which remains valid in the presence of vacuum in the evolving density. Conversely, we show that a global PL inequality may fail when the target vanishes at only one point and, in dimensions at least two, that, for each prescribed lower bound $\lambda\in(0,1)$, no PL constant depending only on $(\d,\lambda)$ can hold uniformly over all targets satisfying $\mu\ge\lambda$.

    \item \textbf{Radial coercivity and whole-space obstructions on $\R^\d$.} For $\d\ge2$, if the source and target are radial, the target has connected support and is bounded from below on that support, and the source is supported in the same set, then we prove a PL inequality and exponential convergence. On the unrestricted whole-space class, if $\mu$ is compactly supported and bounded, then we prove that a uniformly localized source initially separated from $\supp\mu$ by distance $D$ retains a fixed fraction of its initial squared MMD for times proportional to $D$. This bounded-speed persistence rules out both a multiplicative squared-MMD decay modulus uniform over the initial datum and a global PL inequality. We also give an independent static construction of smooth densities for which the dissipation-to-energy ratio tends to zero.

    \item \textbf{Lagrangian criticality and planar convergence.} On $\R^\d$ and $\T^\d$, in every dimension, we prove that a Lagrangian critical point of the Coulomb discrepancy must equal the target whenever $(\rho-\mu)^+$ is absolutely continuous. In dimension two, we show that finite Coulomb energy supplies uniform-in-time tightness through a logarithmic-capacity estimate. We combine this estimate with the dissipation inequality, local compactness of the Coulomb fields, and rigidity of Lagrangian critical points to prove convergence of every solution constructed here whose Coulomb energy is finite at some positive time to the target narrowly, weak-$*$ in $L^\infty$, and strongly in $H^{-\alpha}(\R^2)$ for every $\alpha>0$. These conclusions do not include convergence in the Coulomb MMD, which remains an open question.
\end{enumerate}

We consider the following notion of solution to \eqref{eq:PDE}.

\begin{mydef}\label{def:weak-solution}
	Let $\mu \in \mathcal{P}(\Omega)\cap L^\infty(\Omega)$ and $\rho_0 \in \mathcal{P}(\Omega)$. We say that a curve $\rho:[0,\infty)\to\mathcal{P}(\Omega)$ is a \emph{global weak solution} of
\begin{equation}
\partial_t\rho_t+\div(\rho_t v_t)=0,
\qquad
v_t =-\nabla\g\ast(\rho_t-\mu),
\end{equation}
if, for every $T>0$, the following conditions hold:
\begin{align}
&\rho\in L^q([0,T];L^p(\Omega)) \text{ for some } p,q\in(1,\infty) \text{ such that } 1/p + 1/q >1,\\
&\rho_t\in L^\infty(\Omega) \text{ for a.e. } t\in(0,T), \\
&\int_0^T\int_{\Omega}\bigl(\partial_t\varphi(t,x)+v_t(x)\cdot\nabla_x\varphi(t,x)\bigr)\,d\rho_t(x)\,dt
+\int_{\Omega}\varphi(0,x)\,d\rho_0(x)=0\label{eq:weak-eulerian}
\end{align}
for every $\varphi\in C_c^{\infty}([0,T)\times \Omega)$. If $\rho_0\in L^\infty(\Omega)$, we require $\rho\in L^\infty_{\mathrm{loc}}([0,\infty); L^\infty(\Omega))$.
\end{mydef}

\begin{thm}[The Cauchy problem]\label{thm:Cauchy}
    Let $\d\ge 1$ and $\Omega \in \{\R^\d, \T^\d\}$. For any initial measure $\rho_0 \in \P(\Omega)$ and target $\mu \in \P(\Omega) \cap L^\infty(\Omega)$, there exists a global weak solution $\rho$ to \eqref{eq:PDE} satisfying for all $1\le p\le\infty$ and $t>0$,
    \begin{equation}
        \|\rho(t)\|_{L^p} \le\left(  \frac{1-e^{-t \|\mu\|_{L^\infty}}}{\|\mu\|_{L^\infty}} \right)^{-\frac{p-1}{p}}.
    \end{equation}

    If $\rho_0\in L^\infty(\Omega)$, the solution $\rho$ is unique in the class $L^\infty_{\mathrm{loc}}([0,\infty); L^\infty(\Omega))$.

    Finally, there exist smooth and compactly supported radial probability densities $\rho_0$ and $\mu$ on $\R^\d$ such that the unique solution $\rho$ to \eqref{eq:PDE} with initial data $\rho_0$ and target $\mu$ satisfies, with $\lambda:=\mu(0)>0$, for every $\be\in (0,1]$,
    \begin{equation}
        [\rho_t]_{C^{0,\beta}} \ge c_\be \exp\left(\beta\lambda\frac{\d+2}{2\d}t\right),
        \qquad \forall t\ge t_0,
    \end{equation}
    for some $t_0<\infty$ depending on $\d$ and the data $\mu,\rho_0$, and some $c_\beta>0$ depending additionally on $\beta$.
\end{thm}

We next turn from the Cauchy theory to coercivity and quantitative convergence on the torus. Our first result is an intrinsic metric PL inequality under a target density-ratio condition, valid for every source of finite Coulomb energy. The integral-dissipation form follows whenever the Coulomb potential has a $C^1$ representative; in particular, it applies to bounded sources. The uniform target is included as a special case with constant $2/\d$.

\begin{thm}[Metric PL inequality under a target density-ratio condition]\label{thm:bounded-contrast-PL}\label{thm:uniform-target-PL}
Let $\d\ge1$ and let $\mu\in\P(\T^\d)\cap L^\infty(\T^\d)$. Set
\begin{equation}
m:=\operatorname*{ess\,inf}_{\T^\d}\mu,
\qquad
M:=\operatorname*{ess\,sup}_{\T^\d}\mu.
\end{equation}
Regard $F_\mu(\eta):=\mmd^2(\eta,\mu)$ as extended by $+\infty$ outside its finite-energy domain, and let
\begin{equation}
|\partial F_\mu|(\rho)
:=\limsup_{\eta\to\rho\text{ in }W_2}
\frac{\bigl(F_\mu(\rho)-F_\mu(\eta)\bigr)_+}{W_2(\rho,\eta)}
\end{equation}
denote its descending Wasserstein local slope.\footnote{See \cite[Definition~1.2.4, equation~(1.2.6)]{AGS08}.} Then every $\rho\in\P(\T^\d)$ with $F_\mu(\rho)<\infty$ satisfies
\begin{equation}\label{eq:bounded-contrast-slope-PL}
|\partial F_\mu|^2(\rho)
\ge2\left(m-\frac{\d-1}{\d}M\right)F_\mu(\rho).
\end{equation}
For $\d=1$, the coefficient is positive whenever $m>0$. For $\d\ge2$, it is positive whenever $m>0$ and
\begin{equation}\label{eq:bounded-contrast-condition}
\frac{M}{m}<\frac{\d}{\d-1}.
\end{equation}
\end{thm}

\begin{cor}[Integral PL inequality under a target density-ratio condition]\label{cor:bounded-contrast-integral-PL}
Under the assumptions of \cref{thm:bounded-contrast-PL}, let $\rho\in\P(\T^\d)$ have finite Coulomb energy, set $h:=\g\ast(\rho-\mu)$, and suppose that $h$ has a $C^1$ representative. Then
\begin{equation}\label{eq:bounded-contrast-PL}
\dissip(\rho\mid\mu)
\ge |\partial F_\mu|^2(\rho)
\ge2\left(m-\frac{\d-1}{\d}M\right)\mmd^2(\rho,\mu).
\end{equation}
The regularity assumption on $h$ holds, in particular, if $\rho\in L^p(\T^\d)$ for some $p>\d$.

In particular, if $\mu\equiv1$, then $m=M=1$ and
\begin{equation}\label{eq:uniform-target-PL}
\dissip(\rho\mid1)
\ge |\partial F_1|^2(\rho)
\ge \frac{2}{\d}\mmd^2(\rho,1).
\end{equation}
In particular, if $\|\mu-1\|_{L^\infty}\le\varepsilon$, then
\begin{equation}\label{eq:near-uniform-PL}
\dissip(\rho\mid\mu)
\ge\frac{2}{\d}\bigl(1-(2\d-1)\varepsilon\bigr)
\mmd^2(\rho,\mu),
\end{equation}
which is coercive for $\varepsilon<(2\d-1)^{-1}$.
\end{cor}

\begin{remark}[Why the metric formulation is needed]\label{rem:singular-source-dissipation}
For a general finite-energy measure, $\nabla h$ is only a Lebesgue-almost-everywhere equivalence class, so its integral against a singular source need not be intrinsic. For example, take $\mu\equiv1$ and $\rho=\delta_0$ on $\T$ or, when $\d\ge2$, $\rho=\delta_0(dx_1)\otimes dx_2\cdots dx_\d$ on $\T^\d$. The Coulomb field has a jump across the point or hyperplane supporting $\rho$; changing its value there does not change the field as a Lebesgue equivalence class but does change $\int|\nabla h|^2\,d\rho$. The metric slope avoids this ambiguity.
\end{remark}

\begin{cor}[Convergence under the target density-ratio condition]\label{cor:bounded-contrast-convergence}
Let $\d\ge1$ and let $\mu\in\P(\T^\d)\cap L^\infty(\T^\d)$. With $m$ and $M$ as in \cref{thm:bounded-contrast-PL}, assume that
\begin{equation}
\kappa:=m-\frac{\d-1}{\d}M>0.
\end{equation}
Let $\rho$ be a solution constructed in \cref{thm:Cauchy} with target $\mu$ and arbitrary initial datum $\rho_0\in\P(\T^\d)$. Then, for every $t\ge s>0$,
\begin{equation}\label{eq:bounded-contrast-convergence}
\mmd^2(\rho_t,\mu)
\le e^{-2\kappa(t-s)}\mmd^2(\rho_s,\mu).
\end{equation}
\end{cor}

For the uniform target, \cref{cor:bounded-contrast-convergence} gives the exponent $2/\d$. Its main content is the nondefective global static coercivity and the resulting multiplicative estimate; in dimensions $\d\ge4$, this exponent is not an improvement over the subendpoint asymptotic rates furnished by the following result for general uniformly positive targets.

\begin{thm}[Convergence]\label{thm2}
    Let $\d\ge 1$, $\rho_0\in\P(\T^\d)$, and $\mu\in\P(\T^\d)\cap L^\infty(\T^\d)$ satisfy
    \begin{equation}
        \lambda:=\operatorname*{ess\,inf}_{\T^\d}\mu>0.
    \end{equation}
    Let $\rho$ be any solution constructed in \cref{thm:Cauchy}. For every $\gamma\in(0,2\lambda/3)$ and every $s>0$, there is a constant
    $C_{s,\gamma}<\infty$ such that
    \begin{equation}\label{eq:convergence-subendpoint}
        \forall t\ge s, \quad \mmd^2(\rho_t,\mu)\le C_{s,\gamma}e^{-\gamma(t-s)}.
    \end{equation}
\end{thm}

\begin{remark}[Quantitative form]\label{rem:convergence-quantitative-form}
The proof of \cref{thm2} yields the following more precise estimate. Fix $s>0$ and $\varepsilon\in(0,1)$, and set
\begin{equation}
a_\varepsilon=2\varepsilon\lambda,
\qquad
b_\varepsilon=(1-\varepsilon)\lambda.
\end{equation}
Define
\begin{equation}
\Theta_\varepsilon(\tau):=
\begin{cases}
   \dfrac{e^{-b_\varepsilon\tau}-e^{-a_\varepsilon\tau}}
   {a_\varepsilon-b_\varepsilon}, & \varepsilon\ne\frac13,\\[2ex]
   \tau e^{-\frac{2\lambda}{3}\tau}, & \varepsilon=\frac13,
\end{cases}
\end{equation}
and
\begin{equation}
K_{\d,\mu}(s):=
\begin{cases}
   4\|\g'\|_{L^\infty(\T)}^2, & \d=1,\\[0.5ex]
   C_\d\!\left(1+2^{1/\d}\!\left[\|\mu\|_{L^\infty}
   \left((1-e^{-s\|\mu\|_{L^\infty}})^{-1}+1\right)\right]^{1-1/\d}\right)^2,
   & \d\ge2,
\end{cases}
\end{equation}
where, for $\d\ge2$, $C_\d>0$ denotes the dimension-only constant resulting from the squared near--far Coulomb-field estimate used below. Then, for every $t\ge s$,
\begin{equation}
\mmd^2(\rho_t,\mu)
\le e^{-a_\varepsilon(t-s)}\mmd^2(\rho_s,\mu)+\varepsilon\lambda K_{\d,\mu}(s)e^{-b_\varepsilon s}
\Theta_\varepsilon(t-s).\label{eq:convergence-exact}
\end{equation}
The estimate \eqref{eq:convergence-subendpoint} follows by choosing $\varepsilon$ so that
$\gamma<\min\{2\varepsilon\lambda,(1-\varepsilon)\lambda\}$.
\end{remark}

\begin{remark}
    Chizat, Colombo, Colombo, and Fern\'andez-Real \cite{CCCFR26} independently establish global exponential convergence at the Coulomb endpoint on the torus for uniformly positive targets, including initial densities with vacuum regions. Their argument recovers a positive lower bound along the flow through an exponential gap-filling mechanism and then applies the usual PL inequality. The proof of \cref{thm2} instead keeps the possible vacuum as an explicit defect term. A further distinction is that the solutions in \cref{thm:Cauchy} may start from arbitrary Borel probability measures and become bounded instantaneously.
\end{remark}

The positive results above are complemented by two distinct limitations of PL coercivity on $\T^\d$. First, \cref{prop:counterexvanishingmu} gives a fixed target for which no PL inequality holds globally over all admissible sources. Second, \cref{prop:counterexamplePLinfmu} shows that, for every $\d\ge2$ and $0<\lambda<1$, no PL constant depending only on $(\d,\lambda)$ can hold uniformly over all target densities satisfying $\mu\ge\lambda$. This is consistent with \cref{thm:bounded-contrast-PL}, whose constant depends on both the lower and upper bounds of the target.
\begin{prop}\label{prop:counterexvanishingmu}
    Fix $\d\ge1$. There exists a nonnegative Lipschitz probability density $\mu:\T^\d\to\R$ which vanishes only at one point and a family of probability densities $(\rho_\delta)_{\delta>0}$ such that
    \begin{equation}
        \frac{\mmd^2(\rho_\delta,\mu)}{\dissip(\rho_\delta\mid\mu)}\longrightarrow+\infty
        \qquad\text{as }\delta\downarrow0.
    \end{equation}
    In particular, for every $\varepsilon>0$ there exists a source density $\rho$ such that
    $\mmd^2(\rho,\mu)\ge\varepsilon^{-1}\dissip(\rho\mid\mu)$.
\end{prop}

\begin{prop}\label{prop:counterexamplePLinfmu}
Let $\d\ge2$ and $0<\lambda<1$. There is no constant $c=c(\d,\lambda)>0$ such that
\begin{equation}
\dissip(\rho\mid\mu)\ge c\,\mmd^2(\rho,\mu)
\end{equation}
for every target density $\mu\ge\lambda$ almost everywhere on $\T^\d$ and every source density $\rho$.
\end{prop}

The restriction $\d\ge2$ in the preceding proposition is essential; see \cref{rem:d1-global-pl}.

We next state the positive Euclidean result, which concerns the validity of the PL inequality for radial data satisfying source-support inclusion. It will be contrasted below with obstructions on the unrestricted whole-space class.

\begin{thm}[PL inequality for radial data on $\R^\d$]
\label{thm:connected-support-PL}
Let \(\d\ge 2\) and let
\begin{equation}
A_{a,R}:=\{x\in\R^\d: a\le |x|\le R\},\qquad 0\le a<R<\infty.
\end{equation}
For any radial probability densities $\mu,\rho\in \P(\R^\d)\cap L^\infty(\R^\d)$ satisfying the conditions
\begin{enumerate}[label=\roman*.,leftmargin=3em]
    \item $\supp \mu = A_{a,R} $ for some $0\le a<R<\infty$;
    \item $\supp\rho \subset \supp \mu$;
    \item there exist constants $A,\la>0$ such that $A \ge \mu \ge \la$ almost everywhere on $\supp \mu$,
\end{enumerate}
one has
\begin{equation}
\mmd^2(\rho,\mu) \le C_{\d,\lambda,A}\,\dissip(\rho\mid\mu),
\qquad
C_{\d,\lambda,A}:=\max\Bigl\{\frac{1}{\lambda},\frac{\d^2A}{2\lambda^2}\Bigr\}.
\end{equation}
In particular, if \(\supp\mu=B_R\) and \(0<\lambda\le \mu\le A\) almost everywhere on \(B_R\), then the PL inequality holds for every radial probability density \(\rho\in\P(\R^\d)\cap L^\infty(\R^\d)\) supported in \(B_R\).

\end{thm}

\begin{remark}
    Let us comment on the assumptions of \cref{thm:connected-support-PL}. The theorem uses radial symmetry, connectedness of the target support, and source-support inclusion. The whole-space results below show that the inclusion restriction cannot simply be discarded: if the source may start arbitrarily far from the target, bounded propagation speed prevents any global coercivity or uniform convergence statement. Those examples do not determine, however, whether radial symmetry can be removed while retaining connected target support, source-support inclusion, and the target bounds in \cref{thm:connected-support-PL}.
\end{remark}
\begin{remark}
    Within the radial class, the two support conditions in \cref{thm:connected-support-PL} are sharp; see \cref{prop:sharpi,prop:sharpii}. We discuss the lower-bound condition in \cref{rem:sharpiii}.
\end{remark}

As a corollary of \cref{thm:connected-support-PL}, we obtain, to the best of our knowledge, the first exponential convergence result for the flow \eqref{eq:PDE} on $\R^\d$  for $\d\ge2$, albeit restricted to the radial setting.
\begin{cor}\label{cor:radialconvergence}
    Let $\d\ge 2$ and consider $\mu,\rho_0\in \mathcal{P}(\R^\d)$ satisfying the assumptions of \cref{thm:connected-support-PL}. Then, any global weak solution $\rho$ to \eqref{eq:PDE} constructed as in \cref{thm:Cauchy} satisfies:
    \begin{equation}
        \forall t\ge 0, \quad \mmd^2(\rho_t,\mu) \le e^{-t/C_{\d,\la,A}}\mmd^2(\rho_0,\mu),
    \end{equation}
    where $\displaystyle C_{\d,\la,A}:= \max\Bigl\{\frac{1}{\lambda},\frac{\d^2A}{2\lambda^2}\Bigr\}$.
\end{cor}

We next turn to the unrestricted whole-space class. The first result quantifies the elementary but important fact that a bounded Coulomb velocity cannot transport a remote source to the target faster than linearly in their initial separation.

\begin{thm}[Persistence on the support-travel-time scale]\label{thm:whole-space-persistence}
Let $\d\ge1$ and let $\mu\in\P(\R^\d)\cap L^\infty(\R^\d)$ be compactly supported. Let $(\rho_{0,n})_{n\ge1}\subset\P(\R^\d)\cap L^\infty(\R^\d)$ be compactly supported probability densities such that
\begin{equation}
M_0:=\sup_n\|\rho_{0,n}\|_{L^\infty}<\infty,
\qquad
L_0:=\sup_n\diam(\supp\rho_{0,n})<\infty,
\end{equation}
and
\begin{equation}
D_n:=\dist(\supp\rho_{0,n},\supp\mu)\longrightarrow\infty.
\end{equation}
Let $\rho_{n,t}$ denote the corresponding bounded solution furnished by \cref{thm:Cauchy}, and, for $D\ge2$, define
\begin{equation}
\Phi_\d(D):=
\begin{cases}
D, & \d=1,\\
\log D, & \d=2,\\
1, & \d\ge3.
\end{cases}
\end{equation}
There exist constants $c_0,c_1,C_1>0$ and $n_0\in\N$, depending only on $\d$, $\mu$, $M_0$, and $L_0$, such that, for every $n\ge n_0$ and every $0\le t\le c_0D_n$,
\begin{equation}\label{eq:persistence-comparison}
c_1\Phi_\d(D_n)
\le \mmd^2(\rho_{n,t},\mu)
\le \mmd^2(\rho_{0,n},\mu)
\le C_1\Phi_\d(D_n).
\end{equation}
In particular, with $\eta:=c_1/C_1>0$,
\begin{equation}\label{eq:persistence-fraction}
\mmd^2(\rho_{n,t},\mu)
\ge \eta\,\mmd^2(\rho_{0,n},\mu)
\qquad (0\le t\le c_0D_n).
\end{equation}
\end{thm}

This persistence estimate immediately rules out a decay rate that is uniform over the location of the initial datum: by translating a fixed compactly supported profile farther from the target, one can delay any fixed fractional decrease for arbitrarily long times.

\begin{cor}[No universal whole-space convergence rate]\label{cor:no-universal-rate}
Let $\d\ge1$ and let $\mu\in\P(\R^\d)\cap L^\infty(\R^\d)$ be compactly supported. Fix a nonnegative probability density $\theta\in C_c^\infty(\R^\d)$. There is no function $\gamma:[0,\infty)\to(0,\infty)$ with $\gamma(t)\to0$ as $t\to\infty$ such that, for every $a\in\R^\d$, the solution with initial datum $\rho_{0,a}=\theta(\cdot-a)$ satisfies
\begin{equation}
\mmd^2(\rho_{a,t},\mu)
\le \gamma(t)\mmd^2(\rho_{0,a},\mu)
\qquad\text{for every }t\ge0.
\end{equation}
Thus, a multiplicative decay modulus cannot be uniform over the initial datum even when its shape, $L^\infty$ norm, and support diameter are fixed and only its location varies.
\end{cor}

The same travel-time obstruction also rules out global PL coercivity. Indeed, averaging the dissipation over a persistence interval produces a time at which the dissipation-to-energy ratio is of order $D_n^{-1}$.

\begin{cor}[Dynamical obstruction to global PL inequalities]\label{cor:dynamical-global-PL}
Under the assumptions of \cref{thm:whole-space-persistence}, there are times $t_n\in(0,c_0D_n)$ and a constant $C<\infty$, independent of $n$, such that
\begin{equation}\label{eq:dynamical-PL-ratio}
\frac{\dissip(\rho_{n,t_n}\mid\mu)}{\mmd^2(\rho_{n,t_n},\mu)}
\le \frac{C}{D_n}\longrightarrow0.
\end{equation}
In particular, no positive constant depending only on $\d$, $\mu$, and a common $L^\infty$ bound can yield a PL inequality on the unrestricted class of compactly supported bounded probability densities.
\end{cor}

The preceding obstruction selects bad states along solutions. We complement it with a direct construction of smooth densities whose dissipation-to-energy ratio degenerates independently of the flow.

\begin{prop}[Direct static obstruction to global PL inequalities]\label{prop:static-global-PL}
Let $\d\ge1$ and let $\mu\in\P(\R^\d)\cap L^\infty(\R^\d)$ be compactly supported. There exists a sequence $\rho_n\in C_c^\infty(\R^\d)\cap\P(\R^\d)$ such that
\begin{equation}\label{eq:static-PL-ratio}
\frac{\dissip(\rho_n\mid\mu)}{\mmd^2(\rho_n,\mu)}\longrightarrow0.
\end{equation}
Consequently, there is no finite constant $C=C(\d,\mu)$ such that
\begin{equation}
\mmd^2(\rho,\mu)\le C\,\dissip(\rho\mid\mu)
\end{equation}
for every smooth compactly supported probability density $\rho$.
\end{prop}

\begin{remark}\label{rem:two-global-PL-obstructions}
The direct construction reveals a dimension-dependent mechanism: translations suffice in dimensions one and two, whereas in dimensions at least three the translated source must also be dilated in order to make its self-interaction energy $\int_{\R^\d}\g\ast\rho\,d\rho$ small.
\end{remark}

We conclude the overview of our main results with rigidity of Lagrangian critical points (following the terminology of \cite{BV25}) and its application to planar long-time behavior. 

\begin{thm}[Rigidity under absolute continuity of the positive part]
\label{thm:lagrangian-rigidity}
Let $\d\ge1$ and let $\Omega\in\{\R^\d,\T^\d\}$. Let $\rho,\mu\in\P(\Omega)$, and suppose that $(\rho-\mu)^+$ is absolutely continuous with respect to Lebesgue measure on $\Omega$.

Let $h\in W^{1,1}_{\mathrm{loc}}(\Omega)$ denote the corresponding distributional Coulomb potential, so that
\begin{equation}\label{eq:lagrangian-critical-poisson}
\Delta h=\mu-\rho
\quad\text{in }\mathcal D'(\Omega).
\end{equation}
Assume that
\begin{equation}\label{eq:lagrangian-critical-vanishing}
\nabla h=0
\quad\text{almost everywhere with respect to }(\rho-\mu)^+.
\end{equation}
Then $\rho=\mu$.
\end{thm}

\begin{thm}[Planar convergence]\label{thm:planar-convergence}
Let $\mu\in\P(\R^2)\cap L^\infty(\R^2)$, and let $\rho$ be a global weak solution constructed in \cref{thm:Cauchy}. Suppose that, for some $s_0>0$,
\begin{equation}\label{eq:planar-finite-energy-time}
\mmd^2(\rho_{s_0},\mu)<\infty.
\end{equation}
Then
\begin{equation}\label{eq:planar-narrow-convergence}
\rho_t\rightharpoonup\mu
\qquad\text{narrowly as }t\to\infty.
\end{equation}
Moreover,
\begin{equation}\label{eq:planar-weak-star-convergence}
\rho_t\stackrel{*}{\rightharpoonup}\mu
\qquad\text{in }L^\infty(\R^2),
\end{equation}
and, for every $\alpha>0$,
\begin{equation}\label{eq:planar-negative-sobolev-convergence}
\|\rho_t-\mu\|_{H^{-\alpha}(\R^2)}
\longrightarrow0.
\end{equation}
In particular, the conclusions hold when $\rho_0$ and $\mu$ are bounded, compactly supported probability densities.
\end{thm}

The topologies in \cref{thm:planar-convergence} do not by themselves imply convergence in the Coulomb MMD; see \cref{rem:planar-coulomb-endpoint}.

\subsection{Companion works}

We take this opportunity to advertise three closely related manuscripts in preparation. They address overlapping but complementary aspects of the questions considered here; none of their results is used in the present paper.

Joint work of the second author with Slep\v{c}ev and Wang studies the Wasserstein gradient flow of MMDs generated by the energy kernels
\begin{equation}
K(x,y)=-|x-y|^q,
\qquad
0<q<2,
\end{equation}
on $\R^\d$ \cite{RosenzweigSlepcevWang2026EnergyMMD}. In the regime $\d+q-2>0$, that manuscript establishes global well-posedness in finite subcritical $L^p$ classes and at the $L^\infty$ endpoint, proves global noncollision for the associated $N$-particle dynamics, and obtains a modulated-energy mean-field estimate.\footnote{The analogous modulated-energy mean-field estimate for our setting already follows from existing work (see, e.g., \cite[Section~6.2.1, in particular Theorem~6.3]{Ser24}).} It also studies structural properties of the energy MMD, stationary critical states, and obstructions to global PL inequalities and data-independent decay. For the one-dimensional Coulomb $N$-particle dynamics, it moreover exhibits an explicit two-time-scale mechanism: a slow transport phase governed by the time needed for the particles to reach the support of the target, followed, when the target density is uniformly positive on its support, by exponential relaxation to the discrete equilibrium configuration. This gives a sharper point of contact with the present paper: our whole-space persistence theorem isolates the corresponding support-travel-time obstruction for continuum solutions in every dimension, whereas the energy-kernel manuscript also resolves the subsequent fast phase in the explicit one-dimensional particle setting. 

The authors also study the unconfined, target-free quadratic Wasserstein gradient flow of sub-Coulomb Riesz interaction energies in the range $-2<s<\d-2$ \cite{ChodronDeCourcelRosenzweig2026SubCoulomb}. We construct canonical global weak solutions in the scaling-critical and subcritical Lebesgue classes, prove uniqueness in the corresponding Eulerian class and monotonicity of lower Lebesgue norms, and show that higher Lebesgue integrability is not created. This absence of any hypercontractivity is in stark contrast to the Coulomb/super-Coulomb case. We also develop a radial measure-valued well-posedness theory on annuli. The aforementioned results also apply to the confined setting, including the MMD target case, under suitable assumptions on the confinement. 
The main contrast with the present Coulomb case is that the sub-Coulomb evolution is transport-dominated and preserves, rather than improves, Lebesgue singularities.

Joint work of the second author with Hess-Childs and Serfaty studies the optimal empirical quantization problem for homogeneous and screened Riesz MMDs and diagonal-excluded modulated energies \cite{HessChildsRosenzweigSerfaty2026Quantization}. Given a target probability measure $\mu$ and a sample size $N$, the empirical quantization problem asks for points $x_1,\ldots,x_N$ minimizing the prescribed discrepancy or energy between $\mu$ and the equal-weight empirical measure $\frac{1}{N}\sum_{i=1}^N\delta_{x_i}$; see \cite{GrafLuschgy2000} for the classical theory. The work \cite{HessChildsRosenzweigSerfaty2026Quantization} proves fixed-cardinality existence throughout the homogeneous Riesz and logarithmic range, upper and lower bounds for the quantization rate which match in their dependence on the number of samples, and separation estimates for empirical minimizers. It also treats a larger class of (screened) Riesz-type kernels. This work is a static/equilibrium complement to the present dynamical analysis. We mention that, for energy kernels, overlapping results were recently obtained independently by Colasanto, Focardi, Fornasier, and Mattesini \cite{colasantoSharpRatesMMD2026}.

\subsection{Organization of the paper}

The remainder of the paper is organized as follows. \Cref{sec:cauchy-problem} establishes the Cauchy theory and the regularity counterexample. \Cref{sec:torus-coercivity} treats convergence and coercivity on the torus, including the global PL inequality near the uniform target and the limitations of such coercivity. \Cref{sec:radial-pl} proves the radial source-support-inclusion PL inequality and the sharpness of its support hypotheses. \Cref{sec:whole-space-obstructions} develops the whole-space obstructions at spatial infinity, first dynamically through the support-travel-time scale and then by a direct static construction. \Cref{sec:criticality-planar} classifies Lagrangian critical points for which $(\rho-\mu)^+$ is absolutely continuous and proves qualitative convergence on $\R^2$. \Cref{sec:future-directions} records open problems and future directions suggested by the preceding results.

\subsection{Acknowledgments}

The authors thank L\'ena\"{\i}c Chizat and Roberto Colombo for discussion about their work \cite{CCCFR26}. The second author also thanks Dejan Slep\v{c}ev and Lihan Wang for their collaboration which inspired parts of the present work.

\section{The Cauchy problem}\label{sec:cauchy-problem}

In this section, we study the Cauchy problem \eqref{eq:PDE} and prove \cref{thm:Cauchy}. We first establish the existence of global weak solutions to \eqref{eq:PDE} under the assumption $\rho_0\in \mathcal{P}(\Omega)$ and $\mu\in\mathcal{P}(\Omega)\cap L^\infty(\Omega)$ thanks to an ultracontractivity argument. We do not know whether uniqueness holds for such global weak solutions. Nevertheless, uniqueness is recovered if we assume $\rho_0\in \mathcal{P}(\Omega)\cap L^\infty(\Omega)$ and $\mu\in\mathcal{P}(\Omega)\cap L^\infty(\Omega)$, by adapting the classical Yudovich theory for 2D Euler equations \cite{yudovichNonstationaryFlowIdeal1963}. Finally, we show that the H\"older continuity constants of solutions can deteriorate with time at an exponential rate, even if the initial data and target are smooth.

\subsection{Existence}

In this subsection, we construct global weak solutions to \eqref{eq:PDE} starting from $\rho_0\in \mathcal{P}(\Omega)$.

Our starting point is the following approximation scheme:
\begin{equation}\label{eq:PDEscheme}
\begin{cases}
    \p_t \rho_\ep - \div (\rho_\ep \nab\g\ast (\rho_\ep - \mu_\ep)) = 0, \\
    \rho_\ep\vert_{t=0} = \rho_{0,\ep}. 
\end{cases}
\end{equation}
When $\rho_{0,\ep}$ and $\mu_\ep$ are smooth probability densities, there exists a unique global smooth solution $\rho_\ep$ to \eqref{eq:PDE}; this follows by a straightforward adaptation of the argument in \cite{linHydrodynamicLimitGinzburgLandau2000}, since the only difference is an additional smooth bounded drift term.

The approximation preserves two basic structural properties of the flow: conservation of mass and dissipation of the Coulomb energy.
\begin{itemize}
    \item the mass is conserved:
    \begin{equation}
        \forall t\ge 0, \quad \int_{\Omega} \rho_\ep(t,x)\, dx = \int_{\Omega} \rho_{0,\ep}\, dx = 1,
    \end{equation}
    \item the energy decreases:
    \begin{equation}
       \forall t\ge s\ge 0,\quad \mmd^2(\rho_\ep(t),\mu_\ep) + \int_s^t \dissip(\rho_\ep(\tau)\mid\mu_\ep) \, d\tau = \mmd^2(\rho_{\ep}(s),\mu_\ep).
    \end{equation}
\end{itemize}

We now record the following estimate, which allows us to construct global weak solutions for initial data that are mere probability measures.
\begin{prop}[Ultracontractivity]\label{prop:hyper}
    For any $p \in (1,\infty]$, the smooth solution $\rho_\ep$ to \eqref{eq:PDEscheme} satisfies
    \begin{equation}\label{eq:hyper_est}
        \forall t>0,\quad \|\rho_\ep(t)\|_{L^p(\Omega)} \le \left( e^{-t \|\mu_\ep\|_{L^\infty}} \|\rho_{0,\ep}\|_{L^p(\Omega)}^{-\frac{p}{p-1}} + \frac{1-e^{-t \|\mu_\ep\|_{L^\infty}}}{\|\mu_\ep\|_{L^\infty}} \right)^{-\frac{p-1}{p}}.
    \end{equation}
\end{prop}
\begin{remark}
    Dropping the nonnegative term involving the $L^p$ norm of the initial data, we obtain the estimate
    \begin{equation}
         \forall t>0,\quad \|\rho_\ep(t)\|_{L^p(\Omega)} \le\left(  \frac{1-e^{-t \|\mu_\ep\|_{L^\infty}}}{\|\mu_\ep\|_{L^\infty}} \right)^{-\frac{p-1}{p}},
    \end{equation}
    which is independent of the initial data and provides an instantaneous $L^\infty$ regularization if $\mu\in L^\infty$, hence the name ``ultracontractivity''.
\end{remark}
\begin{proof}[Proof of \cref{prop:hyper}]
    Let us drop the $\ep$ subscript for simplicity. Differentiating with respect to time, we have
    \begin{align}
        \frac{d}{dt} \int_{\Omega} \rho^p\, dx &= \int_{\Omega} p \rho^{p-1} \div(\rho \nabla \g\ast (\rho - \mu)) \, dx \notag\\
        &= -p(p-1) \int_{\Omega} \rho^{p-1} \nabla \rho \cdot (\nabla \g\ast (\rho - \mu)) \, dx \notag\\
        &= -(p-1) \int_{\Omega} \nabla(\rho^p) \cdot (\nabla \g\ast (\rho - \mu)) \, dx \notag\\
        &= (p-1) \int_{\Omega} \rho^p \Delta \g\ast (\rho - \mu) \, dx.
    \end{align}
    Using the fact that $\Delta \g\ast (\rho - \mu) = \mu - \rho$,
    \begin{align}
        \frac{d}{dt} \int_{\Omega} \rho^p\, dx &= (p-1) \int_{\Omega} \rho^p (\mu - \rho) \, dx \notag\\
        &= (p-1) \int_{\Omega} \rho^p \mu \, dx - (p-1) \int_{\Omega} \rho^{p+1} \, dx.
    \end{align}
    By H\"older's inequality, $\displaystyle\int_{\Omega} \rho^p \mu \, dx\le \|\mu\|_{L^\infty} \|\rho\|_{L^p}^p$. Furthermore, since $\displaystyle\int_{\Omega} \rho\, dx = 1$, Jensen's inequality implies $\|\rho\|_{L^p}^p \le \|\rho\|_{L^{p+1}}^{\frac{(p-1)(p+1)}{p}}$, or equivalently $\|\rho\|_{L^{p+1}}^{p+1} \ge \|\rho\|_{L^p}^{\frac{p^2}{p-1}}$. Thus, we obtain the differential inequality
    \begin{equation}
        \frac{d}{dt} \|\rho\|_{L^p}^p\le (p-1) \|\mu\|_{L^\infty} \|\rho\|_{L^p}^p - (p-1) \|\rho\|_{L^p}^{\frac{p^2}{p-1}}.
    \end{equation}
    Setting $y(t) = \|\rho(t)\|_{L^p}^{-\frac{p}{p-1}}$, we find that $y$ satisfies
    \begin{equation}
        y' \ge -\|\mu\|_{L^\infty} y + 1.
    \end{equation}
    Integrating this inequality yields
    \begin{equation}
        y(t) \ge e^{-t \|\mu\|_{L^\infty}} y(0) + \frac{1 - e^{-t \|\mu\|_{L^\infty}}}{\|\mu\|_{L^\infty}},
    \end{equation}
    which, upon substituting back $y(t) =  \|\rho(t)\|_{L^p}^{-\frac{p}{p-1}}$, gives the desired estimate \eqref{eq:hyper_est}. The $L^\infty$ estimate follows by taking $p \to \infty$.
\end{proof}

We now use the ultracontractivity estimate in \cref{prop:hyper} to complete the existence part of \cref{thm:Cauchy} for arbitrary $\rho_0\in\P(\Omega)$. We give the compactness argument on $\R^\d$; on $\T^\d$ the same proof is global and the tightness step below is unnecessary. Let $\varphi$ be a nonnegative compactly supported smooth function of integral one, let $\varepsilon_k\downarrow0$, and set
\begin{equation}
\varphi_k(x):=\varepsilon_k^{-\d}\varphi(x/\varepsilon_k),\qquad
\rho_{0,k}:=\rho_0\ast\varphi_k,\qquad
\mu_k:=\mu\ast\varphi_k.
\end{equation}
Let $\rho_k$ be the corresponding smooth solution and put
\begin{equation}
E_k:=-\nabla\g\ast(\rho_k-\mu_k).
\end{equation}
For every $0<\tau<T$ and every $p\in(1,\infty)$, \cref{prop:hyper} gives
\begin{equation}\label{eq:rho-compactness-bounds}
\sup_k\|\rho_k\|_{L^\infty([\tau,T];L^1\cap L^\infty)}<\infty.
\end{equation}
The usual near--far decomposition of the Coulomb field also gives, uniformly in $k$,
\begin{equation}\label{eq:short-time-field}
\|E_k(t)\|_{L^\infty(\R^\d)}\le C
\begin{cases}
1, & \d=1,\\
1+t^{-(\d-1)/\d}, & \d\ge2,
\end{cases}
\qquad 0<t\le T.
\end{equation}
The right-hand side is integrable at the origin. In addition, Calder\'on--Zygmund estimates and \eqref{eq:rho-compactness-bounds} imply, for every ball $B_R$,
\begin{equation}\label{eq:field-space-compactness}
\sup_k\|E_k\|_{L^\infty([\tau,T];W^{1,p}(B_R))}<\infty.
\end{equation}
Since $\partial_t\rho_k+\div(\rho_kE_k)=0$, differentiation of the Poisson representation yields
\begin{equation}
\partial_tE_k=\nabla(-\Delta)^{-1}\div(\rho_kE_k)
\end{equation}
in $\R^\d$. The operator on the right is of order zero; hence \eqref{eq:rho-compactness-bounds}--\eqref{eq:short-time-field} give
\begin{equation}\label{eq:field-time-compactness}
\sup_k\|\partial_tE_k\|_{L^\infty([\tau,T];L^p(\R^\d))}<\infty.
\end{equation}
The compact embedding $W^{1,p}(B_R)\Subset L^p(B_R)$, together with \eqref{eq:field-time-compactness} and the compactness criterion of \cite{Simon87}, therefore yields, after extraction and a diagonal argument,
\begin{equation}\label{eq:strong-field-convergence}
E_k\longrightarrow E\quad\text{strongly in }C([\tau,T];L^p(B_R))
\end{equation}
for every $\tau,T,R$ and every finite $p$. At the same time, $\rho_k$ converges weak-$*$ in $L^\infty([\tau,T]\times B_R)$ to a nonnegative density $\rho$. Consequently, \eqref{eq:strong-field-convergence} and the weak-$*$ convergence of $\rho_k$ imply
$\rho_kE_k\rightharpoonup\rho E$ locally in spacetime.

It remains to control mass at infinity, identify the limiting field, and recover the initial trace. Choose $\zeta_R\in C_c^\infty(\R^\d)$ with $0\le\zeta_R\le1$, $\zeta_R=1$ on $B_R$, $\zeta_R=0$ outside $B_{2R}$, and $\|\nabla\zeta_R\|_\infty\le C/R$, and set $\chi_R:=1-\zeta_R$. Testing the continuity equation against $\zeta_R$, using mass conservation, and invoking \eqref{eq:short-time-field} give
\begin{equation}\label{eq:tightness-cutoff}
\int\chi_R\rho_k(t)\,dx
\le \int\chi_R\rho_{0,k}\,dx+\frac{C}{R}\int_0^t\|E_k(s)\|_\infty\,ds.
\end{equation}
The mollified initial data are uniformly tight, and the last integral is finite uniformly in $k$; hence $(\rho_k(t))_k$ is uniformly tight on every bounded time interval. Similarly, for $\psi\in C_c^1(\R^\d)$,
\begin{equation}\label{eq:initial-trace-estimate}
\left|\int\psi\rho_k(t)\,dx-\int\psi\rho_{0,k}\,dx\right|
\le\|\nabla\psi\|_\infty\int_0^t\|E_k(s)\|_\infty\,ds,
\end{equation}
whose right-hand side tends to zero with $t$, uniformly in $k$. Passing first $k\to\infty$ and then $t\downarrow0$ gives the initial trace $\rho_t\rightharpoonup\rho_0$ narrowly. The same estimates give $\rho\in C([0,T];\P(\R^\d))$ for the narrow topology.
We briefly spell out the field identification. Fix $0<\tau<T$ and a compact set $K_0\subset\R^\d$. For $\d\ge2$, decompose $-\nabla\g\ast(\rho_k-\mu_k)$ according to the relative displacement into the regions $|z|\le\delta$, $\delta<|z|<R$, and $|z|\ge R$. The near contribution is uniformly $O(\delta)$ on $[\tau,T]\times K_0$, because the positive-time $L^\infty$ bounds for $\rho_k$ and $\mu_k$, together with $\int_{|z|\le\delta}|\nabla\g(z)|\,dz\le C_\d\delta$, give this estimate. For fixed $0<\delta<R$, multiplying $-\nabla\g$ by smooth cutoffs supported in $\{\delta/2<|z|<2R\}$ produces a bounded continuous kernel of compact support, so the local weak-$*$ convergence of $\rho_k$ and the convergence of $\mu_k$ identify the corresponding distributional limit on $(\tau,T)\times K_0$. The far contribution is uniformly $O(R^{1-\d})$ by the mass bounds and the decay $|\nabla\g(z)|\lesssim |z|^{1-\d}$. Sending first $k\to\infty$ and then $\delta\downarrow0$ and $R\uparrow\infty$ identifies the strong local limit. When $\d=1$, the positive-time densities are atomless and the cumulative-mass formula $E_k(t,x)=(\rho_k(t)-\mu_k)((-\infty,x))$ identifies the limit directly from narrow convergence. On $\T^\d$, the same argument uses the decomposition of the periodic Coulomb field into its local singular part and a smooth remainder; compactness is global, and in dimension one the jump of the local kernel is harmless because the limiting measures are atomless. Consequently, \eqref{eq:strong-field-convergence} identifies the local strong limit as
\begin{equation}
E=-\nabla\g\ast(\rho-\mu).
\end{equation}
Passing to the limit in the weak formulation now shows that $\rho$ is a global weak solution in the sense of \cref{def:weak-solution}; it has unit mass and satisfies the bounds in \cref{thm:Cauchy}. On the torus, \eqref{eq:field-space-compactness}--\eqref{eq:strong-field-convergence} hold on the whole domain, and the proof is identical after omitting \eqref{eq:tightness-cutoff}.

We finally record the energy inequality for the constructed solution, which will be used in the later convergence arguments. On every positive time interval on which the Coulomb energy is finite, one has
\begin{equation}\label{eq:weak-energy-inequality}
\mmd^2(\rho_t,\mu)+\int_s^t\dissip(\rho_\tau\mid\mu)\,d\tau
\le\mmd^2(\rho_s,\mu),\qquad t\ge s>0.
\end{equation}
On $\T^\d$ the energy is finite for every $s>0$. On $\R^\d$, \eqref{eq:weak-energy-inequality} holds whenever $\mmd(\rho_s,\mu)<\infty$, in particular for the compactly supported bounded data used in \cref{cor:radialconvergence,thm:whole-space-persistence}. We justify this assertion by restarting the construction at time $s$. By \cref{thm:Cauchy}, $\rho_s\in L^\infty$. Let $\varphi_k$ be a standard approximate identity and set
 $\rho_{s,k}:=\rho_s\ast\varphi_k$ and $\mu_k:=\mu\ast\varphi_k$, with periodic convolution on $\T^\d$. Since convolution by $\varphi_k$ is contractive on $\dot H^{-1}$ and converges strongly to the identity there,
\begin{equation}\label{eq:restart-energy-convergence}
\mmd^2(\rho_{s,k},\mu_k)
\le\mmd^2(\rho_s,\mu),
\qquad
\mmd^2(\rho_{s,k},\mu_k)\longrightarrow\mmd^2(\rho_s,\mu).
\end{equation}
The smooth flows from $\rho_{s,k}$ with targets $\mu_k$ satisfy the energy identity. Their global $L^2$ field bound from \eqref{eq:restart-energy-convergence}, the preceding compactness and field-identification arguments, and weak lower semicontinuity yield a global weak solution $\widetilde\rho$ from $\rho_s$ satisfying \eqref{eq:weak-energy-inequality}. Since $\rho_s$ is bounded, uniqueness in the next subsection identifies $\widetilde\rho_t$ with $\rho_{s+t}$. Restarting at any later positive time of finite energy proves the claim on every such interval.

On $\T^\d$, the original simultaneous-mollification construction also extends the energy inequality to the initial time whenever
\begin{equation}
\mmd^2(\rho_0,\mu)<\infty.
\end{equation}
Indeed, for the approximations $\rho_{0,k}=\rho_0\ast\varphi_k$ and $\mu_k=\mu\ast\varphi_k$ used above,
\begin{equation}
\mmd^2(\rho_{0,k},\mu_k)\longrightarrow\mmd^2(\rho_0,\mu)
\end{equation}
by strong continuity of convolution in $\dot H^{-1}$. The smooth energy identities therefore give a uniform bound for the continuity-equation actions
\begin{equation}
\int_0^t\int_{\T^\d}|E_k|^2\rho_k\,dx\,d\tau.
\end{equation}
The action bound makes the fluxes $j_k:=\rho_kE_k$ weak-$*$ precompact as vector measures on $[0,t]\times\T^\d$. After extraction, the continuity equation passes to the limit, the positive-time field identification above gives $j=\rho E$ almost everywhere on $(0,t)$, and lower semicontinuity of the action yields
\begin{equation}\label{eq:initial-energy-inequality-torus}
\mmd^2(\rho_t,\mu)+\int_0^t\dissip(\rho_\tau\mid\mu)\,d\tau
\le\mmd^2(\rho_0,\mu),
\qquad t>0.
\end{equation}
Since $\rho_t\rightharpoonup\rho_0$ narrowly as $t\downarrow0$ and the Coulomb energy is lower semicontinuous under narrow convergence on the torus, \eqref{eq:initial-energy-inequality-torus} implies
\begin{equation}\label{eq:initial-energy-continuity-torus}
\mmd^2(\rho_t,\mu)\longrightarrow\mmd^2(\rho_0,\mu)
\qquad\text{as }t\downarrow0.
\end{equation}
Moreover, the continuity-equation characterization of absolutely continuous Wasserstein curves \cite[Theorem~8.3.1]{AGS08} and Cauchy--Schwarz give
\begin{equation}\label{eq:initial-action-W2}
W_2^2(\rho_0,\rho_t)
\le t\int_0^t\dissip(\rho_\tau\mid\mu)\,d\tau
\le t\bigl(\mmd^2(\rho_0,\mu)-\mmd^2(\rho_t,\mu)\bigr).
\end{equation}

On $\R^\d$, for compactly supported bounded initial data, the Osgood Lagrangian representation established below and the same near--far field estimate as in \eqref{eq:short-time-field} keep the support and velocity bounded on each finite time interval. Hence $\rho E\in L^\infty([0,T];L^2)$, $\rho-\mu\in W^{1,\infty}([0,T];\dot H^{-1})$, and the Hilbert-space chain rule applies starting at $t=0$. In particular, $\mmd(\rho_s,\mu)\to\mmd(\rho_0,\mu)$ as $s\downarrow0$, and \eqref{eq:weak-energy-inequality} extends to $s=0$.

\subsection{Uniqueness}

We provide an Osgood stability estimate for bounded global weak solutions, which in particular implies uniqueness. The result is reminiscent of the classical Yudovich theory for two-dimensional Euler. Uniqueness in a bounded-density weak class was also obtained independently in \cite{CCCFR26} by a characteristic argument; the proof below uses the modulated energy.

The key ingredient is the following estimate, which can be found, for example, in \cite{SV14}.

\begin{lemma}[Log-Lipschitz regularity of the velocity field]\label{lem:loglip}
    Let $\rho,\mu \in L^1\cap L^\infty(\Omega)$. The velocity field $v = -\nabla \g\ast (\rho - \mu)$ is log-Lipschitz, in the sense that there exists a constant $C$ depending only on $\d$ such that for all $x,y \in \Omega$,
    \begin{equation}\label{eq:loglip_vel}
        |v(x) - v(y)| \le C (\|\rho\|_{L^1\cap L^\infty} + \|\mu\|_{L^1\cap L^\infty}) |x-y| (1+ \indic_{|x-y|<1/e}|\log |x-y||).
    \end{equation}
\end{lemma}

We now combine this logarithmic modulus with the modulated-energy identity to obtain the following Osgood stability estimate.

\begin{prop}[Stability]\label{prop:stability}
    Let $\rho_1,\rho_2$ be two global weak solutions to \eqref{eq:PDE}, with the same target $\mu\in\P(\Omega)\cap L^\infty(\Omega)$ and initial data $\rho_{1,0},\rho_{2,0}\in\P(\Omega)\cap L^\infty(\Omega)$, and suppose that both solutions are bounded on $[0,T]\times\Omega$. On $\R^\d$, assume in addition that
    $F(0):=\mmd^2(\rho_{1,0},\rho_{2,0})<\infty$; this assumption is automatic on $\T^\d$ and when the initial data agree. There exist constants $C_T<\infty$ and $\Lambda_T\ge1$, depending only on $T$, $\d$, $\|\rho_1\|_{L^\infty([0,T]\times\Omega)}$, $\|\rho_2\|_{L^\infty([0,T]\times\Omega)}$, $\|\mu\|_{L^\infty}$, and $F(0)$, such that, with
    \begin{equation}
        F(t):=\mmd^2(\rho_1(t),\rho_2(t)),\qquad Y(t):=\frac{F(t)}{\Lambda_T},
    \end{equation}
    one has $0\le Y(t)\le e^{-1}$ and
    \begin{equation}\label{eq:stability_est}
        Y(t)\le e^{1-e^{-C_Tt}}Y(0)^{e^{-C_Tt}},\qquad 0\le t\le T,
    \end{equation}
    with the right-hand side interpreted as zero when $Y(0)=0$. In particular, bounded global weak solutions with the same initial data are unique.
\end{prop}

\begin{proof}
Set $f:=\rho_1-\rho_2$, $E:=-\nabla\g\ast f$, and $v_i:=-\nabla\g\ast(\rho_i-\mu)$. Since $\rho_i$ and $v_i$ are bounded on $[0,T]$, the fluxes $\rho_iv_i$ belong to $L^\infty([0,T];L^2)$, and the continuity equations imply that $f$ is absolutely continuous with values in $\dot H^{-1}$. The Hilbert-space chain rule therefore justifies, for almost every $t$,
\begin{align}
F'(t)
&=-\int_\Omega |E|^2\rho_1\,dx
 +\int_\Omega v_2\cdot\nabla\g\ast f\,df \notag\\
&\le \frac12\iint_{\Omega^2}(v_2(x)-v_2(y))\cdot\nabla\g(x-y)\,df(x)\,df(y).\label{eq:stability-commutator}
\end{align}
For $\d=1$, $E'=f$, and the second term in \eqref{eq:stability-commutator} is
\begin{equation}
-\int_\Omega v_2E f\,dx=\frac12\int_\Omega (\partial_xv_2)E^2\,dx,
\end{equation}
so its absolute value is bounded by $C F(t)$. We henceforth assume $\d\ge2$.

Let $\chi_\varepsilon$ be a standard mollifier, with $0<\varepsilon<e^{-1}$. By \cref{lem:loglip},
\begin{equation}\label{eq:mollified-loglip}
\|\nabla(v_2\ast\chi_\varepsilon)\|_\infty
\le C_T(1+|\log\varepsilon|),\qquad
\|v_2-v_2\ast\chi_\varepsilon\|_\infty
\le C_T\varepsilon(1+|\log\varepsilon|).
\end{equation}
At each time, $f\in L^1\cap L^\infty\cap\dot H^{-1}$.
Therefore, the extension-by-density formulation of the continuum Coulomb commutator estimate of
\cite[Proposition~3.2, (3.18), and Remark~3.3]{RS25commutator} may be applied on $\R^\d$ to obtain
\begin{equation}\label{eq:RS-commutator-use}
\left|\iint\big((v_2\ast\chi_\varepsilon)(x)-(v_2\ast\chi_\varepsilon)(y)\big)
\cdot\nabla\g(x-y)\,df(x)\,df(y)\right|
\le C\|\nabla(v_2\ast\chi_\varepsilon)\|_\infty F(t).
\end{equation}
The same estimate on the torus follows from the periodic stress-energy construction in
\cite[Section~6.2]{CdCRS25}. The remaining term is bounded by
\begin{equation}
\|v_2-v_2\ast\chi_\varepsilon\|_\infty
\int_\Omega |E|\,d|f|
\le C_T\varepsilon(1+|\log\varepsilon|),
\end{equation}
because $f$ is uniformly bounded in $L^1\cap L^\infty$ and the Coulomb field of such a density is uniformly bounded. Combining these estimates with \eqref{eq:stability-commutator} gives
\begin{equation}\label{eq:osgood-pre-normalized}
F'(t)\le C_TF(t)(1+|\log\varepsilon|)
+C_T\varepsilon(1+|\log\varepsilon|)
\end{equation}
for almost every $t\in[0,T]$.

Taking first $\varepsilon=e^{-1}$ in \eqref{eq:osgood-pre-normalized} gives
$F(t)\le(F(0)+1)e^{C_Tt}$. Increase $C_T$ if necessary and set
\begin{equation}
\Lambda_T:=e(F(0)+1)e^{C_TT}.
\end{equation}
Then $Y=F/\Lambda_T\le e^{-1}$. At times for which $Y(t)>0$, choose $\varepsilon=Y(t)$ in \eqref{eq:osgood-pre-normalized}; after division by $\Lambda_T$ and enlargement of $C_T$, we obtain
\begin{equation}\label{eq:osgood-normalized}
Y'(t)\le C_TY(t)(1-\log Y(t)).
\end{equation}
For $Z(t):=1-\log Y(t)$ this means $Z'(t)\ge-C_TZ(t)$. Hence
\begin{equation}
1-\log Y(t)\ge e^{-C_Tt}(1-\log Y(0)),
\end{equation}
which is equivalent to \eqref{eq:stability_est}. The case $Y(0)=0$ follows by applying the same estimate to $Y(0)+\eta$ and sending $\eta\downarrow0$. This proves uniqueness.
\end{proof}

\subsection{Regularity} 

It is known that \eqref{eq:PDE} propagates the H\"older regularity of the initial datum when the target is also H\"older continuous (see \cite{BV25}). A natural question is whether the corresponding H\"older seminorms can nevertheless remain uniformly controlled in time. The following explicit counterexample gives a negative answer: the density remains zero at the origin while attaining an order-one value at an exponentially small radius, forcing exponential growth of its H\"older seminorm.

Let $\omega_\d$ denote the volume of the unit ball in $\R^\d$, and $\sigma_\d = \d\omega_\d$ be the surface area of the unit sphere.

\begin{prop}\label{prop:regularitycounterex}
Fix $\d\ge1$. There exist radial probability densities $\mu,\rho_0\in C_c^\infty(\R^\d)\subset \mathcal S(\R^\d)$
such that the radial solution $\rho(t)$ of \eqref{eq:PDE} satisfies
\begin{equation}
\sup_{t\ge 0} \|\rho(t)\|_{H^s(\R^\d)} = \infty
\qquad\text{for every } s>\frac \d2.
\end{equation}
More precisely, with $\lambda:=\mu(0)>0$, there exists $t_0<\infty$ depending on the data such that, for every $\beta\in(0,1]$, there exists $c_\beta>0$ for which
\begin{equation}
[\rho(t)]_{C^{0,\beta}(\R^\d)} \ge c_\beta\, e^{\beta\lambda\frac{\d+2}{2\d}t}
\qquad\text{for all } t\ge t_0.
\end{equation}
\end{prop}

\begin{proof}
\ul{Step 1: choice of data.}
Choose radial cutoffs $\psi,\varphi\in C_c^\infty([0,\infty))$ such that
\begin{equation}
\psi\equiv 1 \text{ on } [0,1],\quad \operatorname{supp}\psi\subset [0,2],
\qquad
\varphi\equiv 1 \text{ on } [0,\tfrac12],\quad \operatorname{supp}\varphi\subset [0,1].
\end{equation}
Set
\begin{equation}
\mu(x):=c_\mu\,\psi(|x|),
\qquad
\rho_0(x):=c_0 |x|^2\varphi(|x|),
\end{equation}
where $c_\mu,c_0>0$ are chosen so that $\int_{\R^\d}\mu=\int_{\R^\d}\rho_0=1$.
Let $\lambda := c_\mu = \mu(0)$ and $\alpha := c_0$.
Then
\begin{equation}\label{eq:data-flat}
\mu(r)=\lambda \quad (0\le r\le 1),
\qquad
\rho_0(r)=\alpha r^2 \quad (0\le r\le \tfrac12).
\end{equation}

\medskip
\noindent
\ul{Step 2: radial reduction.}
For a radial density $\rho(t,r)$ define its cumulative mass
\begin{equation}
m_\rho(t,r):=\int_{B_r} \rho(t,x)\,dx,
\qquad
m_\mu(r):=\int_{B_r} \mu(x)\,dx.
\end{equation}
Since radiality is preserved by \eqref{eq:PDE}, we may write the velocity field as $v(t,x)=u(t,|x|)\frac{x}{|x|}$.
Using $\div v(t) = \rho(t)-\mu$ and the divergence theorem, we get for every $r>0$,
\begin{equation}\label{eq:gauss}
\sigma_\d r^{\d-1} u(t,r)
= \int_{\partial B_r} v(t)\cdot n\,dS
= \int_{B_r} (\rho(t)-\mu)\,dx
= m_\rho(t,r)-m_\mu(r).
\end{equation}
Hence
\begin{equation}\label{eq:u}
u(t,r)=\frac{m_\rho(t,r)-m_\mu(r)}{\sigma_\d r^{\d-1}}.
\end{equation}
Differentiating $m_\rho$ in time and using the continuity equation yields
\begin{align}
    \p_t m_\rho(t,r)
    &= \int_{B_r} \p_t \rho(t,x)\,dx
    = -\int_{\partial B_r} \rho(t) v(t)\cdot n\,dS
    = -\sigma_\d r^{\d-1}\rho(t,r)u(t,r)
    \notag\\
    &= -u(t,r)\,\p_r m_\rho(t,r). \label{eq:m-pde}
\end{align}
Thus, $m_\rho$ satisfies the scalar transport equation
\begin{equation}\label{eq:m-transport}
    \p_t m_\rho + u\,\p_r m_\rho = 0.
\end{equation}

Let $R_t(a)$ be the mass-shell trajectory determined by
\begin{equation}\label{eq:R-ode}
\dot R_t(a)=u(t,R_t(a)),
\qquad
R_0(a)=m_\rho(0,\cdot)^{-1}(a),
\qquad a\in(0,1).
\end{equation}
Then \eqref{eq:m-transport} gives $\frac{d}{dt} m_\rho(t,R_t(a)) = 0$, so $m_\rho(t,R_t(a)) = a$.
Combining \eqref{eq:u} and this conservation of mass, we obtain the closed shell ODE
\begin{equation}\label{eq:R-closed}
\dot R_t(a)=\frac{a-m_\mu(R_t(a))}{\sigma_\d R_t(a)^{\d-1}}.
\end{equation}
Now define $U_t(a):=R_t(a)^\d$.
Then \eqref{eq:R-closed} becomes
\begin{equation}\label{eq:U-general}
\dot U_t(a) = \frac{\d}{\sigma_\d}\bigl(a-m_\mu(U_t(a)^{1/\d})\bigr).
\end{equation}
Since $\sigma_\d = \d\omega_\d$, the factor $\d/\sigma_\d$ is just $1/\omega_\d$. 
Differentiating $m_\rho(t,R_t(a)) = a$ in $a$ gives $1 = \p_r m_\rho(t,R_t(a))\,\p_a R_t(a) = \sigma_\d R_t(a)^{\d-1}\rho(t,R_t(a))\,\p_a R_t(a)$.
Since $\p_a U_t(a) = \d R_t(a)^{\d-1}\p_a R_t(a)$, this yields the exact identity
\begin{equation}\label{eq:density-U}
\rho(t,R_t(a)) = \frac{1}{\omega_\d\,\p_a U_t(a)}.
\end{equation}

\medskip
\noindent
\ul{Step 3: explicit solution for small shells.}
From \eqref{eq:data-flat}, for $0\le r\le 1$ we have $m_\mu(r)=\lambda\omega_\d r^\d$.
Hence, as long as $R_t(a)\le 1$,
\begin{equation}\label{eq:U-linear}
\dot U_t(a)=\frac{a}{\omega_\d}-\lambda U_t(a).
\end{equation}
On the other hand, by \eqref{eq:data-flat}, for $0\le r\le \tfrac12$,
\begin{equation}
a = m_\rho(0,r)=\alpha\sigma_\d\int_0^r s^{\d+1}\,ds = \frac{\alpha\sigma_\d}{\d+2} r^{\d+2}.
\end{equation}
Therefore, for shells initially lying in $[0,\tfrac12]$,
\begin{equation}\label{eq:U0}
U_0(a)=R_0(a)^\d = \Gamma a^\theta,
\qquad
\theta:=\frac{\d}{\d+2},
\qquad
\Gamma:=\Bigl(\frac{\d+2}{\alpha\sigma_\d}\Bigr)^\theta.
\end{equation}
Choose $a_* := \min\Bigl\{ m_\rho(0,\tfrac12),\, \omega_\d\lambda 2^{-\d}\Bigr\} >0$.
Solving \eqref{eq:U-linear} with initial datum \eqref{eq:U0} gives for all $a \in (0, a_*]$,
\begin{equation}\label{eq:U-explicit}
U_t(a) = e^{-\lambda t}\Gamma a^\theta + (1-e^{-\lambda t})\frac{a}{\omega_\d\lambda}.
\end{equation}
Differentiating in $a$ gives
\begin{equation}\label{eq:dUda}
\p_a U_t(a) = \frac{1-e^{-\lambda t}}{\omega_\d\lambda} + e^{-\lambda t}\Gamma\theta a^{-\frac{2}{\d+2}}.
\end{equation}
For each fixed $t\ge0$, \eqref{eq:U-explicit} gives $U_t(a)\to0$ and hence $R_t(a)\to0$ as $a\downarrow0$, while \eqref{eq:dUda} and \eqref{eq:density-U} give $\rho(t,R_t(a))=(\omega_\d\partial_aU_t(a))^{-1}\to0$. Since the smooth radial solution is continuous at the origin, it follows that
\begin{equation}\label{eq:rho-origin}
    \rho(t,0)=0 \quad \text{for every } t\ge 0.
\end{equation}

\medskip
\noindent
\ul{Step 4: a shell with order-one density at exponentially small radius.}
Define
\begin{equation}
a_t := (\Gamma\theta\omega_\d\lambda)^{\frac{\d+2}{2}} e^{-\lambda\frac{\d+2}{2}t}.
\end{equation}
For large $t$, $a_t \le a_*$, and substitution in \eqref{eq:dUda} gives
\begin{equation}
\p_a U_t(a_t)=\frac{2-e^{-\lambda t}}{\omega_\d\lambda}.
\end{equation}
Using \eqref{eq:density-U},
\begin{equation}\label{eq:rho-lower}
\rho(t,R_t(a_t)) = \frac{\lambda}{2-e^{-\lambda t}} \ge \frac\lambda2
\qquad\text{for all } t\ge t_0.
\end{equation}
Now set $r_t := R_t(a_t)$. Since $U_t(a_t) = r_t^\d$, using \eqref{eq:U-explicit} we find $r_t \asymp e^{-\lambda\frac{\d+2}{2\d}t}$.

\medskip
\noindent
\ul{Step 5: blow-up of H\"older and Sobolev norms.}
Fix any $\beta\in(0,1]$. By \eqref{eq:rho-origin}, \eqref{eq:rho-lower}, and the scaling of $r_t$,
\begin{equation}\label{eq:holder-blowup}
[\rho(t)]_{C^{0,\beta}(\R^\d)}
\ge \frac{|\rho(t,r_t)-\rho(t,0)|}{r_t^\beta}
\ge \frac{\lambda/2}{r_t^\beta}
\ge c_\beta e^{\beta\lambda\frac{\d+2}{2\d}t}
\qquad (t\ge t_0).
\end{equation}
For any $s>\d/2$, choosing $\beta\in(0,\min\{1,s-\d/2\})$ and using $H^s(\R^\d)\hookrightarrow C^{0,\beta}(\R^\d)$, we conclude that $\|\rho(t)\|_{H^s(\R^\d)} \to \infty$ as $t\to\infty$.
\end{proof}

\begin{remark}
The same computation shows that $\rho(t,r)\sim \alpha e^{\lambda\frac{\d+2}{\d}t} r^2$ as $r\downarrow 0$, meaning the quadratic vanishing at the origin persists but its coefficient grows exponentially.
\end{remark}

Together with the existence and uniqueness results above, \cref{prop:regularitycounterex} completes the proof of \cref{thm:Cauchy}.

\section{Convergence and PL coercivity on \texorpdfstring{$\T^\d$}{the torus}}\label{sec:torus-coercivity}

On the torus $\T^\d$, we first prove \cref{thm:bounded-contrast-PL,cor:bounded-contrast-integral-PL}: an intrinsic metric PL inequality for finite-energy sources under a target density-ratio condition, including the uniform target as a special case, and its integral-dissipation form for sources with a $C^1$ Coulomb potential. We next prove \cref{thm2}: if
\begin{equation}
\lambda:=\operatorname*{ess\,inf}_{\T^\d}\mu>0,
\end{equation}
then, after any positive waiting time, $\mmd^2(\rho_t,\mu)$ decays exponentially for arbitrary probability-measure initial data, without any lower bound on the evolving density and without imposing the target density-ratio condition.

We finally prove \cref{prop:counterexvanishingmu,prop:counterexamplePLinfmu}, which exhibit two distinct limitations of PL coercivity on the torus. The first gives a fixed target, vanishing at a single point, for which no global PL inequality holds; here global refers to an inequality valid for every admissible source $\rho$, without a local condition relating $\rho$ to $\mu$. The second shows that, for every $\d\ge2$ and $0<\lambda<1$, no PL constant depending only on $(\d,\lambda)$ can hold uniformly over all target densities satisfying $\mu\ge\lambda$. The target density-ratio result is consistent with this second obstruction because its constant also depends on an upper bound for the target.

These periodic obstructions arise from target degeneracy or lack of uniform control over the target class. They are distinct from the spatial-infinity mechanism on $\R^\d$ developed in \cref{sec:whole-space-obstructions}, where the target is fixed and the source is placed far away.

\subsection{PL coercivity near the uniform target}

We begin with the uniform target. The point is that the metric PL constant below is valid for every finite-energy source, while its integral-dissipation form is independent of any positive lower bound or upper bound for the source density whenever the source potential is $C^1$.

The proof first establishes the integral estimate for bounded sources and the uniform target. In dimension one, this follows directly from periodic integration by parts. In higher dimensions, we rewrite the dissipation as a Dirichlet energy minus a cubic correction. If the required bound on this correction failed, a small two-sided barrier would produce a smooth positive maximizer of a penalized variational problem. Evaluating its Euler equation at a point of minimum density and using the Hessian trace inequality yields a contradiction. Heat regularization removes the auxiliary smoothness and positivity assumptions, and a final normalization reduces the general target to the uniform case. Starting the Cauchy flow from a general finite-energy source then gives the metric-slope inequality, and a first-order expansion of a $C^1$ Coulomb potential recovers the integral form.

\begin{proof}[Proof of \cref{thm:bounded-contrast-PL,cor:bounded-contrast-integral-PL}]
We first establish the dissipation-to-energy comparison in \eqref{eq:bounded-contrast-PL} when $\rho,\mu\in\P(\T^\d)\cap L^\infty(\T^\d)$. This bounded-source estimate will then serve as the positive-time input for the metric-slope argument.
If $\d=1$, set $E=-\partial_x\g\ast(\rho-1)$. Then $E'=\rho-1$, and periodic integration by parts gives
\begin{equation}
\dissip(\rho\mid1)
=\int_{\T}E^2(1+E')\,dx
=\int_{\T}E^2\,dx
=2\mmd^2(\rho,1).
\end{equation}
We may therefore assume that $\d\ge2$.

Fix $p>2\d$. This choice ensures both $W^{2,p}(\T^\d)\Subset C^1(\T^\d)$, as needed for compactness, and $W^{2,p/2}(\T^\d)\hookrightarrow C^{1,\gamma}(\T^\d)$ for some $\gamma>0$, as needed to bootstrap the Euler equation below. For every nonnegative $\rho\in L^p(\T^\d)$ with $\int_{\T^\d}\rho\,dx=1$, let $h_\rho$ be the mean-zero solution of
\begin{equation}
-\Delta h_\rho=\rho-1,
\end{equation}
and set
\begin{equation}
\Ec(\rho):=\int_{\T^\d}|\nabla h_\rho|^2\,dx =2\mmd^2(\rho,1),
\qquad
\mathcal C(\rho):=\int_{\T^\d}(1-\rho)|\nabla h_\rho|^2\,dx.
\end{equation}
Since
\begin{equation}\label{eq:uniform-cubic-decomposition}
\dissip(\rho\mid1)=\Ec(\rho)-\mathcal C(\rho),
\end{equation}
it suffices to prove
\begin{equation}\label{eq:uniform-cubic-target}
\mathcal C(\rho)\le \alpha_\d\Ec(\rho),
\qquad
\alpha_\d:=\frac{\d-1}{\d}.
\end{equation}

There is a useful formal extremal heuristic behind \eqref{eq:uniform-cubic-target}. For nonconstant $\rho$,
\begin{equation}
\frac{\dissip(\rho\mid1)}{\Ec(\rho)}
=
\frac{\int_{\T^\d}\rho|\nabla h_\rho|^2\,dx}
{\int_{\T^\d}|\nabla h_\rho|^2\,dx},
\end{equation}
so the estimate asserts that the density-weighted field energy retains at least a fraction $1/\d$ of the full Dirichlet energy. Although such a conclusion would be false for an arbitrary pair $(\rho,h)$, here the two are coupled by $\rho=1-\Delta h$; in particular, $\rho\ge0$ gives the one-sided curvature constraint $\Delta h\le1$. Formally, if the quotient $\mathcal C/\Ec$ admitted a smooth positive nonconstant maximizer with value $\ell$, its Euler equation would give
\begin{equation}
|D^2h|^2-(\Delta h)^2+\ell\Delta h=0.
\end{equation}
At a minimum point $x_0$ of $\rho$, setting $\tau:=\Delta h(x_0)=1-\rho(x_0)\in(0,1)$ and using $|D^2h|^2\ge(\Delta h)^2/\d$ would yield
\begin{equation}
0\ge\tau\bigl(\ell-\alpha_\d\tau\bigr),
\end{equation}
and hence $\ell<\alpha_\d$. Thus $\alpha_\d=1-1/\d$ is precisely the constant furnished by the Hessian trace inequality together with the constraint $\Delta h\le1$. The penalization below makes this formal ``worst offender'' argument rigorous: it produces a smooth interior maximizer, while the barrier contributes with the favorable sign at a point of minimum density.

We first prove \eqref{eq:uniform-cubic-target} for smooth strictly positive densities. Suppose, to the contrary, that such a density $\rho^\circ$ satisfies
\begin{equation}
\mathcal C(\rho^\circ)>\alpha_\d\Ec(\rho^\circ).
\end{equation}
The potential is nonconstant, and strict positivity gives
\begin{equation}
\dissip(\rho^\circ\mid1)
\ge \bigl(\min_{\T^\d}\rho^\circ\bigr)\Ec(\rho^\circ)>0.
\end{equation}
Consequently, by \eqref{eq:uniform-cubic-decomposition}, we may choose
\begin{equation}\label{eq:uniform-ell-choice}
\alpha_\d<\ell<\frac{\mathcal C(\rho^\circ)}{\Ec(\rho^\circ)}<1.
\end{equation}

To obtain a smooth interior maximizer, we introduce a two-sided barrier whose $s^p$ term controls concentration and whose $s^{-1}$ term penalizes approach to zero. For $s>0$, define
\begin{equation}\label{eq:uniform-barrier}
\Phi(s):=s^p-1-p(s-1)+s^{-1}-1+(s-1).
\end{equation}
The affine corrections subtract the tangent lines of $s^p$ and $s^{-1}$ at $s=1$, thereby normalizing the barrier at the uniform density. We extend $\Phi$ by $+\infty$ on $(-\infty,0]$. Then $\Phi\ge0$,
\begin{equation}\label{eq:uniform-barrier-properties}
\Phi(1)=\Phi'(1)=0,
\qquad
\Phi''(s)=p(p-1)s^{p-2}+2s^{-3}>0,
\end{equation}
and $\Phi':(0,\infty)\to\R$ is a smooth increasing bijection. Introduce
\begin{equation}
\mathcal X_p:=
\left\{
\rho\in L^p(\T^\d):
\rho>0\ \text{a.e.},\ 
\int_{\T^\d}\rho\,dx=1,\ 
\int_{\T^\d}\rho^{-1}\,dx<\infty
\right\}.
\end{equation}
For $\rho\in\mathcal X_p$, put
\begin{equation}
\mathcal P(\rho):=\int_{\T^\d}\Phi(\rho)\,dx.
\end{equation}
For $\varepsilon>0$, consider
\begin{equation}\label{eq:uniform-penalized-functional}
\mathcal F_{\varepsilon,\ell}(\rho)
:=\mathcal C(\rho)-\ell\Ec(\rho)-\varepsilon\mathcal P(\rho).
\end{equation}
Choose $\varepsilon$ sufficiently small that
$\mathcal F_{\varepsilon,\ell}(\rho^\circ)>0$.

We claim that \eqref{eq:uniform-penalized-functional} has a nonconstant maximizer in $\mathcal X_p$. On the mass-one class, the linear terms in \eqref{eq:uniform-barrier} cancel and
\begin{equation}\label{eq:uniform-penalty-identity}
\mathcal P(\rho)=\int_{\T^\d}\bigl(\rho^p+\rho^{-1}-2\bigr)\,dx.
\end{equation}
Jensen's inequality applied to $s\mapsto s^{-1}$ gives
\begin{equation}
\|\rho\|_{L^p}^p\le \mathcal P(\rho)+1.
\end{equation}
Moreover, Plancherel's theorem and H\"older's inequality imply
\begin{equation}\label{eq:uniform-energy-penalty-control}
\Ec(\rho)=\|\rho-1\|_{\dot H^{-1}}^2\le C_\d\|\rho-1\|_{L^2}^2\le C_{\d,p}\left(1+(1+\mathcal P(\rho))^{2/p}\right).
\end{equation}
Since $\mathcal C(\rho)\le\Ec(\rho)$ and $2/p<1$, it follows that
\begin{equation}
\mathcal F_{\varepsilon,\ell}(\rho)
\le C_{\d,p,\ell}\left(1+(1+\mathcal P(\rho))^{2/p}\right)
-\varepsilon\mathcal P(\rho),
\end{equation}
whose right-hand side tends to $-\infty$ as $\mathcal P(\rho)\to\infty$.

Let $(\rho_n)$ be a maximizing sequence. Passing to a subsequence, we may assume that
$\rho_n\rightharpoonup\bar\rho$ weakly in $L^p$. The mass constraint passes to the limit, and weak lower semicontinuity of the convex integral functional $\mathcal P$ gives
\begin{equation}
\mathcal P(\bar\rho)\le\liminf_{n\to\infty}\mathcal P(\rho_n)<\infty.
\end{equation}
Thus $\bar\rho\in\mathcal X_p$. The elliptic estimate for $h_{\rho_n}$ and the compact embedding $W^{2,p}(\T^\d)\Subset C^1(\T^\d)$ give
\begin{equation}
h_{\rho_n}\longrightarrow h_{\bar\rho}
\quad\text{strongly in }C^1(\T^\d).
\end{equation}
It follows that $\Ec(\rho_n)\to\Ec(\bar\rho)$ and, using the weak convergence of $\rho_n$ together with uniform convergence of $|\nabla h_{\rho_n}|^2$, that
$\mathcal C(\rho_n)\to\mathcal C(\bar\rho)$. Hence $\bar\rho$ is a maximizer for \(\mathcal F_{\varepsilon,\ell}\). Since $\rho\equiv1$ gives value zero while the supremum is positive, every maximizer is nonconstant.

Fix such a maximizer $\bar\rho$, abbreviate $h:=h_{\bar\rho}$, and set
\begin{equation}\label{eq:uniform-euler-source}
A:=|D^2h|^2-(\Delta h)^2+\ell\Delta h.
\end{equation}
The identity
\begin{equation}
\int_{\T^\d}|D^2h|^2\,dx
=\int_{\T^\d}(\Delta h)^2\,dx
\end{equation}
and $\int_{\T^\d}\Delta h\,dx=0$ imply $\int_{\T^\d}A\,dx=0$. Let $w$ be the mean-zero solution of
\begin{equation}
-\Delta w=A.
\end{equation}

For $\zeta\in C^\infty(\T^\d)$, use the mass-preserving multiplicative variation
\begin{equation}
\rho_\tau:=\frac{\bar\rho e^{\tau\zeta}}{\int_{\T^\d}\bar\rho e^{\tau\zeta}\,dx},
\qquad
\left.\frac{d}{d\tau}\right|_{\tau=0}\rho_\tau
=\bar\rho\left(\zeta-\int_{\T^\d}\bar\rho\zeta\,dx\right)=:\sigma.
\end{equation}
If $v$ is the mean-zero solution of $-\Delta v=\sigma$, direct differentiation gives
\begin{align}
\delta\Ec[h](v)&=-2\int_{\T^\d}(\Delta h)v\,dx,\label{eq:uniform-variation-energy}\\
\delta\mathcal C[h](v)&=2\int_{\T^\d}v\bigl(|D^2h|^2-(\Delta h)^2\bigr)\,dx.\label{eq:uniform-variation-cubic}
\end{align}
These identities, initially obtained for smooth $h$ and $v$, extend to the present $W^{2,p}$ functions by smooth approximation and continuity of the displayed terms. Since
\begin{equation}
|s\Phi'(s)|\le C_p(1+s^p+s^{-1}),
\end{equation}
the barrier is differentiable along the multiplicative variation. The first variation of \eqref{eq:uniform-penalized-functional} is therefore
\begin{equation}
0=\int_{\T^\d}\bigl(2w-\varepsilon\Phi'(\bar\rho)\bigr)
\bar\rho\left(\zeta-\int_{\T^\d}\bar\rho\zeta\,dx\right)dx.
\end{equation}
Since $\bar\rho\,dx$ is a probability measure, set
\begin{equation}
\kappa:=\int_{\T^\d}\bigl(2w-\varepsilon\Phi'(\bar\rho)\bigr)\bar\rho\,dx.
\end{equation}
Expanding the centered test function in the preceding identity gives
\begin{equation}
\int_{\T^\d}\bigl(2w-\varepsilon\Phi'(\bar\rho)-\kappa\bigr)\zeta\bar\rho\,dx=0
\qquad\text{for every }\zeta\in C^\infty(\T^\d).
\end{equation}
It follows that the finite signed measure with density $\bigl(2w-\varepsilon\Phi'(\bar\rho)-\kappa\bigr)\bar\rho$ must vanish. Hence
\begin{equation}\label{eq:uniform-euler-nonlocal}
2w-\varepsilon\Phi'(\bar\rho)=\kappa
\qquad \bar\rho\,dx\text{-a.e.}
\end{equation}
Because $\bar\rho>0$ almost everywhere, the equality holds Lebesgue-almost everywhere.

To justify applying $-\Delta$ to \eqref{eq:uniform-euler-nonlocal} and evaluating the resulting pointwise equation at a minimum of $\bar\rho$, we first show that the a priori $L^p$ maximizer is smooth and bounded away from zero. Initially $D^2h,\Delta h\in L^p$, so $A\in L^{p/2}$. Since $p/2>\d$,
\begin{equation}
w\in W^{2,p/2}(\T^\d)\hookrightarrow C^{1,\gamma}(\T^\d)
\end{equation}
for some $\gamma>0$. As $\Phi'$ is a smooth bijection from $(0,\infty)$ onto $\R$, equation \eqref{eq:uniform-euler-nonlocal} implies that $\bar\rho$ has a $C^{1,\gamma}$ representative taking values in a compact subinterval of $(0,\infty)$. Periodic Schauder estimates then bootstrap $h,w,$ and $\bar\rho$ to $C^\infty$. Applying $-\Delta$ to \eqref{eq:uniform-euler-nonlocal} gives the pointwise Euler equation
\begin{equation}\label{eq:uniform-euler-local}
2\left(|D^2h|^2-(\Delta h)^2+\ell\Delta h\right)
+\varepsilon\Delta\Phi'(\bar\rho)=0.
\end{equation}

Since the maximizing density is nonconstant, positive, and has mean one, it has a minimum point $x_0$ with $0<\bar\rho(x_0)<1$. Set $\tau:=1-\bar\rho(x_0)=\Delta h(x_0)\in(0,1)$.
One has $\nabla\bar\rho(x_0)=0$ and $\Delta\bar\rho(x_0)\ge0$, whence
\begin{equation}
\Delta\Phi'(\bar\rho)(x_0)
=\Phi''(\bar\rho(x_0))\Delta\bar\rho(x_0)\ge0.
\end{equation}
Equation \eqref{eq:uniform-euler-local} therefore yields
\begin{equation}\label{eq:uniform-minimum-euler}
|D^2h(x_0)|^2-\tau^2+\ell\tau\le0.
\end{equation}
On the other hand, the matrix trace inequality gives
\begin{equation}
|D^2h(x_0)|^2\ge\frac1\d(\Delta h(x_0))^2=\frac{\tau^2}{\d}.
\end{equation}
Combining this with \eqref{eq:uniform-minimum-euler},
\begin{equation}
0\ge \tau\bigl(\ell-\alpha_\d\tau\bigr)>0,
\end{equation}
where the strict inequality follows from $\tau<1$ and $\ell>\alpha_\d$. This contradiction proves \eqref{eq:uniform-cubic-target} for smooth strictly positive densities.

Now, let $\rho\in\P(\T^\d)\cap L^\infty(\T^\d)$ be nonnegative, set $h:=h_\rho$, and let $P_s=e^{s\Delta}$ be the periodic heat semigroup. Set
\begin{equation}
\rho_s=P_s\rho,
\qquad
h_s=P_sh.
\end{equation}
For $s>0$, the density $\rho_s$ is smooth and strictly positive, has mass one, and satisfies $-\Delta h_s=\rho_s-1$. Hence the dissipation-to-energy comparison in \eqref{eq:uniform-target-PL} holds for $(\rho_s,h_s)$. Periodic elliptic regularity gives $h\in W^{2,q}$ for every finite $q$; choosing $q>\d$ yields
\begin{equation}
\nabla h_s\longrightarrow\nabla h\quad\text{uniformly},
\qquad
\rho_s\longrightarrow\rho\quad\text{in }L^1.
\end{equation}
The dissipation and energy terms in \eqref{eq:uniform-target-PL} therefore converge to the corresponding quantities for $(\rho,h)$. This proves the bounded-source uniform-target estimate.

We now return to arbitrary $\d\ge1$ and derive the general estimate from this special case.
Let $h=\g\ast(\rho-\mu)$, so that $\Delta h=\mu-\rho$. If $h$ is nonconstant, then $N:=\operatorname*{ess\,sup}(\mu-\rho)>0$, and
\begin{equation}
\widehat\rho:=1-\frac{\mu-\rho}{N},
\qquad
\widehat h:=\frac{h}{N}
\end{equation}
satisfy $\widehat\rho\in\P(\T^\d)\cap L^\infty(\T^\d)$, $\widehat\rho\ge0$, and $-\Delta\widehat h=\widehat\rho-1$. Applying the uniform-target estimate just proved and rescaling gives
\begin{equation}\label{eq:upper-laplacian-scalar}
\int_{\T^\d}(\mu-\rho)|\nabla h|^2\,dx
\le\frac{\d-1}{\d}N\int_{\T^\d}|\nabla h|^2\,dx
\le\frac{\d-1}{\d}M\int_{\T^\d}|\nabla h|^2\,dx.
\end{equation}
Here the last inequality follows from $\mu-\rho\le\mu$, and hence $N\le M$.
Since $\mu\ge m$,
\begin{equation}
\dissip(\rho\mid\mu)=\int_{\T^\d}\mu|\nabla h|^2\,dx-\int_{\T^\d}(\mu-\rho)|\nabla h|^2\,dx
\ge\left(m-\frac{\d-1}{\d}M\right)\int_{\T^\d}|\nabla h|^2\,dx.
\end{equation}
This proves the dissipation-to-energy comparison in \eqref{eq:bounded-contrast-PL} for bounded sources. The case of constant $h$ is immediate. Finally, $\|\mu-1\|_{L^\infty}\le\varepsilon$ implies $m\ge1-\varepsilon$ and $M\le1+\varepsilon$, which gives the dissipation-to-energy comparison in \eqref{eq:near-uniform-PL}.

We next prove the metric-slope estimate. Let $\rho\in\P(\T^\d)$ satisfy $F_\mu(\rho)<\infty$, and set
\begin{equation}
\kappa:=m-\frac{\d-1}{\d}M.
\end{equation}
The assertion is immediate if $\kappa\le0$, so suppose that $\kappa>0$. Let $(\rho_t)_{t\ge0}$ be a solution furnished by the simultaneous-mollification construction in the proof of \cref{thm:Cauchy}, with initial datum $\rho$. For every $s>0$, the density $\rho_s$ is bounded, and hence the estimate just proved, \eqref{eq:weak-energy-inequality}, and Gr\"onwall's lemma give
\begin{equation}
F_\mu(\rho_t)
\le e^{-2\kappa(t-s)}F_\mu(\rho_s),
\qquad t\ge s>0.
\end{equation}
Sending $s\downarrow0$ and using \eqref{eq:initial-energy-continuity-torus}, we obtain
\begin{equation}\label{eq:bounded-contrast-decay-from-zero}
F_\mu(\rho_t)
\le e^{-2\kappa t}F_\mu(\rho).
\end{equation}
If $F_\mu(\rho)>0$, then \eqref{eq:initial-action-W2} and \eqref{eq:bounded-contrast-decay-from-zero} imply
\begin{equation}
\frac{F_\mu(\rho)-F_\mu(\rho_t)}{W_2(\rho,\rho_t)}
\ge
\left(\frac{F_\mu(\rho)-F_\mu(\rho_t)}{t}\right)^{1/2}
\ge
\left(\frac{1-e^{-2\kappa t}}{t}F_\mu(\rho)\right)^{1/2}.
\end{equation}
Since $\rho_t\to\rho$ in $W_2$ as $t\downarrow0$ (recall the bound \eqref{eq:initial-action-W2}), taking the limit superior in the definition of the descending slope yields
\begin{equation}
|\partial F_\mu|^2(\rho)\ge2\kappa F_\mu(\rho).
\end{equation}
The case $F_\mu(\rho)=0$ is immediate, completing the proof of \eqref{eq:bounded-contrast-slope-PL}.

It remains to relate the metric slope to the integral dissipation under the hypothesis of \cref{cor:bounded-contrast-integral-PL}. Let $\eta$ have finite Coulomb energy. The quadratic energy identity gives
\begin{equation}
F_\mu(\eta)
=F_\mu(\rho)+\int_{\T^\d}h\,d(\eta-\rho)
+\frac12\|\eta-\rho\|_{\dot H^{-1}}^2.
\end{equation}
If $\pi$ is an optimal coupling of $\rho$ and $\eta$, it follows that
\begin{equation}
F_\mu(\rho)-F_\mu(\eta)
\le\int_{\T^\d\times\T^\d}\bigl(h(x)-h(y)\bigr)\,d\pi(x,y).
\end{equation}
Because $h\in C^1(\T^\d)$, let $\omega$ be a modulus of continuity for $\nabla h$, and set $\delta:=W_2(\rho,\eta)$. For $\delta>0$ sufficiently small, set $r_\delta:=\sqrt{\delta}$, which is below the injectivity radius. If $\dist(x,y)\le r_\delta$ and $\gamma_{xy}$ is the constant-speed minimizing geodesic from $x$ to $y$, then
\begin{equation}
 h(x)-h(y)
 =-\left\langle\nabla h(x),\dot\gamma_{xy}(0)\right\rangle+R(x,y),
 \qquad
 |R(x,y)|\le\omega(r_\delta)\dist(x,y).
\end{equation}
Since the first marginal of $\pi$ is $\rho$, Cauchy--Schwarz gives
\begin{equation}
\int_{\{\dist(x,y)\le r_\delta\}}\bigl(h(x)-h(y)\bigr)\,d\pi(x,y)
\le
\left(\int_{\T^\d}|\nabla h|^2\,d\rho\right)^{1/2}\delta
+\omega(r_\delta)\delta.
\end{equation}
On the complementary set, the Lipschitz continuity of $h$ and the optimality of $\pi$ give
\begin{equation}
\int_{\{\dist(x,y)>r_\delta\}}\bigl|h(x)-h(y)\bigr|\,d\pi(x,y)
\le
\|\nabla h\|_{L^\infty}\frac{\delta^2}{r_\delta}.
\end{equation}
Consequently,
\begin{equation}
\int_{\T^\d\times\T^\d}\bigl(h(x)-h(y)\bigr)\,d\pi(x,y)
\le
\left(\int_{\T^\d}|\nabla h|^2\,d\rho\right)^{1/2}\delta
+\left(\omega(\sqrt{\delta})+\|\nabla h\|_{L^\infty}\sqrt{\delta}\right)\delta.
\end{equation}
Since the expression in parentheses tends to zero as $\eta\to\rho$ in $W_2$, this proves
\begin{equation}
\int_{\T^\d\times\T^\d}\bigl(h(x)-h(y)\bigr)\,d\pi(x,y)
\le
\left(\int_{\T^\d}|\nabla h|^2\,d\rho\right)^{1/2}W_2(\rho,\eta)
+o\bigl(W_2(\rho,\eta)\bigr).
\end{equation}
Consequently,
\begin{equation}
|\partial F_\mu|^2(\rho)
\le\int_{\T^\d}|\nabla h|^2\,d\rho
=\dissip(\rho\mid\mu).
\end{equation}
Combining this inequality with \eqref{eq:bounded-contrast-slope-PL} proves \eqref{eq:bounded-contrast-PL}, and the uniform and near-uniform cases follow by the same substitutions as above. Finally, if $\rho\in L^p(\T^\d)$ for some $p>\d$, Sobolev embedding gives $h\in W^{2,p}(\T^\d)\hookrightarrow C^1(\T^\d)$.
\end{proof}

\begin{remark}[A source lower bound]\label{rem:uniform-source-lower-bound}
If $m_\rho:=\operatorname*{ess\,inf}_{\T^\d}\rho$ and $h:=\g\ast(\rho-1)$, then the uniform-target estimate applied to
\begin{equation}
\widetilde\rho:=\frac{\rho-m_\rho}{1-m_\rho},
\qquad
\widetilde h:=\frac{h}{1-m_\rho},
\end{equation}
when $m_\rho<1$, gives the refinement
\begin{equation}
\dissip(\rho\mid1)
\ge2\left(\frac1\d+\frac{\d-1}{\d}m_\rho\right)
\mmd^2(\rho,1).
\end{equation}
The case $m_\rho=1$ is trivial.
\end{remark}

\begin{proof}[Proof of \cref{cor:bounded-contrast-convergence}]
By \cref{thm:Cauchy}, the solution is bounded after every positive waiting time. Combining \cref{cor:bounded-contrast-integral-PL} with the energy inequality \eqref{eq:weak-energy-inequality} and applying the integral form of Gr\"onwall's lemma gives \eqref{eq:bounded-contrast-convergence}.
\end{proof}

\begin{remark}[Dimension one]\label{rem:d1-global-pl}
If $\d=1$ and $E=-\g'\ast(\rho-\mu)$, then $E'=\rho-\mu$ and periodic integration by parts gives
\begin{equation}
\dissip(\rho\mid\mu)
=\int_{\T}E^2\rho\,dx
=\int_{\T}E^2\mu\,dx
\ge2\left(\operatorname*{ess\,inf}_{\T}\mu\right)
\mmd^2(\rho,\mu).
\end{equation}
Thus every uniformly positive target satisfies a global PL inequality in dimension one, without imposing the target density-ratio condition.
If $\rho$ is a solution constructed in \cref{thm:Cauchy} and $m:=\operatorname*{ess\,inf}_{\T}\mu>0$, then the energy inequality gives
\begin{equation}
\mmd^2(\rho_t,\mu)
\le e^{-2m(t-s)}\mmd^2(\rho_s,\mu),
\qquad t\ge s>0.
\end{equation}
\end{remark}

\subsection{Quantitative rates via a \emph{defective} PL inequality}

Because we do not assume or invoke a positive lower bound on $\rho_t$, the usual PL estimate cannot be applied directly. The first lemma below isolates the low-density region as an explicit defect, and the second shows that this defect decays exponentially along the flow. A post-waiting-time bound on the Coulomb field then turns the energy inequality into the inhomogeneous Gr\"onwall estimate used to prove \cref{thm2}.
\begin{lemma}[Defective PL inequality]\label{lem:defectPL}
    Let $\rho, \mu \in \mathcal{P}(\T^\d) \cap L^\infty(\T^\d)$ be such that $\mu \ge \lambda > 0$ almost everywhere. Then, for every $\varepsilon > 0$,
    \begin{equation}
        \dissip(\rho\mid\mu) \ge 2 \varepsilon \lambda \mmd^2(\rho, \mu) - \|\nabla \g \ast (\rho - \mu)\|_{L^{\infty}}^2 \| (\varepsilon \lambda - \rho)_+ \|_{L^1}.
    \end{equation}
\end{lemma}
\begin{proof}
    Define the electric field $E := -\nabla \g \ast (\rho - \mu)$. We split the dissipation into two parts:
    \begin{align}
        \int_{\T^\d} \rho |E|^2 \, dx &= \int_{\{\rho \ge \varepsilon \lambda\}} \rho |E|^2 \, dx + \int_{\{\rho < \varepsilon \lambda\}} \rho |E|^2 \, dx \notag\\
        &\ge \varepsilon \lambda \int_{\T^\d} |E|^2 \, dx - \varepsilon \lambda \int_{\{\rho < \varepsilon \lambda\}} |E|^2 \, dx + \int_{\{\rho < \varepsilon \lambda\}} \rho |E|^2 \, dx \notag\\
        &= \varepsilon \lambda \int_{\T^\d} |E|^2 \, dx - \int_{\{\rho < \varepsilon \lambda\}} (\varepsilon \lambda - \rho) |E|^2 \, dx \notag\\
        &\ge \varepsilon \lambda (2 \mmd^2(\rho, \mu)) - \|E\|_\infty^2 \int_{\T^\d} (\varepsilon \lambda - \rho)_+ \, dx.
    \end{align}
    The result follows.
\end{proof}
We next quantify this defect along the flow.
\begin{lemma}\label{lem:decaydefect}
    Let $\rho_0\in\P(\T^\d)$, let $\mu\in \P(\T^\d)\cap L^\infty(\T^\d)$ satisfy $\operatorname*{ess\,inf}_{\T^\d}\mu =: \la >0$, and let $\rho$ be a global weak solution to \eqref{eq:PDE} as constructed above. Then, for all $\varepsilon\in (0,1)$ and $t>0$,
    \begin{equation}
         \int_{\T^\d} (\varepsilon \la - \rho(t))_+\, dx \le  e^{-\la (1-\varepsilon)t} \varepsilon\la.
    \end{equation}
\end{lemma}
\begin{proof}
    We perform the computation for the smooth approximate solutions. For the simultaneous mollifications used in the construction, $\mu_k=\mu\ast\varphi_k$ still satisfies $\mu_k\ge\lambda$, so the estimate below is uniform in $k$. At every fixed positive time, after extraction, $\rho_k(t)\stackrel{*}{\rightharpoonup}\rho(t)$ in $L^\infty(\T^\d)$. Since $r\mapsto(\varepsilon\lambda-r)_+$ is convex and continuous, the functional $\eta\mapsto\int_{\T^\d}(\varepsilon\lambda-\eta)_+\,dx$ is weak-$*$ lower semicontinuous on $L^\infty(\T^\d)$. It therefore suffices to prove the estimate for the smooth approximate solutions. To justify the chain rule, one may additionally replace the positive-part function by a smooth convex approximation and then pass to the limit.
    
    Let
    $z=(\varepsilon\lambda-\rho)_+$. Since the equation is
    $\partial_t\rho+v\cdot\nabla\rho+\rho\div v=0$, with
    $v=-\nabla\g\ast(\rho-\mu)$ and $\div v=\rho-\mu$, one has on the set $\{z>0\}$
    \begin{equation}
       (\partial_t+v\cdot\nabla)z=\rho(\rho-\mu).
    \end{equation}
    Therefore, for $q\ge1$,
    \begin{align}
    \frac{d}{dt}\int_{\T^\d}z^q\,dx
    &=\int_{\T^\d}z^q\div v\,dx
      +q\int_{\{z>0\}}z^{q-1}\rho(\rho-\mu)\,dx\notag\\
    &=\int_{\T^\d}z^q\left(1+q\frac{\rho}{z}\right)(\rho-\mu)\,dx,\label{eq:dissipPlus}
    \end{align}
    where the quotient is understood only on $\{z>0\}$. On that set, $\rho\le\varepsilon\lambda$ and $\mu\ge\lambda$, hence
    $\rho-\mu\le-(1-\varepsilon)\lambda$. The term containing $q\rho/z$ has the same sign and may be discarded. Thus,
    \begin{equation}
    \frac{d}{dt}\int_{\T^\d}z^q\,dx
    \le-(1-\varepsilon)\lambda\int_{\T^\d}z^q\,dx.
    \end{equation}
    Taking $q=1$, applying Gr\"onwall's lemma between $s$ and $t$, and using
    $\int z(s)\le\varepsilon\lambda$, we find
    \begin{equation}
    \int_{\T^\d}(\varepsilon\lambda-\rho(t))_+\,dx
    \le e^{-(1-\varepsilon)\lambda(t-s)}\varepsilon\lambda.
    \end{equation}
    Sending $s\downarrow0$ proves the claim.
\end{proof}

With the two defect estimates established, we turn to the proof of \cref{thm2}.

\begin{proof}[Proof of \cref{thm2}]
Write $F(t)=\mmd^2(\rho_t,\mu)$ and $E(t)=-\nabla\g\ast(\rho_t-\mu)$. We first bound the field on $[s,\infty)$. If $\d=1$, the derivative of the periodic Green function is bounded, and therefore
\begin{equation}
\|E(t)\|_\infty\le\|\g'\|_\infty\|\rho_t-\mu\|_{L^1}\le2\|\g'\|_\infty.
\end{equation}
If $\d\ge2$, the singular part of the periodic Coulomb kernel satisfies the standard near--far estimate. Since $\|\rho_t-\mu\|_{L^1}\le2$, the dimension-only constant $C_\d$ in the definition of $K_{\d,\mu}(s)$ may be chosen so that
\begin{equation}
\|E(t)\|_\infty^2
\le C_\d\left(
1+2^{1/\d}\|\rho_t-\mu\|_{L^\infty}^{1-1/\d}
\right)^2.
\end{equation}
The ultracontractive estimate in \cref{prop:hyper} and $t\ge s$ then give, in both cases,
\begin{equation}\label{eq:field-bound-Ks}
\|E(t)\|_\infty^2\le K_{\d,\mu}(s).
\end{equation}

On every finite interval $[s,T]$, the positive-time bounds on $\rho$ and $E$ imply
\begin{equation}
\rho E\in L^\infty([s,T];L^2),
\qquad
\rho-\mu\in W^{1,\infty}([s,T];\dot H^{-1}).
\end{equation}
The Hilbert-space chain rule therefore gives $F'=-\dissip(\rho\mid\mu)$ almost everywhere. Combining this identity with \cref{lem:defectPL,lem:decaydefect} and \eqref{eq:field-bound-Ks}, we obtain for almost every $t\ge s$

\begin{equation}\label{eq:inhomogeneous-F}
F'(t)\le-2\varepsilon\lambda F(t)
+\varepsilon\lambda K_{\d,\mu}(s)e^{-(1-\varepsilon)\lambda t}.
\end{equation}
Variation of constants gives
\begin{align}
F(t)
&\le e^{-a_\varepsilon(t-s)}F(s)
+\varepsilon\lambda K_{\d,\mu}(s)
\int_s^t e^{-a_\varepsilon(t-r)}e^{-b_\varepsilon r}\,dr\notag\\
&=e^{-a_\varepsilon(t-s)}F(s)
+\varepsilon\lambda K_{\d,\mu}(s)e^{-b_\varepsilon s}
\Theta_\varepsilon(t-s),
\end{align}
which is \eqref{eq:convergence-exact}. Notice that the resonant value $\varepsilon=1/3$ is exactly the case $a_\varepsilon=b_\varepsilon=2\lambda/3$ and produces the factor $(t-s)e^{-2\lambda(t-s)/3}$.

Finally, let $0<\gamma<2\lambda/3$. One can choose $\varepsilon\in(0,1)$ so that
\begin{equation}
\gamma<\min\{2\varepsilon\lambda,(1-\varepsilon)\lambda\}.
\end{equation}
Both terms in \eqref{eq:convergence-exact} are then bounded by a constant times $e^{-\gamma(t-s)}$, which proves \eqref{eq:convergence-subendpoint}.
\end{proof}

\subsection{\texorpdfstring{Proof of \cref{prop:counterexvanishingmu}}{Proof of the single-point vanishing counterexample}}

We first give a one-dimensional construction. We then record why its transverse lift does not prove the single-point vanishing statement in higher dimensions, before giving a genuinely localized construction.

\begin{lemma}[One-dimensional counterexample]\label{lem:one-dimensional-counterexample}
There exists a Lipschitz probability density $\mu\in\P(\T)$, vanishing only at the origin, and Lipschitz probability densities $(\rho_\delta)_{\delta>0}$ such that $\rho_\delta\to\mu$ almost everywhere and
\begin{equation}
\frac{\mmd^2(\rho_\delta,\mu)}{\dissip(\rho_\delta\mid\mu)}\longrightarrow+\infty
\qquad(\delta\downarrow0).
\end{equation}
\end{lemma}

\begin{proof}
Identify $\T$ with $(-\frac12,\frac12]$ and set $\mu(x)=4|x|$. Then $\mu$ is a Lipschitz probability density and vanishes only at $0$. Choose an even function $\chi\in C_c^\infty((-\frac32,\frac32))$ with $0\le\chi\le1$, $\chi=1$ on $[-1,1]$, and $y\chi'(y)\le0$. For $0<\delta<1/6$ and a fixed $c\in(0,1)$, define in this coordinate chart
\begin{equation}\label{eq:Edelta-one-dimensional}
E_\delta(x):=-2c\,x|x|\,\chi(x/\delta),
\end{equation}
and extend it periodically. Its support is contained strictly inside the fundamental interval, and it is a periodic $C^1$ function of mean zero. Set
\begin{equation}
\rho_\delta:=\mu+E_\delta'.
\end{equation}
Then $\int_\T\rho_\delta=1$. On the support of $E_\delta$,
\begin{equation}
\rho_\delta(x)=4|x|\left(1-c\chi(x/\delta)-\frac c2\frac{x}{\delta}\chi'(x/\delta)\right)
\ge4(1-c)|x|,
\end{equation}
because $y\chi'(y)\le0$; outside this support $\rho_\delta=\mu$. Thus, $\rho_\delta$ is a nonnegative Lipschitz probability density. Since $E_\delta'=\rho_\delta-\mu$ and $E_\delta$ has mean zero, it is precisely the electric field $-\g'\ast(\rho_\delta-\mu)$.

For any one-dimensional periodic pair, integration by parts gives
\begin{equation}\label{eq:one-dimensional-D-identity}
\dissip(\rho\mid\mu)=\int_\T E^2\rho\,dx
=\int_\T E^2E'\,dx+\int_\T E^2\mu\,dx
=\int_\T E^2\mu\,dx.
\end{equation}
Consequently, with positive constants depending only on $c$ and $\chi$,
\begin{align}
\mmd^2(\rho_\delta,\mu)
&=\frac12\int_\T E_\delta^2\,dx
=2c^2\delta^5\int_\R y^4\chi(y)^2\,dy,\\
\dissip(\rho_\delta\mid\mu)
&=\int_\T E_\delta^2\mu\,dx
=16c^2\delta^6\int_\R |y|^5\chi(y)^2\,dy.
\end{align}
Their ratio is a positive constant times $\delta^{-1}$. The almost-everywhere convergence follows because $\rho_\delta=\mu$ outside $(-3\delta/2,3\delta/2)$.
\end{proof}

\begin{remark}[The transverse lift]
Let $\mu_1$ and $\rho_{\delta,1}$ be the preceding one-dimensional densities and define on $\T^\d$
$\mu(x)=\mu_1(x_1)$ and $\rho_\delta(x)=\rho_{\delta,1}(x_1)$. The marginalization property of the periodic Green function shows that the energy and dissipation are exactly the corresponding one-dimensional quantities. This is the simplest higher-dimensional counterexample. However, its target vanishes on the codimension-one subtorus $\{x_1=0\}$, not at a single point, so it does not prove \cref{prop:counterexvanishingmu} when $\d\ge2$.
\end{remark}

\begin{proof}[Proof of \cref{prop:counterexvanishingmu}]
The one-dimensional case follows from \cref{lem:one-dimensional-counterexample}. Assume $\d\ge2$ and identify $\T^\d$ with $(-\frac12,\frac12]^\d$. Fix $r_0\in(0,1/8)$ and choose $\eta\in C_c^\infty(B_{3r_0})$ such that $0\le\eta\le1$ and $\eta=1$ on $B_{2r_0}$. Define
\begin{equation}
\widetilde\mu(x):=\eta(x)|x|^2+1-\eta(x),\qquad
\mu:=\frac{\widetilde\mu}{\int_{\T^\d}\widetilde\mu}.
\end{equation}
This is a smooth probability density that vanishes only at the origin, and there is a constant $c_0>0$ such that
\begin{equation}\label{eq:localized-target-quadratic}
\mu(x)=c_0|x|^2\qquad\text{for }|x|\le2r_0.
\end{equation}

Choose a radial cutoff $\chi\in C_c^\infty(B_1)$ with $0\le\chi\le1$ and $\chi=1$ on $B_{1/2}$, and set
\begin{equation}
\Phi(y):=|y|^4\chi(y),\qquad M:=\|(\Delta\Phi)_-\|_{L^\infty}.
\end{equation}
Fix $a>0$ so small that $aM\le c_0/4$ (with no restriction if $M=0$). For $0<\delta<r_0$, define the smooth periodic function
\begin{equation}
\Phi_\delta(x):=a\delta^4\Phi(x/\delta)
\end{equation}
in $B_\delta$ and extend it by zero outside that coordinate ball. Put
\begin{equation}\label{eq:localized-source-definition}
E_\delta:=\nabla\Phi_\delta,\qquad
\rho_\delta:=\mu+\div E_\delta=\mu+\Delta\Phi_\delta.
\end{equation}
Since $\int_{\T^\d}\Delta\Phi_\delta=0$, the source has mass one. We verify its nonnegativity. If $|x|\le\delta/2$, then
\begin{equation}
\Delta\Phi(x/\delta)=4(\d+2)|x/\delta|^2\ge0.
\end{equation}
If $\delta/2\le|x|\le\delta$, then \eqref{eq:localized-target-quadratic} and the choice of $a$ give
\begin{equation}
\rho_\delta(x)
=\delta^2\left(c_0|x/\delta|^2+a\Delta\Phi(x/\delta)\right)
\ge\delta^2(c_0/4-aM)\ge0.
\end{equation}
Outside $B_\delta$, $\rho_\delta=\mu$. Thus, $\rho_\delta$ is a smooth probability density.

Because $\rho_\delta-\mu=\Delta\Phi_\delta$, the zero-mean periodic solution of
$-\Delta h_\delta=\rho_\delta-\mu$ differs from $-\Phi_\delta$ only by a constant. Hence
$-\nabla\g\ast(\rho_\delta-\mu)=\nabla\Phi_\delta=E_\delta$. A change of variables gives
\begin{equation}\label{eq:localized-energy-scaling}
\mmd^2(\rho_\delta,\mu)
=\frac{a^2}{2}\delta^{\d+6}\int_{B_1}|\nabla\Phi(y)|^2\,dy.
\end{equation}
Moreover, on the support of $E_\delta$, the bracket in the preceding formula for $\rho_\delta$ is uniformly bounded, so $0\le\rho_\delta\le C\delta^2$. Therefore
\begin{equation}\label{eq:localized-dissipation-scaling}
\dissip(\rho_\delta\mid\mu)
=\int_{\T^\d}|E_\delta|^2\rho_\delta\,dx
\le C a^2\delta^{\d+8}.
\end{equation}
The quotient of \eqref{eq:localized-energy-scaling} by \eqref{eq:localized-dissipation-scaling} is bounded below by a positive constant times $\delta^{-2}$ and therefore tends to infinity.
\end{proof}

\subsection{\texorpdfstring{Proof of \cref{prop:counterexamplePLinfmu}}{Proof of the target-class uniformity obstruction}}

Here, the target never degenerates: every $\mu_\varepsilon$ stays above the same $\lambda$. The obstruction instead comes from concentrating the remaining target mass near one point and choosing a source that vanishes on a fixed ball. This confines the dissipation to a fixed shell while the interior Coulomb energy diverges as $\varepsilon\downarrow0$ in $\d\ge2$; the proof first constructs the pair and then verifies admissibility and these two estimates.

    Fix $0<\lambda<1$, fix a radius $R\in(0,1/10)$, and choose $\delta\in (0,R/(\d-1)]$ with $R+\delta<1/5$. Let $B_R:=B(0,R)\subset\T^\d$ and let $\eta_\varepsilon\in C_c^\infty(B_\varepsilon)$ be nonnegative, radial, and normalized by
    \begin{equation}
    \int_{\T^\d} \eta_\varepsilon\,dx=1.
    \end{equation}
    Define the target density
    \begin{equation}
    \mu_\varepsilon(x):=\lambda + (1-\lambda)\eta_\varepsilon(x).
    \end{equation}
    Then $\mu_\varepsilon\in C^\infty(\T^\d)$, $\mu_\varepsilon\ge \lambda$, and $\int_{\T^\d}\mu_\varepsilon=1$.

    Let $v_\varepsilon=v_\varepsilon(r)$ solve the radial Dirichlet problem
    \begin{equation}
    -v_\varepsilon''(r)-\frac{\d-1}{r}v_\varepsilon'(r)=\mu_\varepsilon(r)
    \qquad (0<r<R),
    \qquad
    v_\varepsilon'(0)=0,
    \qquad
    v_\varepsilon(R)=0.
    \end{equation}
    Set
    \begin{equation}
    M:=\int_{B_R}\mu_\varepsilon(x)\,dx = \lambda |B_R| + (1-\lambda),
    \qquad
    q:= -v_\varepsilon'(R)=\frac{M}{\sigma_\d R^{\d-1}},
    \end{equation}
    where the last identity follows from radial Gauss law; in particular, $q$ is independent of $\varepsilon$.

    Now define a radial profile on the shell $R\le r\le R+\delta$ by
    \begin{equation}
    W(r):=q\delta\Big(\frac12 - \frac{r-R}{\delta} + \frac12\Big(\frac{r-R}{\delta}\Big)^2\Big).
    \end{equation}
    Then
    \begin{equation}
    W(R)=\frac{q\delta}{2},\qquad W'(R)=-q,
    \qquad W(R+\delta)=0,
    \qquad W'(R+\delta)=0.
    \end{equation}
    Finally, set the global radial potential
    \begin{equation}
    u_{\varepsilon,\delta}(x):=
    \begin{cases}
    v_\varepsilon(|x|)+\dfrac{q\delta}{2}, & |x|\le R,\\[1ex]
    W(|x|), & R\le |x|\le R+\delta,\\[1ex]
    0, & |x|\ge R+\delta.
    \end{cases}
    \end{equation}
    Since the value and radial derivative match at $r=R$ and $r=R+\delta$, the function $u_{\varepsilon,\delta}$ is $C^1$ and piecewise $C^\infty$.

    Define the source density by
    \begin{equation}
    \rho_{\varepsilon,\delta}:= \mu_\varepsilon + \Delta u_{\varepsilon,\delta}.
    \end{equation}
    The following lemma collects the required properties of the constructed pair.

    \begin{lemma}
    For every fixed $\delta$ as above and every $0<\varepsilon<R$:
    \begin{enumerate}[label=\roman*.,leftmargin=2em]
    \item $\rho_{\varepsilon,\delta}\ge 0$ and $\int_{\T^\d}\rho_{\varepsilon,\delta}=1$;
    \item $\rho_{\varepsilon,\delta}=0$ on $B_R$ and $\rho_{\varepsilon,\delta}=\lambda$ on $\T^\d\setminus B_{R+\delta}$;
    \item the dissipation $\dissip (\rho_{\varepsilon,\delta}\mid \mu_\varepsilon)$ is bounded uniformly in $\varepsilon$;
    \item the energy $\mmd^2 (\rho_{\varepsilon,\delta},\mu_\varepsilon)$ diverges as $\varepsilon\downarrow0$:
    \begin{equation}
    \mmd^2 (\rho_{\varepsilon,\delta},\mu_\varepsilon)
    \ge c_{\d,\lambda}
    \begin{cases}
    \log(R/\varepsilon), & \d=2,\\[0.5ex]
    \varepsilon^{2-\d}-R^{2-\d}, & \d\ge 3,
    \end{cases}
    \end{equation}
    for some $c_{\d,\lambda}>0$ independent of $\varepsilon$.
    Consequently,
    \begin{equation}
    \frac{\dissip (\rho_{\varepsilon,\delta}\mid \mu_\varepsilon)}{\mmd^2 (\rho_{\varepsilon,\delta},\mu_\varepsilon)}\longrightarrow 0
    \qquad (\varepsilon\downarrow 0).
    \end{equation}
    \end{enumerate}
    \end{lemma}

    \begin{proof}
    \ul{Step 1: explicit form of $\rho_{\varepsilon,\delta}$.}
    Inside $B_R$ we have $\Delta u_{\varepsilon,\delta}=\Delta v_\varepsilon=-\mu_\varepsilon$, hence
    \begin{equation}
    \rho_{\varepsilon,\delta}=0 \qquad\text{on }B_R.
    \end{equation}
    Outside $B_{R+\delta}$ we have $u_{\varepsilon,\delta}=0$, so
    \begin{equation}
    \rho_{\varepsilon,\delta}=\mu_\varepsilon=\lambda
    \qquad\text{on }\T^\d\setminus B_{R+\delta},
    \end{equation}
    because $\eta_\varepsilon$ is supported in $B_\varepsilon\subset B_R$.

    In the shell, writing $s=(r-R)/\delta\in[0,1]$, one has
    \begin{equation}
    W'(r)=-q(1-s),
    \qquad
    W''(r)=\frac{q}{\delta}.
    \end{equation}
    Therefore,
    \begin{equation}
    \Delta u_{\varepsilon,\delta}(x)
    = \frac{q}{\delta}-\frac{(\d-1)q}{|x|}(1-s)
    \qquad (R<|x|<R+\delta),
    \end{equation}
    and so
    \begin{equation}
    \rho_{\varepsilon,\delta}(x)
    = \lambda + \frac{q}{\delta}-\frac{(\d-1)q}{|x|}(1-s).
    \end{equation}
    Since $|x|\ge R$, $0\le 1-s\le 1$, and $\delta\le R/(\d-1)$, we get
    \begin{equation}
    \rho_{\varepsilon,\delta}(x)
    \ge \lambda + \frac{q}{\delta}-\frac{(\d-1)q}{R}
    \ge \lambda>0.
    \end{equation}
    Thus, $\rho_{\varepsilon,\delta}\ge 0$ everywhere.

    Also,
    \begin{equation}
    \int_{\T^\d}\rho_{\varepsilon,\delta}
    = \int_{\T^\d}\mu_\varepsilon + \int_{\T^\d}\Delta u_{\varepsilon,\delta}
    =1+0=1,
    \end{equation}
    because $u_{\varepsilon,\delta}$ is $C^1$ and periodic. This proves (i)--(ii).

    \medskip
    \noindent
    \ul{Step 2: uniform upper bound on the dissipation.}
    Because $\rho_{\varepsilon,\delta}=0$ in $B_R$ and $\nabla u_{\varepsilon,\delta}=0$ outside $B_{R+\delta}$, only the shell contributes to the dissipation:
    \begin{equation}
    \dissip (\rho_{\varepsilon,\delta}\mid \mu_\varepsilon)
    =\int_{A_{R,R+\delta}} |\nabla u_{\varepsilon,\delta}|^2\rho_{\varepsilon,\delta}\,dx.
    \end{equation}
    In the shell,
    \begin{equation}
    |\nabla u_{\varepsilon,\delta}(x)| = q(1-s)\le q,
    \end{equation}
    and
    \begin{equation}
    \rho_{\varepsilon,\delta}(x)
    \le \lambda + \frac{q}{\delta}
    \qquad (R<|x|<R+\delta),
    \end{equation}
    so
    \begin{equation}
    \dissip (\rho_{\varepsilon,\delta}\mid \mu_\varepsilon)
    \le |A_{R,R+\delta}| q^2\Big(\lambda+\frac{q}{\delta}\Big)=:C_{\d,\lambda,R,\delta}.
    \end{equation}
    This constant is independent of $\varepsilon$.

    \medskip
    \noindent
    \ul{Step 3: divergence of the energy.}
    Let
    \begin{equation}
    m_\varepsilon(r):=\int_{B_r}\mu_\varepsilon(x)\,dx.
    \end{equation}
    By radial Gauss law,
    \begin{equation}
    -v_\varepsilon'(r)=\frac{m_\varepsilon(r)}{\sigma_\d r^{\d-1}}
    \qquad (0<r<R).
    \end{equation}
    Hence the interior part of the Dirichlet energy is
    \begin{equation}
    \frac12\int_{B_R}|\nabla v_\varepsilon|^2\,dx
    = \frac{\sigma_\d}{2}\int_0^R r^{\d-1}|v_\varepsilon'(r)|^2\,dr
    = \frac{1}{2\sigma_\d}\int_0^R \frac{m_\varepsilon(r)^2}{r^{\d-1}}\,dr.
    \end{equation}
    Now $\eta_\varepsilon$ is supported in $B_\varepsilon$ and has total mass $1$, so for every $r\in[\varepsilon,R]$,
    \begin{equation}
    m_\varepsilon(r)\ge (1-\lambda)\int_{B_r}\eta_\varepsilon(x)\,dx = 1-\lambda.
    \end{equation}
    Therefore,
    \begin{equation}
    \mmd^2 (\rho_{\varepsilon,\delta},\mu_\varepsilon)
    \ge \frac12\int_{B_R}|\nabla v_\varepsilon|^2\,dx
    \ge \frac{(1-\lambda)^2}{2\sigma_\d}\int_\varepsilon^R r^{1-\d}\,dr.
    \end{equation}
    This equals
    \begin{equation}
    \frac{(1-\lambda)^2}{4\pi}\log\frac{R}{\varepsilon}
    \qquad (\d=2),
    \end{equation}
    and
    \begin{equation}
    \frac{(1-\lambda)^2}{2\sigma_\d(\d-2)}\big(\varepsilon^{2-\d}-R^{2-\d}\big)
    \qquad (\d\ge 3).
    \end{equation}
    Hence $\mmd^2(\rho_{\varepsilon,\delta},\mu_\varepsilon)\to\infty$ as $\varepsilon\downarrow0$. Since the dissipation remains bounded uniformly in $\varepsilon$, the ratio tends to $0$, which proves \cref{prop:counterexamplePLinfmu}.
    \end{proof}

\section{The PL inequality on \texorpdfstring{$\R^\d$}{Euclidean space} for radial data}\label{sec:radial-pl}

The goal of this section is to prove \cref{thm:connected-support-PL}, to the best of our knowledge the first PL inequality on $\R^\d$ for the Coulomb discrepancy, under radial symmetry and source-support inclusion. This is a deliberately restricted positive result in dimensions $\d\ge2$. \Cref{sec:whole-space-obstructions} contrasts it with global obstructions that arise when the source may be placed arbitrarily far from the target.
Radial symmetry reduces the Coulomb field, energy, and dissipation to one-dimensional expressions in the cumulative radial masses. We record these identities first; they underlie both the sharpness examples and the quantile-based proof of the positive result.

For a radial probability measure $f$, define its cumulative radial mass by
\begin{equation}
m_f(r):=f(\overline B_r)
=\int_{\R^\d}\indic_{\{|x|\le r\}}\,df(x),
\qquad r\ge0.
\end{equation}
If $f$ is absolutely continuous with radial density, again denoted by $f$, then
\begin{equation}
m_f(r)=\sigma_\d\int_0^r s^{\d-1}f(s)\,ds.
\end{equation}
Here, $\omega_\d:=|B_1|$ and $\sigma_\d:=|\mathbb S^{\d-1}|=\d\omega_\d$. The function $m_f$ is nondecreasing, bounded, and right-continuous, which is enough for the computations below.

If we set
\begin{equation}
M(r):=m_\rho(r)-m_\mu(r),
\end{equation}
since $\g$ is the fundamental solution of $-\D$, we have, for \(x\neq 0\),\footnote{Recall that the Laplacian of a radial function $f$ is $\D f(r) = r^{-(\d-1)}\partial_r (r^{\d-1} \partial_r f)$.}
\begin{equation}
E(x)\equiv -\nab\g\ast (\rho - \mu) (x)=\frac{M(|x|)}{\sigma_\d |x|^{\d-1}}\frac{x}{|x|}.
\end{equation}
Consequently,
\begin{align}
\mmd^2(\rho,\mu) &= \frac{1}{2\sigma_\d}\int_0^\infty \frac{M(r)^2}{r^{\d-1}}\,dr, \label{eq:radial-energy}\\
\dissip(\rho\mid\mu) &= \frac{1}{\sigma_\d}\int_0^\infty \frac{M(r)^2}{r^{\d-1}}\rho(r)\,dr. \label{eq:radial-diss}
\end{align}

\subsection{Sharpness of the support assumptions}

The support conditions (i)--(ii) in \cref{thm:connected-support-PL} are necessary, as the following counterexamples show.

The first condition asserts that there is no gap in the support of the target. In short, this connectedness assumption allows mass from the source to be transported evenly to the target, without having to jump over a gap where the target vanishes. 
\begin{prop}[Sharpness of condition i]\label{prop:sharpi}
    Fix \(\d\ge 2\). There exist bounded radial probability densities \(\mu\) and \(\rho_\delta\), all supported in \(B_1\), such that
\begin{equation}
0<\lambda\le \mu\le A<\infty \quad \text{almost everywhere on }\supp\mu,
\qquad
\supp\rho_\delta\subset\supp\mu,
\end{equation}
but
\begin{equation}
\frac{\mmd^2(\rho_\delta,\mu)}{\dissip(\rho_\delta\mid\mu)}\to\infty
\qquad \text{as } \delta\downarrow 0.
\end{equation}
\end{prop}

\begin{proof}
Choose radii \(0<a<b<1\) and set
\begin{equation}
S:=B_a\cup (B_1\setminus B_b),
\qquad
c:=|S|^{-1},
\qquad
\mu:=c\,\indic_{S}.
\end{equation}
Thus, \(\mu\) is radial and constant on its support; in the proposition statement one may take \(\lambda=A=c\).

For small \(\delta>0\), define the inner and outer thin shells
\begin{equation}
I_\delta:=B_a\setminus B_{a-\delta},
\qquad
O_\delta:=B_{b+\delta}\setminus B_b.
\end{equation}
Remove the target mass from \(O_\delta\) and move it to \(I_\delta\):
\begin{equation}
\rho_\delta := c\,\indic_{S\setminus O_\delta}
+ c\frac{|O_\delta|}{|I_\delta|}\indic_{I_\delta}.
\end{equation}
Then \(\rho_\delta\) is a radial probability density supported in \(S=\supp\mu\), hence \(\supp\rho_\delta\subset\supp\mu\).

Let
\begin{equation}
\varepsilon_\delta:=\mu(O_\delta)=c|O_\delta|\sim C\delta.
\end{equation}
If
\begin{equation}
M_\delta(r):=m_{\rho_\delta}(r)-m_\mu(r),
\end{equation}
then by construction:
\begin{itemize}[leftmargin=2em]
    \item \(M_\delta(r)=0\) for \(0\le r\le a-\delta\);
    \item \(M_\delta\) rises from \(0\) to \(\varepsilon_\delta\) on \([a-\delta,a]\);
    \item \(M_\delta(r)=\varepsilon_\delta\) for all \(r\in[a,b]\);
    \item \(M_\delta\) falls from \(\varepsilon_\delta\) to \(0\) on \([b,b+\delta]\);
    \item \(M_\delta(r)=0\) for \(r\ge b+\delta\).
\end{itemize}
Hence
\begin{equation}
\mmd^2(\rho_\delta,\mu)
\ge \frac{\varepsilon_\delta^2}{2\sigma_\d}\int_a^b r^{1-\d}\,dr
\sim C\delta^2.
\end{equation}
For the dissipation, note that \(\rho_\delta=0\) on the whole gap \([a,b]\) and also on the outer shell \(O_\delta\). Therefore, the only contribution comes from the inner shell \(I_\delta\):
\begin{equation}
\dissip(\rho_\delta\mid\mu)
= \frac{1}{\sigma_\d}\int_{a-\delta}^a \frac{M_\delta(r)^2}{r^{\d-1}}\rho_\delta(r)\,dr.
\end{equation}
Since \(|M_\delta(r)|\le \varepsilon_\delta\), \(|O_\delta|/|I_\delta|=O(1)\) as \(\delta\downarrow0\), and \(|I_\delta|\sim C\delta\), the density \(\rho_\delta\) is uniformly bounded on \(I_\delta\), and one gets
\begin{equation}
\dissip(\rho_\delta\mid\mu)\le C\,\varepsilon_\delta^2 |I_\delta| \lesssim \delta^3.
\end{equation}
Thus, the ratio diverges like \(\delta^{-1}\).
\end{proof}

The second condition requires the source to be supported within the support of the target. If this condition is violated, mass must be transported across a region where the target vanishes; this can
result in a large energy discrepancy while the dissipation remains small. Put differently, the repulsive nature of the Coulomb potential is sufficiently strong to prevent the flow from ``pulling'' all the
mass of the source into the support of the target.

\begin{prop}[Sharpness of condition ii]\label{prop:sharpii}
Fix \(\d\ge 2\) and \(R>1\). Let \(\mu\) be any radial probability density such that
\begin{equation}
\supp\mu = B_1,
\qquad
0<\lambda\le \mu(x)\le A<\infty
\quad\text{for a.e. }x\in B_1.
\end{equation}
Then there exists a family of radial probability densities \((\rho_\delta)_{\delta>0}\), each supported in \(B_R\), such that
\begin{equation}
\mmd(\rho_\delta,\mu)\to 0,
\qquad
\frac{\mmd^2(\rho_\delta,\mu)}{\dissip(\rho_\delta\mid\mu)}\to\infty
\quad\text{as }\delta\downarrow 0.
\end{equation}
In particular, there is no inequality of the form
\begin{equation}
\mmd^2(\rho,\mu) \le C(\mu,R)\,\dissip(\rho\mid\mu)
\end{equation}
valid for all radial probability densities \(\rho\) supported in \(B_R\).
\end{prop}
\begin{proof}

Consider any radial target \(\mu\) supported in \(B_1\) with \(0<\lambda\le \mu\le A\) almost everywhere on \(B_1\).
Fix radii \(1<S<S+h<R\) and define the outer annulus
\begin{equation}
O:=B_{S+h}\setminus B_S,
\qquad
\eta:=\frac{1}{|O|}\indic_O.
\end{equation}
For \(\delta\in(0,1)\), let
\begin{equation}
\varepsilon_\delta:=\mu(B_1\setminus B_{1-\delta}).
\end{equation}
Because \(\lambda\le \mu\le A\) almost everywhere on \(B_1\), one has \(\varepsilon_\delta\asymp\delta\).
Now define
\begin{equation}
\rho_\delta := \mu\,\indic_{B_{1-\delta}} + \varepsilon_\delta\eta.
\end{equation}
Thus, \(\rho_\delta\) is radial, supported in \(B_R\), and differs from \(\mu\) by moving a thin outer shell of mass to the disjoint annulus \(O\).

Let
\begin{equation}
M_\delta(r):=m_{\rho_\delta}(r)-m_\mu(r).
\end{equation}
If \(r\in[1,S]\), then all the annulus mass still lies outside \(B_r\), so
\begin{equation}
m_{\rho_\delta}(r)=1-\varepsilon_\delta,
\qquad
m_\mu(r)=1,
\end{equation}
and therefore
\begin{equation}
M_\delta(r)=-\varepsilon_\delta
\qquad \text{for all } r\in[1,S].
\end{equation}
Using \eqref{eq:radial-energy},
\begin{equation}
\mmd^2(\rho_\delta,\mu)
\ge \frac{1}{2\sigma_\d}\int_1^S \frac{\varepsilon_\delta^2}{r^{\d-1}}\,dr
\sim C\varepsilon_\delta^2.
\end{equation}
In fact, the matching upper bound \(\mmd^2(\rho_\delta,\mu)\lesssim \varepsilon_\delta^2\) is also immediate from \eqref{eq:radial-energy}, so \(\mmd(\rho_\delta,\mu)\to 0\).

For the dissipation, note that \(\rho_\delta=0\) on the whole gap \([1-\delta,S]\). Thus, the only contribution comes from the outer annulus \([S,S+h]\), where \(\rho_\delta=\varepsilon_\delta/|O|\) and \(|M_\delta(r)|\le \varepsilon_\delta\). Therefore,
\begin{equation}
\dissip(\rho_\delta\mid\mu)
\le \frac{1}{\sigma_\d}\int_S^{S+h} \frac{\varepsilon_\delta^2}{r^{\d-1}}\frac{\varepsilon_\delta}{|O|}\,dr
\lesssim \varepsilon_\delta^3.
\end{equation}
Hence
\begin{equation}
\frac{\mmd^2(\rho_\delta,\mu)}{\dissip(\rho_\delta\mid\mu)}\gtrsim \varepsilon_\delta^{-1}\to\infty.
\end{equation}
\end{proof}

\begin{remark}[On condition iii]\label{rem:sharpiii}
    The one-dimensional mechanism in \cref{lem:one-dimensional-counterexample} suggests an analogous annular obstruction when a continuous radial target degenerates at an interior radius. Establishing such a counterexample requires controlling the radial Jacobian in both the energy and the dissipation, and we do not pursue that separate construction here. Thus \cref{prop:sharpi,prop:sharpii} establish sharpness of the two support hypotheses, while the sharpness of the uniform lower bound in the radial class remains open in the present analysis.
\end{remark}

\subsection{\texorpdfstring{Proof of \cref{thm:connected-support-PL}}{Proof of the radial PL theorem}}

Having established sharpness of the two support hypotheses, we turn to the positive estimate. The proof decomposes the energy and dissipation over radial mass labels and reduces the PL inequality to a pointwise shellwise bound, which is then integrated in the mass parameter.

For a radial probability density \(\nu\) supported in \(A_{a,R}\), define its radial quantile by
\begin{equation}
Q_\nu(\alpha):=\inf\{r\in[a,R]: m_\nu(r)\ge \alpha\},
\qquad \alpha\in(0,1).
\end{equation}
Since $m_\nu$ is right-continuous, $Q_\nu(\alpha)$ is well-defined for all $\alpha \in (0,1)$ and the infimum is attained. We always have \(m_\nu(Q_\nu(\alpha))\ge \alpha\) for every \(\alpha\in(0,1)\). Moreover, if \(\nu\) has a density with respect to the Lebesgue measure, then $Q_\nu(\alpha)>0$ and $m_\nu$ is continuous, with \(m_\nu(Q_\nu(\alpha))=\alpha\) for almost every \(\alpha\in(0,1)\).
For fixed \(\alpha\in(0,1)\), set
\begin{equation}
s_\alpha:=Q_\mu(\alpha),
\qquad
\Gamma_\alpha(r):=\frac{m_\mu(r)-\alpha}{\sigma_\d r^{\d-1}}
\quad (r>0),
\qquad
\Psi_\alpha(r):=\int_{s_\alpha}^r \Gamma_\alpha(t)\,dt
\quad (r>0).
\end{equation}
Observe that \(\Gamma_\alpha(s_\alpha)=0\) and \(\Gamma_\alpha\) is nonnegative on \([s_\alpha,R]\) and nonpositive on \((0,s_\alpha]\). Hence \(\Psi_\alpha\) is well-defined on $(0,R]$.

\begin{lemma}[Shell decomposition of energy and dissipation]\label{lem:shell-decomposition}
For radial probability densities \(\rho,\mu\), one has
\begin{align}
\mmd^2(\rho,\mu) &= \int_0^1 \Psi_\alpha\bigl(Q_\rho(\alpha)\bigr)\,d\alpha, \label{eq:shell-energy}\\
\dissip(\rho\mid\mu) &= \int_0^1 \Gamma_\alpha\bigl(Q_\rho(\alpha)\bigr)^2\,d\alpha. \label{eq:shell-diss}
\end{align}
\end{lemma}

\begin{proof}
We begin with the dissipation formula. Let $\nu_\rho:=(|\cdot|)_\#(\rho\,dx)$ be the radial law of $\rho$. Its distribution function is $m_\rho$, and $d\nu_\rho(r)=\sigma_\d r^{\d-1}\rho(r)\,dr$. Since $Q_\rho$ is the generalized inverse of $m_\rho$, the quantile pushforward identity gives, for every nonnegative Borel function $\varphi$,
\begin{equation}\label{eq:radial-quantile-pushforward}
\int_0^1\varphi(Q_\rho(\alpha))\,d\alpha
=\int_0^\infty\varphi(r)\,\sigma_\d r^{\d-1}\rho(r)\,dr.
\end{equation}
Applying \eqref{eq:radial-quantile-pushforward} to the nonnegative Borel function
$\varphi(r)=\bigl((m_\rho(r)-m_\mu(r))/(\sigma_\d r^{\d-1})\bigr)^2$ for $r>0$, with an arbitrary value at $r=0$, and using \eqref{eq:radial-diss}, we obtain
\begin{equation}
\dissip(\rho\mid\mu)
= \int_0^1
\Bigl(\frac{m_\rho(Q_\rho(\alpha))-m_\mu(Q_\rho(\alpha))}{\sigma_\d Q_\rho(\alpha)^{\d-1}}\Bigr)^2\,d\alpha.
\end{equation}
The radial law $\nu_\rho$ is nonatomic, and therefore $m_\rho(Q_\rho(\alpha))=\alpha$ for almost every $\alpha\in(0,1)$. Consequently,
\begin{equation}
\dissip(\rho\mid\mu)
= \int_0^1
\Bigl(\frac{\alpha-m_\mu(Q_\rho(\alpha))}{\sigma_\d Q_\rho(\alpha)^{\d-1}}\Bigr)^2\,d\alpha
= \int_0^1 \Gamma_\alpha(Q_\rho(\alpha))^2\,d\alpha.
\end{equation}
This proves \eqref{eq:shell-diss}.

For the energy, fix \(\alpha\) and note that
\begin{equation}
\Psi_\alpha(Q_\rho(\alpha))
=
\int_0^\infty \Gamma_\alpha(r)
\bigl(\indic_{\{Q_\mu(\alpha)\le r\}}-\indic_{\{Q_\rho(\alpha)\le r\}}\bigr)\,dr.
\end{equation}
Integrating in \(\alpha\) and using Fubini\footnote{Since \(\alpha\) lies between \(m_\mu(r)\) and \(m_\rho(r)\) in the integrand, we have \(|m_\mu(r)-\alpha|\le |m_\mu(r)-m_\rho(r)|\). The absolute integral is thus bounded by \(\int_0^\infty \frac{|m_\mu(r)-m_\rho(r)|^2}{\sigma_\d r^{\d-1}}\,dr = 2\mmd^2(\rho,\mu) < \infty\).},
\begin{equation}
\int_0^1 \Psi_\alpha(Q_\rho(\alpha))\,d\alpha
=
\int_0^\infty \frac{1}{\sigma_\d r^{\d-1}}
\int_0^1 (m_\mu(r)-\alpha)
\bigl(\indic_{\{\alpha\le m_\mu(r)\}}-\indic_{\{\alpha\le m_\rho(r)\}}\bigr)
\,d\alpha\,dr.
\end{equation}
If \(a=m_\mu(r)\) and \(b=m_\rho(r)\), then a direct scalar computation gives
\begin{equation}
\int_0^1 (a-\alpha)(\indic_{\{\alpha\le a\}}-\indic_{\{\alpha\le b\}})\,d\alpha
= \frac12 (b-a)^2.
\end{equation}
Hence
\begin{equation}
\int_0^1 \Psi_\alpha(Q_\rho(\alpha))\,d\alpha
=
\frac{1}{2\sigma_\d}\int_0^\infty \frac{(m_\rho(r)-m_\mu(r))^2}{r^{\d-1}}\,dr
=\mmd^2(\rho,\mu),
\end{equation}
which proves \eqref{eq:shell-energy}.
\end{proof}

It therefore remains to compare $\Psi_\alpha$ with $\Gamma_\alpha^2$ uniformly in the mass label.

\begin{lemma}[Shellwise PL inequality]\label{lem:shellwise-PL}
Under the assumptions of \cref{thm:connected-support-PL}, for every \(\alpha\in(0,1)\) and every \(r\in[a,R]\cap(0,\infty)\),
\begin{equation}
\Psi_\alpha(r)
\le C_0\,\Gamma_\alpha(r)^2,
\qquad
C_0:=\max\Bigl\{\frac{1}{\lambda},\frac{\d^2A}{2\lambda^2}\Bigr\}.
\end{equation}
\end{lemma}

\begin{proof}
Fix \(\alpha\) and write \(s:=s_\alpha=Q_\mu(\alpha)\).

\smallskip
\noindent\emph{Case 1: \(r\le s\).}
For \(t\in[r,s]\),
\begin{equation}
|\Gamma_\alpha(t)|
= \frac{1}{\sigma_\d t^{\d-1}}\int_t^s \sigma_\d u^{\d-1}\mu(u)\,du.
\end{equation}
Since \(t\mapsto t^{-(\d-1)}\) is decreasing and the interval of integration shrinks as \(t\) increases, one has
\begin{equation}
|\Gamma_\alpha(t)|\le |\Gamma_\alpha(r)|
\qquad \text{for all } t\in[r,s].
\end{equation}
Therefore,
\begin{equation}
\Psi_\alpha(r)=\int_r^s |\Gamma_\alpha(t)|\,dt
\le (s-r)|\Gamma_\alpha(r)|.
\end{equation}
On the other hand, using \(\mu\ge \lambda\),
\begin{equation}
|\Gamma_\alpha(r)|
= \frac{1}{\sigma_\d r^{\d-1}}\int_r^s \sigma_\d u^{\d-1}\mu(u)\,du
\ge \lambda(s-r).
\end{equation}
Hence
\begin{equation}
\Psi_\alpha(r)\le \frac{1}{\lambda}\Gamma_\alpha(r)^2.
\end{equation}
\smallskip
\noindent\emph{Case 2: \(r\ge s\).}
For \(t\in[s,r]\),
\begin{equation}
\Gamma_\alpha(t)
= \frac{1}{\sigma_\d t^{\d-1}}\int_s^t \sigma_\d u^{\d-1}\mu(u)\,du
\le \frac{A}{\sigma_\d t^{\d-1}}\int_s^t \sigma_\d u^{\d-1}\,du
\le A(t-s).
\end{equation}
Integrating gives
\begin{equation}
\Psi_\alpha(r)=\int_s^r \Gamma_\alpha(t)\,dt
\le \frac{A}{2}(r-s)^2.
\end{equation}
Conversely, using \(\mu\ge \lambda\),
\begin{equation}
\Gamma_\alpha(r)
\ge \frac{\lambda}{\sigma_\d r^{\d-1}}\int_s^r \sigma_\d u^{\d-1}\,du
\ge \frac{\lambda}{\d}(r-s).
\end{equation}
Therefore,
\begin{equation}
\Psi_\alpha(r)
\le \frac{A}{2}(r-s)^2
\le \frac{\d^2A}{2\lambda^2}\Gamma_\alpha(r)^2.
\end{equation}
Combining the two cases yields the claim.
\end{proof}

\begin{proof}[Proof of \cref{thm:connected-support-PL}]
By \cref{lem:shell-decomposition,lem:shellwise-PL}, 
\begin{equation}
\mmd^2(\rho,\mu)
= \int_0^1 \Psi_\alpha(Q_\rho(\alpha))\,d\alpha
\le C_0\int_0^1 \Gamma_\alpha(Q_\rho(\alpha))^2\,d\alpha
=C_0\,\dissip(\rho\mid\mu),
\end{equation}
where we have used that $\supp\rho\subset \supp\mu$ in order to apply the shellwise inequality for every \(\alpha\).
Since
\begin{equation}
C_0 = \max\Bigl\{\frac{1}{\lambda},\frac{\d^2A}{2\lambda^2}\Bigr\},
\end{equation}
this is exactly the desired estimate.
\end{proof}

To pass from the static PL inequality to convergence of the flow, we verify that the radial dynamics preserve source-support inclusion and then use the energy--dissipation inequality.

\begin{proof}[Proof of \cref{cor:radialconvergence}]
Because the target and initial datum are radial, uniqueness implies that the constructed solution remains radial. Since $\rho_0,\mu\in L^\infty$ and have compact support, the associated velocity is Osgood continuous in space on bounded time intervals. The Lagrangian representation for continuity equations with Osgood vector fields \cite{AB08} therefore applies, and $\rho_t$ is the pushforward of $\rho_0$ by the unique flow.

The uniqueness of the radial flow preserves the ordering of radii. For almost every mass label $\alpha\in(0,1)$, let $R_t(\alpha)$ be the radial flow characteristic with
\begin{equation}
R_0(\alpha)=Q_{\rho_0}(\alpha).
\end{equation}
The mass enclosed by $R_t(\alpha)$ remains equal to $\alpha$, and radial Gauss law gives
\begin{equation}\label{eq:radial-shell-corollary}
\dot R_t(\alpha)=\frac{\alpha-m_\mu(R_t(\alpha))}{\sigma_\d R_t(\alpha)^{\d-1}}
\end{equation}
whenever $R_t(\alpha)>0$. If $a>0$, the right-hand side of \eqref{eq:radial-shell-corollary} is nonnegative at $R_t(\alpha)=a$, because $m_\mu(a)=0$, and it is nonpositive at $R_t(\alpha)=R$, because $m_\mu(R)=1$. Uniqueness for the scalar characteristic equation therefore makes $[a,R]$ invariant. If $a=0$, no inner-boundary formula is needed: radial characteristics have nonnegative radius, while the same outer barrier at $R$ applies. We conclude directly at the weak-solution level that
\begin{equation}\label{eq:support-confinement}
\supp\rho_t\subset A_{a,R}=\supp\mu\qquad\text{for every }t\ge0.
\end{equation}

Set $F(t)=\mmd^2(\rho_t,\mu)$. By \eqref{eq:weak-energy-inequality}, \eqref{eq:support-confinement}, and \cref{thm:connected-support-PL}, for $t\ge s>0$,
\begin{equation}
F(t)+\frac1{C_{\d,\lambda,A}}\int_s^tF(\tau)\,d\tau\le F(s).
\end{equation}
The integral form of Gr\"onwall's lemma yields
\begin{equation}
F(t)\le e^{-(t-s)/C_{\d,\lambda,A}}F(s).
\end{equation}
The compactly supported bounded data have finite Coulomb discrepancy, and the $\dot H^{-1}$ chain rule used after \eqref{eq:weak-energy-inequality} gives $F(s)\to F(0)$ as $s\downarrow0$. Sending $s\downarrow0$ proves the stated estimate.
\end{proof}

\section{Whole-space obstructions at spatial infinity}\label{sec:whole-space-obstructions}

In the unrestricted whole-space setting, the key obstruction is the time needed for a remote source to reach the target. We first prove the travel-time persistence estimate \cref{thm:whole-space-persistence} and derive its consequences for uniform convergence rates and global PL coercivity in \cref{cor:no-universal-rate,cor:dynamical-global-PL}; we then give the independent static obstruction \cref{prop:static-global-PL}, which does not use the flow. Throughout the section, $\mu\in\P(\R^\d)\cap L^\infty(\R^\d)$ is compactly supported. We write $\sigma_\d:=|\mathbb S^{\d-1}|$ (with $\sigma_1=2$) and use the standard representatives of the whole-space Coulomb kernel,
\begin{equation}\label{eq:whole-space-coulomb-kernel}
\g(x)=
\begin{cases}
-\frac12|x|, & \d=1,\\[0.5ex]
-\frac{1}{2\pi}\log|x|, & \d=2,\\[0.5ex]
\frac{1}{(\d-2)\sigma_\d}|x|^{2-\d}, & \d\ge3.
\end{cases}
\end{equation}
For compactly supported bounded densities, all the energies below are finite, and, with $E=-\nabla\g\ast(\rho-\mu)$,
\begin{equation}\label{eq:whole-space-dirichlet-energy}
\mmd^2(\rho,\mu)=\frac12\int_{\R^\d}|E|^2\,dx,
\qquad
\div E=\rho-\mu.
\end{equation}

\hypertarget{bkm:bounded-speed-persistence}{}
\bookmarksetupnext{dest=bkm:bounded-speed-persistence}
\subsection{Bounded-speed propagation and persistence}\label{subsec:bounded-speed-persistence}

The persistence theorem combines the energy inequality with three auxiliary estimates. We first bound the velocity and the resulting support displacement, then obtain a lower bound for the squared MMD of a source that remains separated from the target, and finally compare this with an upper bound for the initial energy of a uniformly localized source.

\begin{lemma}[Uniform velocity and support displacement]\label{lem:uniform-velocity-support}
Let $\rho_0\in\P(\R^\d)\cap L^\infty(\R^\d)$ be compactly supported, and let $\rho_t$ be the corresponding bounded solution. There is a constant
\begin{equation}
V=V\bigl(\d,\|\rho_0\|_{L^\infty},\|\mu\|_{L^\infty}\bigr)<\infty
\end{equation}
such that
\begin{equation}\label{eq:uniform-velocity}
\|\nabla\g\ast(\rho_t-\mu)\|_{L^\infty}\le V
\qquad\text{for every }t\ge0.
\end{equation}
Moreover, if $K:=\supp\mu$ and $D_0:=\dist(\supp\rho_0,K)$, then
\begin{equation}\label{eq:support-distance-propagation}
\dist(\supp\rho_t,K)\ge D_0-Vt
\qquad\text{for every }t\ge0.
\end{equation}
\end{lemma}

\begin{proof}
Taking $p=\infty$ in \eqref{eq:hyper_est} and passing to the constructed solution gives
\begin{equation}\label{eq:uniform-density-bound-section5}
\|\rho_t\|_{L^\infty}
\le M_*:=\max\{\|\rho_0\|_{L^\infty},\|\mu\|_{L^\infty}\}
\qquad (t\ge0).
\end{equation}
If $\d=1$, then $\|\g'\|_{L^\infty(\R)}=1/2$, and hence the left-hand side of \eqref{eq:uniform-velocity} is at most one. If $\d\ge2$, the usual near--far decomposition gives, for every $R>0$ and every $f\in L^1\cap L^\infty$,
\begin{equation}
\|\nabla\g\ast f\|_{L^\infty}
\le C_\d\bigl(R\|f\|_{L^\infty}+R^{1-\d}\|f\|_{L^1}\bigr).
\end{equation}
Optimizing in $R$ and using $\|\rho_t-\mu\|_{L^1}\le2$ together with \eqref{eq:uniform-density-bound-section5}, we obtain \eqref{eq:uniform-velocity} with a constant independent of time and of the location of $\supp\rho_0$.

By \cref{lem:loglip}, the velocity is Osgood continuous in space. The uniqueness and representation theorem for continuity equations with Osgood vector fields \cite{AB08} therefore gives a unique flow $X_t$ such that $\rho_t=(X_t)_\#\rho_0$. From \eqref{eq:uniform-velocity},
\begin{equation}
|X_t(x)-x|\le\int_0^t\|v_s\|_{L^\infty}\,ds\le Vt.
\end{equation}
Taking the distance to the fixed closed set $K$ and then the infimum over $x\in\supp\rho_0$ yields \eqref{eq:support-distance-propagation}.
\end{proof}

The support-displacement estimate identifies the relevant time scale; the next lemma converts continued separation from the target into a lower bound for the squared MMD.

\begin{lemma}[Annular squared-MMD lower bound]\label{lem:annular-mmd-lower}
Fix $x_\mu\in\supp\mu$ and $R_\mu\ge1$ such that $\supp\mu\subset B_{R_\mu}(x_\mu)$. If $\rho\in\P(\R^\d)\cap L^\infty(\R^\d)$ is compactly supported and
\begin{equation}
D:=\dist(\supp\rho,\supp\mu)>R_\mu,
\end{equation}
then
\begin{equation}\label{eq:annular-mmd-lower}
\mmd^2(\rho,\mu)
\ge \frac{1}{2\sigma_\d}\int_{R_\mu}^{D}r^{1-\d}\,dr.
\end{equation}
\end{lemma}

\begin{proof}
Let $E=-\nabla\g\ast(\rho-\mu)$. For $R_\mu<r<D$, the ball $B_r(x_\mu)$ contains all the target mass and none of the source mass. Thus, for almost every such $r$, the Gauss law and \eqref{eq:whole-space-dirichlet-energy} give
\begin{equation}
\int_{\partial B_r(x_\mu)}E\cdot n\,dS
=\int_{B_r(x_\mu)}(\rho-\mu)\,dx=-1.
\end{equation}
By Cauchy--Schwarz,
\begin{equation}
\int_{\partial B_r(x_\mu)}|E|^2\,dS
\ge \frac{1}{|\partial B_r|}
=\frac{1}{\sigma_\d r^{\d-1}}.
\end{equation}
Integrating in $r$ and using \eqref{eq:whole-space-dirichlet-energy} proves \eqref{eq:annular-mmd-lower}. For $\d=1$, the boundary consists of the two endpoints of the interval $B_r(x_\mu)$, and the same calculation gives the factor $\sigma_1=2$.
\end{proof}

To compare this persistence lower bound with the initial discrepancy uniformly over translated sources, we also need an upper bound based only on their common localization data.

\begin{lemma}[Initial squared-MMD upper bound for uniformly localized sources]\label{lem:localized-initial-upper}
Let $M,L<\infty$. There are constants $C<\infty$ and $D_*>2$, depending only on $\d$, $\mu$, $M$, and $L$, such that every compactly supported probability density $\rho$ satisfying
\begin{equation}
\|\rho\|_{L^\infty}\le M,
\qquad
\diam(\supp\rho)\le L,
\qquad
D:=\dist(\supp\rho,\supp\mu)\ge D_*,
\end{equation}
obeys
\begin{equation}\label{eq:localized-initial-upper}
\mmd^2(\rho,\mu)\le C\Phi_\d(D).
\end{equation}
\end{lemma}

\begin{proof}
Write
\begin{equation}\label{eq:energy-interaction-decomposition}
\mmd^2(\rho,\mu)
=\frac12\iint\g(x-y)\,d\rho(x)\,d\rho(y)
+\frac12\iint\g(x-y)\,d\mu(x)\,d\mu(y)
-\iint\g(x-y)\,d\rho(x)\,d\mu(y).
\end{equation}
Let $L_\mu:=\diam(\supp\mu)$. Since the two supports are compact, their distance is attained. It follows from the diameter bounds that
\begin{equation}\label{eq:cross-distance-range}
D\le |x-y|\le D+L+L_\mu
\qquad (x\in\supp\rho,\ y\in\supp\mu).
\end{equation}

If $\d=1$, the two self-interaction terms in \eqref{eq:energy-interaction-decomposition} are nonpositive, while \eqref{eq:cross-distance-range} gives
\begin{equation}
-\iint\g(x-y)\,d\rho(x)\,d\mu(y)
=\frac12\iint|x-y|\,d\rho(x)\,d\mu(y)
\le \frac12(D+L+L_\mu).
\end{equation}
This proves \eqref{eq:localized-initial-upper}.

If $\d=2$, the positive part of $\g(z)=-(2\pi)^{-1}\log|z|$ is integrable near the origin. Therefore, the $L^\infty$ and mass bounds give a uniform upper bound on the source self-interaction; the target self-interaction is a fixed finite number. By \eqref{eq:cross-distance-range},
\begin{equation}
-\iint\g(x-y)\,d\rho(x)\,d\mu(y)
\le \frac{1}{2\pi}\log(D+L+L_\mu).
\end{equation}
For $D\ge D_*$ this is bounded by $C\log D$.

Finally, suppose $\d\ge3$. The kernel is positive, so the cross term in \eqref{eq:energy-interaction-decomposition} can be discarded in an upper bound. Moreover,
\begin{equation}
\sup_x\int \g(x-y)\rho(y)\,dy
\le C_\d\left(M\int_{B_1}|z|^{2-\d}\,dz+1\right),
\end{equation}
and the target self-interaction is finite. Thus, the energy is bounded uniformly in $D$, which is \eqref{eq:localized-initial-upper} because $\Phi_\d(D)=1$.
\end{proof}

We now apply these estimates to the family in \cref{thm:whole-space-persistence}, choosing a time interval on which each source remains a fixed fraction of its initial distance from the target.

\begin{proof}[Proof of \cref{thm:whole-space-persistence}]
The constants in \cref{lem:uniform-velocity-support} may be chosen uniformly in $n$, because the initial $L^\infty$ norms have the common bound $M_0$. Let $V$ be such a common velocity bound and set
\begin{equation}
c_0:=\frac{1}{2(1+V)}.
\end{equation}
Then \eqref{eq:support-distance-propagation} implies
\begin{equation}
\dist(\supp\rho_{n,t},\supp\mu)
\ge D_n-Vt\ge\frac{D_n}{2}
\qquad (0\le t\le c_0D_n).
\end{equation}
For all sufficiently large $n$, \cref{lem:annular-mmd-lower} therefore gives
\begin{equation}\label{eq:persistence-annular-evaluation}
\mmd^2(\rho_{n,t},\mu)
\ge \frac{1}{2\sigma_\d}\int_{R_\mu}^{D_n/2}r^{1-\d}\,dr.
\end{equation}
The right-hand side is bounded below by $c_1D_n$ if $\d=1$, by $c_1\log D_n$ if $\d=2$, and by a positive constant $c_1$ if $\d\ge3$, after increasing $n_0$ if necessary. This proves the first inequality in \eqref{eq:persistence-comparison}.

The last inequality in \eqref{eq:persistence-comparison} follows from \cref{lem:localized-initial-upper} with $M=M_0$ and $L=L_0$. Since the data are compactly supported and bounded, the energy inequality \eqref{eq:weak-energy-inequality} extends to the initial time, as explained after that display. Hence
\begin{equation}
\mmd^2(\rho_{n,t},\mu)\le\mmd^2(\rho_{0,n},\mu),
\end{equation}
which completes \eqref{eq:persistence-comparison}. Dividing the first and last bounds gives \eqref{eq:persistence-fraction}.
\end{proof}

\subsection{Consequences for convergence rates and PL inequalities}

The persistence estimate has two immediate but logically distinct consequences. Applying it to translations of a fixed profile rules out a decay modulus uniform over the source location; averaging the dissipation over the same travel-time interval produces states whose dissipation-to-energy ratio tends to zero, ruling out global PL coercivity.

\begin{proof}[Proof of \cref{cor:no-universal-rate}]
Choose translations $a_n$ such that
\begin{equation}
D_n:=\dist(\supp\theta(\cdot-a_n),\supp\mu)\longrightarrow\infty.
\end{equation}
The translated data have a common $L^\infty$ norm and support diameter, so \cref{thm:whole-space-persistence} applies. Set $T_n=c_0D_n$. If the asserted estimate held, then \eqref{eq:persistence-fraction} would imply
\begin{equation}
\eta\,\mmd^2(\rho_{0,a_n},\mu)
\le \mmd^2(\rho_{a_n,T_n},\mu)
\le \gamma(T_n)\mmd^2(\rho_{0,a_n},\mu).
\end{equation}
Since the initial energy is positive, $\gamma(T_n)\ge\eta$. Since $T_n\to\infty$, this contradicts $\gamma(t)\to0$.
\end{proof}

\begin{proof}[Proof of \cref{cor:dynamical-global-PL}]
Let $T_n=c_0D_n$ and let $\eta>0$ be given by \eqref{eq:persistence-fraction}. The energy inequality from the initial time gives
\begin{equation}
\int_0^{T_n}\dissip(\rho_{n,t}\mid\mu)\,dt
\le \mmd^2(\rho_{0,n},\mu)-\mmd^2(\rho_{n,T_n},\mu)
\le (1-\eta)\mmd^2(\rho_{0,n},\mu).
\end{equation}
Consequently, there exists a time $t_n\in(0,T_n)$ such that
\begin{equation}
\dissip(\rho_{n,t_n}\mid\mu)
\le \frac{1-\eta}{T_n}\mmd^2(\rho_{0,n},\mu).
\end{equation}
Using again \eqref{eq:persistence-fraction},
\begin{equation}
\frac{\dissip(\rho_{n,t_n}\mid\mu)}{\mmd^2(\rho_{n,t_n},\mu)}
\le \frac{1-\eta}{\eta c_0D_n},
\end{equation}
which proves \eqref{eq:dynamical-PL-ratio}. The maximum principle gives a common $L^\infty$ bound on all the selected states, while \cref{lem:uniform-velocity-support} keeps them compactly supported. They therefore contradict any global PL inequality on the stated class.
\end{proof}

\begin{remark}\label{rem:scope-whole-space-obstructions}
The preceding results do not rule out convergence for a fixed initial datum, nor rates whose constants or waiting times depend on the support location, moments, or other localization information. In fact, \cref{thm:planar-convergence} shows that in dimension two every solution constructed here whose Coulomb energy is finite at some positive time converges to the target narrowly, weak-$*$ in $L^\infty$, and strongly in $H^{-\alpha}$ for every $\alpha>0$. There is no contradiction: the translated planar sources used above have Coulomb energy of order $\log D$ and therefore do not remain in a fixed energy sublevel as their separation tends to infinity. The obstructions also do not rule out a nonradial PL inequality under the common connected-support hypotheses of \cref{thm:connected-support-PL}. They show that any unrestricted whole-space statement must account for the time needed to transport mass across a large spatial separation.
\end{remark}

\subsection{A direct static PL obstruction}

The dynamical argument above produces bad states along trajectories that remain remote from the target. We now prove the same failure directly by constructing explicit smooth test densities, without invoking the flow.

\begin{proof}[Proof of \cref{prop:static-global-PL}]
Fix a nonnegative probability density $\theta\in C_c^\infty(B_1)$. We treat the three Coulomb regimes separately.

\smallskip
\noindent\emph{Dimension $\d=1$.}
For $R>0$, set $\rho_R(x)=\theta(x-R)$. Since $\g(x)=-|x|/2$ and the target is compactly supported, the two self-interaction terms in \eqref{eq:energy-interaction-decomposition} are independent of $R$, while
\begin{equation}
-\iint\g(x-y)\rho_R(x)\mu(y)\,dx\,dy
=\frac12\iint|x-y|\rho_R(x)\mu(y)\,dx\,dy
\ge cR
\end{equation}
for all sufficiently large $R$. Hence $\mmd^2(\rho_R,\mu)\ge cR$. On the other hand, $\|\g'\|_{L^\infty}=1/2$, so
\begin{equation}
\dissip(\rho_R\mid\mu)
\le \|\g'\ast(\rho_R-\mu)\|_{L^\infty}^2\le1.
\end{equation}
The ratio in \eqref{eq:static-PL-ratio} therefore tends to zero.

\smallskip
\noindent\emph{Dimension $\d=2$.}
Let $e_1$ be the first coordinate vector and set $\rho_R(x)=\theta(x-Re_1)$. Uniformly for $x\in\supp\rho_R$ and $y\in\supp\mu$,
\begin{equation}
\log|x-y|=\log R+O(R^{-1}).
\end{equation}
The source self-interaction is independent of $R$ and the target self-interaction is fixed, so \eqref{eq:energy-interaction-decomposition} yields
\begin{equation}
\mmd^2(\rho_R,\mu)=\frac{1}{2\pi}\log R+O(1).
\end{equation}
The source self-field on $\supp\rho_R$ is a translate of the bounded field $\nabla\g\ast\theta$ on $B_1$, whereas
\begin{equation}
\sup_{x\in\supp\rho_R}|\nabla\g\ast\mu(x)|\le \frac{C}{R}.
\end{equation}
Thus, $\dissip(\rho_R\mid\mu)\le C$, and the ratio again tends to zero.

\smallskip
\noindent\emph{Dimensions $\d\ge3$.}
A pure translation does not make the source self-field small in dimensions $\d\ge3$. We therefore also dilate the profile while placing its center at distance of order $r^2$ from the target.
For $r\ge2$, put $a_r=r^2e_1$ and
\begin{equation}
\rho_r(x):=r^{-\d}\theta\left(\frac{x-a_r}{r}\right).
\end{equation}
By scaling,
\begin{align}
\iint\g(x-y)\rho_r(x)\rho_r(y)\,dx\,dy
&=r^{2-\d}\iint\g(x-y)\theta(x)\theta(y)\,dx\,dy,\label{eq:static-self-energy-scaling}\\
\int|\nabla\g\ast\rho_r|^2\rho_r\,dx
&=r^{2-2\d}\int|\nabla\g\ast\theta|^2\theta\,dx.\label{eq:static-self-diss-scaling}
\end{align}
Moreover, the distance from $\supp\rho_r$ to $\supp\mu$ is comparable to $r^2$, and therefore
\begin{equation}
\sup_{x\in\supp\rho_r}\left(|\g\ast\mu(x)|+|\nabla\g\ast\mu(x)|\right)
\le C\left(r^{4-2\d}+r^{2-2\d}\right).
\end{equation}
It follows from \eqref{eq:energy-interaction-decomposition} and \eqref{eq:static-self-energy-scaling} that
\begin{equation}
\mmd^2(\rho_r,\mu)
\longrightarrow \frac12\iint\g(x-y)\mu(x)\mu(y)\,dx\,dy>0.
\end{equation}
Using \eqref{eq:static-self-diss-scaling} and the target-field estimate,
\begin{equation}
\dissip(\rho_r\mid\mu)
\le 2r^{2-2\d}\int|\nabla\g\ast\theta|^2\theta\,dx
+Cr^{4-4\d}
\longrightarrow0.
\end{equation}
This completes the proof in every dimension.
\end{proof}

\begin{remark}\label{rem:compare-static-sharpness}
The static construction above is different from \cref{prop:sharpii}. That proposition is a bounded-region, radial counterexample in which $\mmd(\rho_\delta,\mu)\to0$. By contrast, \cref{prop:static-global-PL} is an obstruction at spatial infinity, applies to every compactly supported bounded target, and does not require $\mmd(\rho_n,\mu)\to0$. The dynamical statement \cref{cor:dynamical-global-PL} adds the travel-time interpretation absent from both static constructions.
\end{remark}

\section{Criticality and planar convergence}\label{sec:criticality-planar}

The whole-space obstructions in the preceding section rule out decay rates and PL constants that are uniform over unrestricted initial data, but they do not determine the asymptotic behavior of a fixed solution. We next prove the rigidity result \cref{thm:lagrangian-rigidity} and use it, together with planar tightness, to establish \cref{thm:planar-convergence}. We conclude with a static example showing why these compactness arguments do not by themselves yield convergence in the Coulomb topology $\dot H^{-1}$.

\subsection{Rigidity under absolute continuity of the positive part}

Whenever the Coulomb first variation is defined, the Lagrangian criticality condition for $\mmd^2(\cdot,\mu)$ is
\begin{equation}\label{eq:lagrangian-criticality}
\nabla\g\ast(\rho-\mu)=0
\qquad \rho\text{-a.e.}
\end{equation}
For a singular source, this condition requires a specified representative of the field. The analytic statement in \cref{thm:lagrangian-rigidity} only requires the field to vanish almost everywhere with respect to $(\rho-\mu)^+$. Since the theorem assumes that $(\rho-\mu)^+$ is absolutely continuous with respect to Lebesgue measure, this weaker condition is independent of the representative on Lebesgue-null sets. On $\R^2$, a literal logarithmic convolution may require a logarithmic moment. This is why the theorem is formulated through the distributional Poisson equation for a Coulomb potential.

\begin{proof}[Proof of \cref{thm:lagrangian-rigidity}]
Write $(\rho-\mu)^+=f\,dx$ and set $A:=\{f>0\}$. By the mutual singularity of the positive and negative parts of $\rho-\mu$, the absolutely continuous part of $(\rho-\mu)^-$ vanishes on $A$. The condition in \eqref{eq:lagrangian-critical-vanishing} and the positivity of $f$ on $A$ give
\begin{equation}\label{eq:critical-gradient-on-positive-set}
\nabla h=0
\qquad \mathcal{L}^\d\text{-a.e. on }A.
\end{equation}

Suppose first that $\d\ge2$. By \cite[Theorem~1.1]{APR20}, the absolutely continuous part of the distributional Laplacian of a locally integrable function vanishes almost everywhere on each vector level set of its gradient. On the other hand, the Poisson equation and the Jordan decomposition give
\begin{equation}
\bigl(\Delta h\bigr)^{\mathrm a}
=-f\,dx
\qquad\text{on }A.
\end{equation}
Applying the cited result to the level set $\{\nabla h=0\}$, locally after periodic lifting when $\Omega=\T^\d$, and using \eqref{eq:critical-gradient-on-positive-set}, we conclude that $f=0$ almost everywhere on $A$. Thus $(\rho-\mu)^+=0$.

If $\d=1$, set $v:=\partial_xh$. Since
\begin{equation}
Dv=(\rho-\mu)^--(\rho-\mu)^+
\end{equation}
is a locally finite signed measure, $v\in BV_{\mathrm{loc}}$. By \cite[Theorem~3.83]{AFP00}, $v$ is approximately differentiable\footnote{Following \cite{APR20} (cf. \cite[Definition~3.70]{AFP00}), a measurable function $v$ is approximately differentiable at $x$ if there exist $\widetilde v(x),a\in\R$ such that, for every $\varepsilon>0$,
\begin{equation}
\lim_{r\downarrow0}
\frac{
\mathcal{L}^1\bigl(
\{y\in B_r(x):
 |v(y)-\widetilde v(x)-a(y-x)|
 >\varepsilon|y-x|\}
\bigr)}
{\mathcal{L}^1(B_r(x))}
=0.
\end{equation}
The number $a$, denoted by $D_{\mathrm{ap}}v(x)$, is the approximate derivative of $v$ at $x$, while $\widetilde v(x)$ is its approximate limit.} at $\mathcal{L}^1$-almost every point, and
\begin{equation}
(Dv)^{\mathrm a}=D_{\mathrm{ap}}v\,\mathcal{L}^1.
\end{equation}
Moreover, for every $c\in\R$,
\begin{equation}
D_{\mathrm{ap}}v=0
\qquad
\mathcal{L}^1\text{-a.e. on }\{v=c\};
\end{equation}
see \cite[Theorem~6.3]{EG15} and also \cite[Section~2, equation~(2.1)]{APR20}. Since $(Dv)^{\mathrm a}=-f\,dx$ on $A$, applying this locality property on $\{v=0\}$ and using \eqref{eq:critical-gradient-on-positive-set} again gives $f=0$ almost everywhere on $A$. Thus $(\rho-\mu)^+=0$ also in this case.

In every dimension, we therefore have
\begin{equation}\label{eq:critical-measure-domination}
\mu-\rho=(\rho-\mu)^-\ge0.
\end{equation}
Its total mass is zero because both measures are probabilities, so it vanishes identically.
\end{proof}

\begin{remark}[Variational scope]\label{rem:variational-lagrangian-criticality}
In the usual finite-energy setting one takes $h=\g\ast(\rho-\mu)$, up to an irrelevant additive constant. Every Lagrangian critical point satisfies the weaker vanishing condition in \eqref{eq:lagrangian-critical-vanishing}, since $(\rho-\mu)^+\le\rho$. Thus \cref{thm:lagrangian-rigidity} accommodates an arbitrary probability target and any source for which $(\rho-\mu)^+$ is absolutely continuous whenever $h=\g\ast(\rho-\mu)$ is well defined. By contrast, the distributional stationarity equation
\begin{equation}
\div\bigl(\rho\nabla h\bigr)=0
\end{equation}
is weaker than Lagrangian criticality and is not covered by the theorem without an additional tangent-space or energy argument.
\end{remark}

\begin{remark}[Related critical-point results]
\label{rem:BV-critical-points}
Boufad\`ene and Vialard distinguish Lagrangian critical points from the stronger notion of a Wasserstein critical point, defined by requiring the source itself to minimize every sufficiently small-step JKO problem. Their Proposition~3.3 shows that a Wasserstein critical point has $0$ in its Wasserstein subdifferential; combined with their characterization of the Coulomb subgradient, this implies Lagrangian criticality \cite[Lemma~2.4, Definitions~3.1--3.2, and Proposition~3.3]{BV25}. Their Theorem~3.4 proves, in the finite-energy Coulomb setting, that a Lagrangian critical point agrees with the target on the interior of its support \cite[Theorem~3.4]{BV25}. This does not yield global equality even for an absolutely continuous source, since a positive-measure closed set may have empty interior. \Cref{thm:lagrangian-rigidity} closes this gap more generally whenever $(\rho-\mu)^+$ is absolutely continuous; in particular, it includes densities supported on fat-Cantor-type sets.

Boufad\`ene and Vialard also obtain restrictions on singular Wasserstein critical points in the Euclidean finite-Coulomb-energy setting. Write
\begin{equation}
\rho-\mu=(\rho-\mu)^+-(\rho-\mu)^-
\end{equation}
for the Jordan decomposition and set $\sigma:=(\rho-\mu)^+$. Their Lemma~4.9 shows that $\rho$ is not Wasserstein critical if there exist a set $A$ with $\sigma(A)>0$, a number $0<\delta<2$, and constants $C,r_0>0$ such that
\begin{equation}
\sigma(B(x,r))\ge Cr^{\d-\delta}
\end{equation}
for $\sigma$-almost every $x\in A$ and every $0<r\le r_0$ \cite[Lemma~4.9]{BV25}. Their Remark~4.11 records, in particular, that the criterion applies when $\sigma$ is absolutely continuous with respect to surface measure on a $(\d-1)$-dimensional submanifold \cite[Remark~4.11]{BV25}. Thus their result excludes certain sufficiently lower-dimensional cases for $\sigma$ under the stronger JKO-based notion, whereas \cref{thm:lagrangian-rigidity} excludes every Lagrangian critical point for which $(\rho-\mu)^+$ is absolutely continuous. Configurations for which $(\rho-\mu)^+$ has a singular component remain outside this combined classification.
\end{remark}

\begin{remark}\label{rem:singular-lagrangian-criticality}
The hypothesis $(\rho-\mu)^+\ll dx$ in \cref{thm:lagrangian-rigidity} cannot be omitted without further assumptions. For example, on $\R$ let
\begin{equation}
\rho=\delta_0,
\qquad
\mu=\frac12\bigl(\delta_{-1}+\delta_1\bigr).
\end{equation}
For the one-dimensional Coulomb kernel $\g(x)=-|x|/2$, the symmetric precise representative of the corresponding field vanishes at the point charged by $\rho$, although $\rho\ne\mu$. Here $(\rho-\mu)^+=\delta_0$ is singular. A theory for Lagrangian critical points for which $(\rho-\mu)^+$ is singular must therefore specify the representative of the field and incorporate capacity or dimension information.
\end{remark}

\subsection{Logarithmic capacity and tightness}
\label{ssec:planar-log-capacity}

We next specialize to $\R^2$, where
\begin{equation}
\g(x)=-\frac{1}{2\pi}\log|x|.
\end{equation}
The vanishing logarithmic capacity of infinity provides a tightness estimate from the $L^2$ norm of the Coulomb field.

\begin{lemma}[Logarithmic-capacity tightness]\label{lem:planar-log-tightness}
Let $\rho,\mu\in\P(\R^2)$, and suppose that $E\in L^2(\R^2;\R^2)$ satisfies
\begin{equation}
\div E=\rho-\mu
\quad\text{in }\mathcal D'(\R^2).
\end{equation}
Then, for every $R>1$,
\begin{equation}\label{eq:planar-log-tightness}
\rho(B_{R^2}^c)
\le
\mu(B_R^c)
+\left(\frac{2\pi}{\log R}\right)^{1/2}\|E\|_{L^2}.
\end{equation}
In particular, if $E=-\nabla\g\ast(\rho-\mu)$ and the Coulomb energy is finite, then
\begin{equation}\label{eq:planar-log-tightness-mmd}
\rho(B_{R^2}^c)
\le
\mu(B_R^c)
+2\sqrt{\frac{\pi\mmd^2(\rho,\mu)}{\log R}}.
\end{equation}
\end{lemma}

\begin{proof}
Define
\begin{equation}
\eta_R(x):=
\begin{cases}
0, & |x|\le R,\\[1mm]
\dfrac{\log(|x|/R)}{\log R}, & R<|x|<R^2,\\[2mm]
1, & |x|\ge R^2.
\end{cases}
\end{equation}
A direct computation gives
\begin{equation}\label{eq:log-cutoff-dirichlet}
\int_{\R^2}|\nabla\eta_R|^2\,dx
=
\frac{2\pi}{\log R}.
\end{equation}
Since $\rho$ and $\mu$ have equal mass, the first integral below is unchanged if $\eta_R$ is replaced by $\eta_R-1$. Since $1-\eta_R$ is compactly supported, we may integrate by parts and use Cauchy-Schwarz to obtain
\begin{equation}
\left|\int_{\R^2}\eta_R\,d(\rho-\mu)\right| = \left|\int_{\R^2}(\eta_R-1)\div E\,dx\right| = \left|\int_{\R^2}E\cdot\nabla\eta_R\,dx\right| \le \left(\frac{2\pi}{\log R}\right)^{1/2}\|E\|_{L^2}.
\end{equation}
Moreover, $\eta_R=1$ on $B_{R^2}^c$, while $\eta_R=0$ on $B_R$ and $0\le\eta_R\le1$. Hence
\begin{equation}
\rho(B_{R^2}^c) \le \int\eta_R\,d\rho \le \int\eta_R\,d\mu + \left|\int\eta_R\,d(\rho-\mu)\right| \le \mu(B_R^c) + \left(\frac{2\pi}{\log R}\right)^{1/2}\|E\|_{L^2},
\end{equation}
which proves \eqref{eq:planar-log-tightness}. The identity $\|E\|_{L^2}^2=2\mmd^2(\rho,\mu)$ gives \eqref{eq:planar-log-tightness-mmd}.
\end{proof}

If $\rho$ is a solution for which the energy inequality \eqref{eq:weak-energy-inequality} holds after some $s_0>0$, then \cref{lem:planar-log-tightness} gives
\begin{equation}\label{eq:planar-orbit-tightness}
\sup_{t\ge s_0}\rho_t(B_{R^2}^c)
\le
\mu(B_R^c)
+2\sqrt{\frac{\pi\mmd^2(\rho_{s_0},\mu)}{\log R}}
\longrightarrow0
\qquad\text{as }R\to\infty.
\end{equation}
Thus a finite-energy planar trajectory is uniformly tight without any moment assumption. This does not contradict the translated-source construction in \cref{sec:whole-space-obstructions}: in dimension two the energy of a fixed positive amount of mass translated to distance $R$ grows like $\log R$, so those examples do not remain in a fixed energy sublevel.

\subsection{Planar convergence}

The proof follows a standard LaSalle-type $\omega$-limit argument for Wasserstein gradient flows; see, e.g., \cite{CGW23}. We include the details needed to verify the compactness and rigidity inputs in the present whole-space planar Coulomb setting.

\begin{proof}[Proof of \cref{thm:planar-convergence}]
For $t\ge s_0$, abbreviate
\begin{equation}
F(t):=\mmd^2(\rho_t,\mu),
\qquad
D(t):=\dissip(\rho_t\mid\mu).
\end{equation}
By \eqref{eq:weak-energy-inequality}, $F$ has a nonincreasing representative and
\begin{equation}\label{eq:planar-tail-dissipation}
F(t)+\int_s^tD(\tau)\,d\tau\le F(s),
\qquad t\ge s\ge s_0.
\end{equation}
In particular, $F(t)$ converges to a finite limit. The tightness estimate \eqref{eq:planar-orbit-tightness} shows that $\{\rho_t:t\ge s_0\}$ is uniformly tight. Also, \cref{thm:Cauchy} gives
\begin{equation}\label{eq:planar-uniform-density}
M:=\sup_{t\ge s_0}\|\rho_t\|_{L^\infty}<\infty.
\end{equation}

Let $t_n\to\infty$ be arbitrary. By tightness, after passing to a subsequence,
\begin{equation}\label{eq:planar-omega-limit}
\rho_{t_n}\rightharpoonup\bar\rho
\end{equation}
narrowly for some $\bar\rho\in\P(\R^2)$. The bound \eqref{eq:planar-uniform-density} passes to the limit, so $\bar\rho$ is absolutely continuous and, identifying it with its density,
\begin{equation}
0\le\bar\rho\le M
\qquad\text{a.e.}
\end{equation}

Since $F$ has a limit, \eqref{eq:planar-tail-dissipation} gives
\begin{equation}\label{eq:unit-interval-dissipation}
\int_{t_n}^{t_n+1}D(t)\,dt
\le F(t_n)-F(t_n+1)
\longrightarrow0.
\end{equation}
Choose a Lebesgue point $s_n\in[t_n,t_n+1]$ of $D$ such that
\begin{equation}\label{eq:good-times}
D(s_n)\le\int_{t_n}^{t_n+1}D(t)\,dt\longrightarrow0.
\end{equation}
For every $\varphi\in C_c^1(\R^2)$, the continuity equation and Cauchy--Schwarz give
\begin{equation}
\left|\int\varphi\,d(\rho_{s_n}-\rho_{t_n})\right| \le \|\nabla\varphi\|_{L^\infty}\int_{t_n}^{s_n}\int_{\R^2}|E_t|\,d\rho_t\,dt \le \|\nabla\varphi\|_{L^\infty}\left(\int_{t_n}^{t_n+1}D(t)\,dt\right)^{1/2},
\end{equation}
where
\begin{equation}
E_t:=-\nabla\g\ast(\rho_t-\mu).
\end{equation}
It follows from \eqref{eq:unit-interval-dissipation} and uniform tightness that
\begin{equation}\label{eq:good-times-same-limit}
\rho_{s_n}\rightharpoonup\bar\rho
\end{equation}
narrowly.

Set $E_n:=E_{s_n}$. The uniform $L^1\cap L^\infty$ bounds for $\rho_{s_n}-\mu$ and the standard near--far decomposition yield
\begin{equation}
\sup_n\|E_n\|_{L^\infty}<\infty.
\end{equation}
For every $1<p<\infty$, Calder\'on--Zygmund estimates similarly give
\begin{equation}
\sup_n\|\nabla E_n\|_{L^p}<\infty.
\end{equation}
Choosing $p>2$ and using Morrey compactness, we may pass to a further subsequence such that
\begin{equation}\label{eq:omega-field-local-uniform}
E_n\longrightarrow\bar E
\qquad\text{locally uniformly on }\R^2.
\end{equation}

We identify the limit field. Let $\Phi\in C_c^\infty(\R^2;\R^2)$. Then
\begin{equation}
\int_{\R^2}E_n(x)\cdot\Phi(x)\,dx
=
\int_{\R^2}H_\Phi(y)\,d(\rho_{s_n}-\mu)(y),
\end{equation}
where
\begin{equation}
H_\Phi(y):=-\int_{\R^2}\nabla\g(x-y)\cdot\Phi(x)\,dx.
\end{equation}
The kernel $-\nabla\g$ is locally integrable and decays like $|y|^{-1}$ at infinity, so $H_\Phi\in C_0(\R^2)$. Therefore \eqref{eq:good-times-same-limit} implies
\begin{equation}\label{eq:omega-field-identification}
\bar E=-\nabla\g\ast(\bar\rho-\mu),
\qquad
\div\bar E=\bar\rho-\mu
\end{equation}
in the sense of distributions.

For every nonnegative $\chi\in C_c^\infty(\R^2)$, the locally uniform field convergence and narrow convergence of the measures give
\begin{equation}
\int_{\R^2}\chi|\bar E|^2\,d\bar\rho = \lim_{n\to\infty}\int_{\R^2}\chi|E_n|^2\,d\rho_{s_n} \le \|\chi\|_{L^\infty}\lim_{n\to\infty}D(s_n)=0.
\end{equation}
Thus
\begin{equation}\label{eq:omega-lagrangian-critical}
\bar E=0
\qquad \bar\rho\text{-a.e.}
\end{equation}
Since $\bar\rho,\mu\in L^1(\R^2)\cap L^\infty(\R^2)$ and $\g$ is even, define the normalized Coulomb potential
\begin{equation}
\bar h(x):=\int_{\R^2}\bigl(\g(x-y)-\g(y)\bigr)\,d(\bar\rho-\mu)(y).
\end{equation}
The integral is well defined and locally bounded: the logarithmic singularities are locally integrable, while, for every $R>0$, one has $\g(x-y)-\g(y)=O_R(|y|^{-1})$ as $|y|\to\infty$, uniformly for $x\in B_R$. Distributionally,
\begin{equation}
\nabla\bar h=\nabla\g\ast(\bar\rho-\mu)=-\bar E,
\qquad
\Delta\bar h=\mu-\bar\rho.
\end{equation}
Since $\mu-\bar\rho\in L^\infty(\R^2)$, Calder\'on--Zygmund estimates give $\bar h\in W^{2,p}_{\mathrm{loc}}(\R^2)$ for every $1<p<\infty$, and in particular $\bar h\in W^{1,1}_{\mathrm{loc}}(\R^2)$. By \eqref{eq:omega-lagrangian-critical}, $\nabla\bar h=0$ almost everywhere with respect to $\bar\rho$, hence also with respect to $(\bar\rho-\mu)^+\le\bar\rho$. Since $(\bar\rho-\mu)^+$ is absolutely continuous, \cref{thm:lagrangian-rigidity} gives $\bar\rho=\mu$.

We have proved that every sequence $t_n\to\infty$ has a subsequence along which $\rho_{t_n}$ converges narrowly to $\mu$. Uniform tightness then implies the full convergence \eqref{eq:planar-narrow-convergence}.

To prove \eqref{eq:planar-weak-star-convergence}, let $f\in L^1(\R^2)$ and choose $\varphi\in C_c(\R^2)$ with $\|f-\varphi\|_{L^1}$ arbitrarily small. Then
\begin{equation}
\left|\int_{\R^2}f(\rho_t-\mu)\,dx\right| \le \left|\int_{\R^2}\varphi(\rho_t-\mu)\,dx\right| + \left(M+\|\mu\|_{L^\infty}\right)\|f-\varphi\|_{L^1}.
\end{equation}
The first term tends to zero by narrow convergence, proving weak-$*$ convergence in $L^\infty$.

Finally, set $f_t:=\rho_t-\mu$. Narrow convergence implies
\begin{equation}
\widehat f_t(\xi)\longrightarrow0
\qquad\text{for every }\xi\in\R^2,
\end{equation}
while the mass and density bounds give
\begin{equation}
\sup_{t\ge s_0}\|f_t\|_{L^2}<\infty.
\end{equation}
Fix $\alpha>0$ and $R>0$. Since $|\widehat f_t|\le2$, dominated convergence gives
\begin{equation}
\int_{|\xi|\le R}(1+|\xi|^2)^{-\alpha}
|\widehat f_t(\xi)|^2\,d\xi
\longrightarrow0.
\end{equation}
By Plancherel,
\begin{equation}
\int_{|\xi|>R}(1+|\xi|^2)^{-\alpha}|\widehat f_t(\xi)|^2\,d\xi \le (1+R^2)^{-\alpha}\|\widehat f_t\|_{L^2}^2 \le C(1+R^2)^{-\alpha},
\end{equation}
uniformly for $t\ge s_0$. Sending first $t\to\infty$ and then $R\to\infty$ proves \eqref{eq:planar-negative-sobolev-convergence}.
\end{proof}

\begin{remark}[Failure of compactness at the Coulomb endpoint]\label{rem:planar-coulomb-endpoint}
The convergences asserted in \cref{thm:planar-convergence} do not imply convergence in MMD\@. To see this, let $\mu,\theta\in C_c^\infty(\R^2)$ be probability densities. For $R>e$, put
\begin{equation}
\theta_R(x):=\theta(x-Re_1),
\qquad
\varepsilon_R:=(\log R)^{-1/2},
\end{equation}
where $e_1=(1,0)$, and define
\begin{equation}
\rho_R:=(1-\varepsilon_R)\mu+\varepsilon_R\theta_R.
\end{equation}
Then $(\rho_R)_R$ is uniformly bounded in $L^\infty$, is tight, and satisfies
\begin{equation}
\rho_R\rightharpoonup\mu,
\qquad
\rho_R\stackrel{*}{\rightharpoonup}\mu\text{ in }L^\infty,
\qquad
\|\rho_R-\mu\|_{H^{-\alpha}}\longrightarrow0
\quad(\alpha>0).
\end{equation}
On the other hand,
\begin{equation}
\rho_R-\mu=\varepsilon_R(\theta_R-\mu).
\end{equation}
The compact supports and the expansion $\log|Re_1+x-y|=\log R+O(R^{-1})$ give
\begin{equation}
\mmd^2(\theta_R,\mu)
=
\frac{1}{2\pi}\log R+O(1).
\end{equation}
Consequently,
\begin{equation}
\mmd^2(\rho_R,\mu)
=
\varepsilon_R^2\mmd^2(\theta_R,\mu)
\longrightarrow\frac{1}{2\pi}.
\end{equation}
Thus a vanishing amount of mass at an increasingly remote scale can retain a nonzero amount of logarithmic Coulomb energy while disappearing in every topology asserted in \cref{thm:planar-convergence}. 
Obviously, this static sequence does not show that a particular solution fails to converge in MMD; but it does show that narrow convergence, uniform $L^\infty$ control, tightness, and bounded Coulomb energy alone do not force such convergence.
\end{remark}

\section{Open problems and future directions}\label{sec:future-directions}

The preceding results give a Cauchy theory for the Coulomb MMD flow, establish quantitative relaxation under several coercivity hypotheses, classify Lagrangian critical points for which $(\rho-\mu)^+$ is absolutely continuous, and prove qualitative planar convergence while exhibiting obstructions to uniform decay and deterioration of strong spatial regularity. Nevertheless, they leave open questions concerning selection, sharp coercivity, singular critical points, stronger convergence topologies, rates, and regularity. We record several of them below. 

\medskip\noindent\textbf{Measure initial data, singular targets, and canonical selection.}\quad
\Cref{thm:Cauchy} constructs a global weak solution from an arbitrary initial probability measure by smooth approximation, but uniqueness is proved only when the initial datum is essentially bounded. It would be useful to determine whether the equation defines a canonical semigroup on all of $\P(\Omega)$, and in particular whether the approximation limit is independent of the regularization procedure before one enters the bounded uniqueness class. Closely related is the comparison between solutions selected by smooth approximation and minimizing movements of JKO type. The assumption $\mu\in L^\infty$ is also structural in the present Cauchy theory: it enters the ultracontractive estimate and the compactness argument. In dimension one, quantile methods already provide a Wasserstein gradient-flow theory for arbitrary targets in $\P_2(\R)$, including discrete targets \cite{DuongSteinBeinertHertrichSteidl26}. In dimensions $\d\ge2$, extending the present Eulerian theory to singular targets, especially empirical measures, would require a suitable interpretation of the singular target field and of the dissipation at source--target collisions. A general higher-dimensional result should identify a natural class in which existence, stability, and selection are compatible.

\medskip\noindent\textbf{Sharp torus PL constants and optimal decay rates.}\quad
For the uniform target, \cref{thm:bounded-contrast-PL} gives the global-in-the-source constant $2/\d$. For a general target satisfying the density-ratio condition, the same theorem gives the constant
\begin{equation}
2\left(m-\frac{\d-1}{\d}M\right)
\end{equation}
when $m=\operatorname*{ess\,inf}\mu$ and $M=\operatorname*{ess\,sup}\mu$ satisfy the stated condition. We do not know whether the uniform-target constant or the admissible density-ratio range is sharp. It is natural to determine the optimal constant for the uniform target and the largest target class for which a global-in-the-source PL inequality holds. Outside the density-ratio regime, \cref{thm2} still gives every squared-MMD decay exponent strictly smaller than $2m/3$ for arbitrary measure initial data after any positive waiting time. What is the optimal decay rate uniform over arbitrary initial probability measures, and can it depend usefully on target regularity, geometry, or quantitative non-concentration? In dimension one, the periodic identity gives the global constant $2m$ for every uniformly positive target, so these questions principally concern dimensions at least two.

\medskip\noindent\textbf{Degenerate targets and non-PL relaxation on the torus.}\quad
For the Coulomb endpoint, Chizat, Colombo, Colombo, and Fern\'andez-Real \cite{CCCFR26} prove global weak-$*$ convergence in the bounded-density class even when the target is not uniformly positive. By \cref{prop:counterexvanishingmu}, a global PL inequality may fail even when the target vanishes at only one point. How are the relaxation rate and the topology of convergence determined by the order of vanishing, the dimension and geometry of the zero set, and the regularity of the target? Possible mechanisms include local \L{}ojasiewicz-type inequalities with an exponent different from the PL one, algebraic or stretched-exponential decay, and a distinction between convergence in MMD, Wasserstein, and uniform topologies. A sharp theory should connect these alternatives to quantitative features of the target degeneracy.

\medskip\noindent\textbf{Nonradial coercivity and confinement on the whole space.}\quad
The radial proof of \cref{thm:connected-support-PL} uses one-dimensional shell ordering to provide two ingredients at once: a static coercivity inequality when the target has connected support containing the source support, and a confinement principle which preserves this source-support inclusion along the flow. A basic question is whether either conclusion has a genuinely nonradial analogue. One may seek geometric hypotheses on the target support, such as convexity, star-shapedness, or quantitative boundary regularity, under which a PL inequality can be proved and the source remains confined. The two issues should be separated: a static inequality need not imply that its hypotheses are dynamically invariant. The counterexamples in \cref{sec:whole-space-obstructions} do not settle this problem, since they exploit sources placed arbitrarily far from the target and therefore violate source-support inclusion. Understanding the nonradial problem may require a replacement for shell ordering which combines a coercivity estimate controlling the Coulomb energy by the dissipation with a geometric confinement argument for the transported support.

\medskip\noindent\textbf{Fixed-data asymptotics after the support-travel-time obstruction.}\quad
\Cref{thm:whole-space-persistence} rules out convergence rates uniform over the location of a localized initial source, whereas \cref{thm:planar-convergence} proves narrow, weak-$*$ in $L^\infty$, and strong $H^{-\alpha}$ convergence for every solution constructed here on $\R^2$ whose Coulomb energy is finite at some positive time. The planar result does not establish convergence in the Coulomb MMD or provide a rate. Can one exclude the vanishing-mass, remote-energy mechanism along an actual trajectory and prove convergence in $\dot H^{-1}$? Which tail-energy, logarithmic-integrability, localization, or moment quantities would suffice, and how would they control a rate? To the best of our knowledge, in the selected weak-solution class of \cref{thm:Cauchy}, even narrow convergence for one fixed solution remains open in dimensions at least three; the Coulomb energy does not provide the logarithmic confinement mechanism used in dimension two.

\medskip\noindent\textbf{Singular Lagrangian critical points.}\quad
\Cref{thm:lagrangian-rigidity} classifies every Lagrangian critical point for which $(\rho-\mu)^+$ is absolutely continuous. As \cref{rem:singular-lagrangian-criticality} shows, nontrivial configurations with $(\rho-\mu)^+$ singular can already occur once a representative of the field is specified. The results of Boufad\`ene and Vialard discussed in \cref{rem:BV-critical-points} exclude certain lower-dimensional cases for $(\rho-\mu)^+$ under their stronger Wasserstein-critical notion, but full-dimensional singular cases and some mixed absolutely continuous--singular configurations remain unclassified. A complete theory should characterize Lagrangian critical points for which $(\rho-\mu)^+$ is singular, distinguish them from blocked-JKO critical points, and identify the capacity-sensitive representative of the Coulomb field appropriate on sets charged by singular sources.

\medskip\noindent\textbf{Fine-scale formation and optimal regularity propagation.}\quad
The construction in \cref{prop:regularitycounterex} shows that instantaneous $L^\infty$ regularization can coexist with exponential growth of H\"older seminorms and unbounded growth of every supercritical Sobolev norm along a smooth radial solution. Several regularity questions remain. Are there matching upper bounds for H\"older or Sobolev growth, and is the exponent produced by the example optimal? More generally, a precise description of how MMD convergence can coexist with deterioration of strong spatial norms would clarify the distinction between relaxation of the discrepancy and regularization of the evolving density.

\bibliographystyle{alpha}
\bibliography{coulomb_refs}
\end{document}